\documentclass[11pt]{amsart}
\usepackage[T1]{fontenc}
\usepackage{txfonts}
\usepackage{amssymb,verbatim}
\usepackage[all]{xy}
\usepackage{subfig}
\usepackage{tikz}
\usepackage{color}
\usepackage{a4wide}
\usepackage{mathrsfs}
\usepackage{hyperref}
\usepackage{wrapfig}
\usetikzlibrary{arrows}
\DeclareMathAlphabet{\mathpzc}{OT1}{pzc}{m}{it}

\newcommand{\R}{\mathbb{R}}
\newcommand{\C}{\mathbb{C}}

\newcommand\Z{\mathbb{Z}}
\newcommand{\N}{\mathbb{N}}
\newcommand{\Q}{\mathbb{Q}}
\renewcommand{\S}{\mathbb{S}}
\newcommand{\Cb}{\mathbf{C}}
\newcommand{\Bb}{\mathbf{B}}

\newcommand{\Hs}{\mathscr{H}}
\newcommand{\Hb}{\mathbf{H}}

\newcommand{\Lb}{\mathbf{L}}

\newcommand{\T}{\mathrm{T}}
\newcommand{\Tb}{\mathbf{T}}
\newcommand{\Ts}{\mathscr{T}}

\newcommand{\Lev}{\mathrm{L}}
\newcommand{\jj}{\mathbf{j}}

\newcommand{\Ar}{\mathrm{A}}

\newcommand{\Ccal}{\mathcal{C}}

\newcommand{\Lcal}{\mathcal{L}}
\newcommand{\Mcal}{\mathcal{M}}
\newcommand{\Ocal}{\mathcal{O}}
\newcommand{\Pcal}{\mathcal{P}}
\newcommand{\Qcal}{\mathcal{Q}}
\newcommand{\Scal}{\mathcal{S}}

\newcommand{\Ucal}{\mathcal{U}}
\newcommand{\Vcal}{\mathcal{V}}

\newcommand{\Dis}{\mathrm{D}}

\newcommand{\Id}{\mathrm{Id}}

\newcommand{\CP}{\mathbb{P}_{\mathbb{C}}}

\newcommand{\Mod}{\mathcal{M}}
\newcommand{\UC}{\mathscr{C}}

\newcommand{\Aa}{\textrm{Area}}

\newcommand{\ee}{\mathbf{e}}

\newcommand{\gb}{\mathbf{g}}
\newcommand\Ric{{\rm Ric}}
\newcommand\la{\lambda}
\newcommand{\vv}{\mathbf{v}}

\newcommand{\vide}{\varnothing}
\newcommand{\Sig}{\Sigma}
\newcommand{\sig}{\sigma}

\newcommand{\eps}{\epsilon}

\newcommand{\ol}{\overline}

\newcommand{\lra}{\longrightarrow}
\newcommand{\ra}{\rightarrow}
\newcommand{\smin}{\setminus}

\newcommand{\DS}{\displaystyle}

\newcommand\cal{\mathcal}

\newtheorem{Theorem}{Theorem}[section]
\newtheorem{Corollary}[Theorem]{Corollary}
\newtheorem{Lemma}[Theorem]{Lemma}
\newtheorem{Proposition}[Theorem]{Proposition}

\newtheorem{Definition}[Theorem]{Definition}

\theoremstyle{remark}
\newtheorem{Remark}[Theorem]{Remark}
\begin{document}
\title[Singular K\"ahler-Einstein metrics on $\ol{\Mod}_{0,n}$ ]{Complex hyperbolic volume and intersection of boundary divisors in moduli spaces of genus zero curves}

\author{Vincent Koziarz}
\address{Univ. Bordeaux, IMB, CNRS, UMR 5251, F-33400 Talence, France}
\email[V.~Koziarz]{vincent.koziarz@math.u-bordeaux.fr}
\author{Duc-Manh Nguyen}
\email[D.-M.~Nguyen]{duc-manh.nguyen@math.u-bordeaux.fr}

\date{\today}
\begin{abstract}
We show that the complex hyperbolic metrics defined by Deligne-Mostow and Thurston on $\Mod_{0,n}$ are singular K\"ahler-Einstein metrics when $\Mod_{0,n}$ is embedded in the Deligne-Mumford-Knudsen compactification  $\ol{\Mod}_{0,n}$. As a consequence, we obtain a  formula computing the volumes of $\Mod_{0,n}$ with respect to these metrics using intersection of boundary divisors of $\ol{\Mod}_{0,n}$. In the case of rational weights, following an idea of Y.~Kawamata, we show that these metrics  actually  represent the  first Chern class of some line bundles on $\ol{\Mod}_{0,n}$, from which other formulas computing the same volumes are derived.
\end{abstract}

\maketitle

\section{Introduction}
Let $n\geq 3$ and $\Mcal_{0,n}$ be the moduli space of Riemann surfaces of genus $0$ with $n$ marked points. Let $\mu=(\mu_1,\dots,\mu_n)$ be real weights satisfying $0<\mu_s<1$ and $\sum\mu_s=2$. Following ideas of E.~Picard, P.~Deligne and G.~D.~Mostow~\cite{DeligneMostow86} constructed --- for certain rational values of the $\mu_s$'s satisfying some integrality conditions --- complex hyperbolic lattices which enable in particular to endow $\Mcal_{0,n}$ with a complex hyperbolic metric $\Omega_\mu$. The volume of the corresponding orbifolds has been computed by several authors in some special cases when $n=5$ (see e.g. \cite{Yo,Sauter,Par09, KaMo15}).

A few years later, W.~P. Thurston noticed~\cite{Thurston98}  that for any $n$-uple of real weights satisfying the two simple conditions above, one can construct naturally a  metric completion of $(\Mcal_{0,n},\Omega_\mu)$, which can be endowed with a {\em cone-manifold structure}. He observed in particular that $(\Mcal_{0,n},\Omega_\mu)$ always has finite volume (see Section~\ref{sec:proof:main} for our normalization of the metric and the volume element;  we will use equally the notation $\Omega_\mu$ for the metric and its associated K\"ahler form).
In a more recent paper~\cite{McMullen}, C.~T.~McMullen computed the volume of $(\Mcal_{0,n},\Omega_\mu)$ using a Gauss-Bonnet theorem for cone manifolds.
Our first purpose in this paper is to compute the same volume by other methods, using ideas coming from complex (algebraic) geometry  with an approach in the spirit of Chapter 17 of \cite{DeligneMostowbook}. Along the way, we will show in particular that $\Omega_\mu$ is actually a {\em singular K\"ahler-Einstein metric} on the Deligne-Mumford-Knudsen compactification $\ol{\Mod}_{0,n}$ of $\Mod_{0,n}$.

In order to state our main results, we need a few basic facts about $\ol{\Mod}_{0,n}$~(see {\em e.g} \cite{DeligneMumford, Knudsen, Keel92, ACG11}). The moduli space $\Mod_{0,n}$ has complex dimension $N:=n-3$ and its complement in the smooth variety $\ol{\Mod}_{0,n}$ is the union of finitely many divisors called {\em boundary divisors}, or {\em vital divisors}, each of which uniquely corresponds to a partition of $\{1,\dots,n\}$ into two subsets $I_0\sqcup I_1$ such that $\min\{|I_0|,|I_1|\} \geq 2$, see~\cite{Keel92} for instance. We will denote by $\Pcal$ the set of partitions satisfying this condition. For each partition $\Scal:=\{I_0,I_1\} \in \Pcal$, we denote by $D_\Scal$ the corresponding divisor in  $\ol{\Mod}_{0,n}$. Exchanging $I_0$ and $I_1$ if necessary, we will always assume that $\mu_\Scal:=\sum_{s\in I_1}\mu_s \leq 1$ (in order to lighten the notation, we do not write explicitly the dependence of the coefficients $\mu_\Scal$ on
$\mu$). 

For any $s\in\{1,\dots,n\}$, we also define the divisor class $\psi_s$ on $\ol{\Mod}_{0,n}$ associated to the pullback of the relative cotangent bundle of the universal curve by the section corresponding to the $s$-th marked point.

Finally, if $D$ is a divisor on $\ol{\Mod}_{0,n}$, $D^N$ means as usual that we take the $N$-th self-intersection of $D$.
Our main result is the
\begin{Theorem}\label{theorem:main}
Let $n\geq 4$ and $\Mcal_{0,n}$ be the moduli space of Riemann surfaces of genus $0$ with $n$ marked points. Let $\mu=(\mu_1,\dots,\mu_n)$ be real weights satisfying $0<\mu_s<1$ and $\sum\mu_s=2$. Let $D_\mu:=\sum_{\Scal\in\Pcal} \lambda_\Scal \,D_\Scal$ where
$$\lambda_\Scal=(|I_1|-1)(\mu_\Scal-1)+1.$$
Then the volume of $(\Mcal_{0,n},\Omega_\mu)$ satisfies
\begin{eqnarray}\label{eq:main:thm}
\int_{\Mod_{0,n}}\Omega_\mu^{N}=\frac{1}{(N+1)^N}\Bigl(K_{\ol{\Mod}_{0,n}}+D_\mu\Bigr)^N&=&\frac{1}{(N+1)^N}\left(\sum_\Scal \Bigl(|I_1|-1\Bigr) \biggl(\mu_\Scal-\frac{|I_1|}{N+2}\biggr) \,D_\Scal \right)^N\\
&=&\frac1{2^N}\left(-\sum_{s=1}^n\mu_s\,\psi_s+\sum_\Scal\mu_\Scal\, D_\Scal\right)^N \nonumber
\end{eqnarray}
where $\Omega_\mu$ denotes the K\"ahler form associated with the metric and $K_{\ol{\Mod}_{0,n}}$ is the canonical divisor of $\ol{\Mod}_{0,n}$.
\end{Theorem}
In Corollary~\ref{coro:mcmullen}, we compare formula~(\ref{eq:main:thm}) with the one obtained by McMullen in~\cite{McMullen}.
There exists an algorithm to calculate the intersection numbers of divisors of the type $D_\Scal$ (see~\cite{Kontsevich-Manin} and Appendix~\ref{appendix:intersection} below), but doing the calculation by hand is rather involved. However, those computations can be carried out efficiently by a computer program by C.~Faber.

Besides providing an alternative method to compute the volume of $(\Mod_{0,n},\Omega_\mu)$, our approach also sheds light on the relation between Thurston's compactification $\ol{\Mod}_{0,n}^\mu$ of $\Mod_{0,n}$ and  $\ol{\Mod}_{0,n}$. Recall that Thurston identified $\Mod_{0,n}$ with the space of flat surfaces homeomorphic to the sphere $\S^2$ having $n$ conical singularities with cone angles given by $2\pi(1-\mu_s)$ up to rescaling. A stratum of $\ol{\Mod}^\mu_{0,n}$ consists of flat surfaces which are the limits when some clusters of singularities collapse into points. On the other hand, each stratum of $\ol{\Mod}_{0,n}$  is encoded by a tree whose vertices are labelled by the subsets in a partition of $\{1,\dots,n\}$. Every point in such a stratum represents a stable curve with several irreducible components. Among those components, there is a particular one that we call {\em $\mu$-principal component} whose definition depends on $\mu$ (see Section~\ref{sec:princ:component}).
To each stratum $S$ of $\ol{\Mod}^\mu_{0,n}$, we have a corresponding stratum $\tilde{S}$ of $\ol{\Mod}_{0,n}$ such that, for  any flat surface represented by a point in $S$, the underlying Riemann surface with punctures is isomorphic to the $\mu$-principal component of the stable curves represented by some points in $\tilde{S}$. So in some sense, one can say that $\ol{\Mod}^\mu_{0,n}$ is obtained from $\ol{\Mod}_{0,n}$ by ``contracting'' every boundary stratum to its $\mu$-principal factor.

In the literature, one can find compactifications of $\Mcal_{0,n}$ which are different from the Deligne-Mumford-Knudsen one $\ol\Mcal_{0,n}$ (see in particular the papers of B. Hasset~\cite{hassett} and D.~I.~Smyth~\cite{Smyth}). These compactifications are contractions of $\ol\Mcal_{0,n}$ and are in general singular. Actually, when compact, Thurston completions corresponding to weights $\mu$ as above are compactifications considered by Smyth, but for our purpose it is more convenient to work on the smooth model $\ol\Mcal_{0,n}$ and we will not insist on this point of view (see also Remark~\ref{rk:Th:cone:angle} below).

By construction, $\Omega_\mu$ is the curvature of a Hermitian metric on a holomorphic line bundle over $\Mod_{0,n}$. When all the weights in $\mu$ are rational, Y.~Kawamata~\cite{kaw} observed that this line bundle admits a natural extension to $\ol{\Mod}_{0,n}$. It turns out that  $\Omega_\mu$  can be considered as a representative in the sense of currents of the first Chern class of this extended line bundle. It can be shown that the latter is effective. We develop this algebro-geometric approach in Section~\ref{sec:mu:rat:ext:sect}. By constructing explicit sections and determining their zero divisor,  we provide other formulas for the volume which avoid metric considerations. Even though at first glance this approach seems to work only in the case of rational weights, by  continuity argument, our formulas are actually valid for all values of $\mu$ satisfying the hypothesis of Theorem~\ref{theorem:main}. Namely, we get the following
\begin{Theorem}\label{thm:kawa:formpart}
For each $1\leq s<s'\leq n$,  define
$$
\lambda(s,s')=\left\{
\begin{array}{l}
0\ \ {\rm if}\ \sum_{k=s}^{s'}\mu_k\leq1\ {\rm or}\ \sum_{k=s+1}^{s'-1}\mu_k\geq1,\\
\min\bigl\{\mu_s,\mu_{s'},\sum_{k=s}^{s'}\mu_k-1,1-\sum_{k=s+1}^{s'-1}\mu_k\bigr\}\ \ {\rm otherwise}
\end{array}
\right.$$
and
$$\delta_\Scal(s,s')=\left\{
\begin{array}{l}
1\ {\rm if}\ \{s,s'\}\subset I_1\\
0\ {\rm otherwise}
\end{array}
\right.
.
$$
Then the effective $\R$-divisor
$$
D_\sigma:=\sum_\Scal\sum_{1\leq s<s'\leq n} \delta_\Scal(s,s')\lambda(s,s')\, D_\Scal
$$
satisfies
\begin{equation}\label{eq:kawa:form}
\int_{\Mod_{0,n}}\Omega_\mu^{N}= \frac1{(N+1)^N}(K_{\ol{\Mod}_{0,n}}+D_\mu)^{N}= D_\sigma^N.
\end{equation}
\end{Theorem}

In this paper, many objects and quantities depend on the weights $\mu$. However, as we already said for the coefficients $\mu_\Scal$, this dependence will not always appear explicitly but the reader will have to keep it in mind.

\begin{Remark}\label{rk:sum:neq:1}
Whenever there exists a partition $\{I_0,I_1\} \in \Pcal$ such that $\sum_{s\in I_0}\mu_s=\sum_{s\in I_1}\mu_s=1$, the metric completion of Thurston is not compact and our method does not provide directly a formula for the volume of $(\Mcal_{0,n},\Omega_\mu)$. However,  the formulas in Theorem~\ref{theorem:main} remain valid by continuity arguments (as in~\cite{McMullen}). For these reasons, we will assume through out this paper  that the sum of the weights for indices in any subset of $\{1,\dots,n\}$ is always different from $1$.
\end{Remark}

\subsection*{Outline} The paper is organized as follows.
\begin{enumerate}
\item In Section~\ref{sec:loc:sys:on:P1} we collect the necessary background from the paper of Deligne and Mostow~\cite{DeligneMostow86}. Associated to any weight vector $\mu=(\mu_1,\dots,\mu_n) \in \R^n_{>0}$ such that $\mu_1+\dots+\mu_n=2$, we have a rank one local system $\Lb$ on the punctured sphere $\CP^1\smin\{x_1,\dots,x_n\}$ with  monodromy $\exp(2\imath\pi\mu_s)$ at $x_s$, which is equipped with a Hermitian metric. Assuming $\mu_s \not\in \Z$ for some  $s \in \{1,\dots,n\}$,  we have $\dim_\C H^1(\CP^1\smin\{x_1,\dots,x_s\},\Lb) = n-2$.
Up to a multiplicative constant, there exists a unique section $\omega$ of the bundle $\Omega^1(\Lb)$ which is holomorphic on $\CP^1\smin\{x_1,\dots,x_n\}$, and has valuation $-\mu_s$ at $x_s$. This section defines a non-zero cohomology class in $H^1(\CP^1\smin\{x_1,\dots,x_n\},\Lb)$.

One can obviously move the points $x_1,\dots,x_n$ around,  therefore $H^1(\CP^1\smin\{x_1,\dots,x_n\},\Lb)$ and $\omega$ give rise to a local system $\Hb$ of rank $n-2$ and a holomorphic line bundle $\Lcal$ on $\Mod_{0,n}$,  the fiber of $\Lcal$ over the point $m\simeq (\CP^1,\{x_1,\dots,x_n\})\in \Mod_{0,n}$ is the line generated by $\omega$ in $H^1(\CP^1\smin\{x_1,\dots,x_n\},\Lb)$.  Projectivizing $\Hb$, we get a flat $\CP^{n-3}$-bundle over $\Mod_{0,n}$, and $\Lcal$ provides us with a multivalued section $\Xi_\mu$ of this bundle.  The pullback  $\widetilde{\Xi}_\mu$ of $\Xi_\mu$ to $\widetilde{\Mod}_{0,n}$ is an \'etale mapping from $\widetilde{\Mod}_{0,n}$ to $\CP^{n-3}$.

 The Hermitian form of $\Lb$ gives rise to a Hermitian form $((.,.))$ on $H^1(\CP^1\smin\{x_1,\dots,x_n\},\Lb)$. In the case $0<\mu_s<1$ for all $s$, this Hermitian form has signature $(1,n-3)$ and $((\omega,\omega)) >0$. It follows that the section $\widetilde{\Xi}_\mu$ takes values in the ball $\Bb:=\{\langle v \rangle \in \CP^{n-3}, \ ((v,v)) >0\}\subset \CP^{n-3}$ (here we identify $H^1(\CP^1\smin\{x_1,\dots,x_n\},\Lb)$ with $\C^{n-2}$). The pullback of the canonical metric on $\Bb$ by $\widetilde{\Xi}_\mu$ provides us with a complex hyperbolic metric on $\Mod_{0,n}$, which will be denoted by $\Omega_\mu$. By definition,  $\Omega_\mu$ is also the Chern form of the Hermitian line  bundle  $(\Lcal,((.,.)))$. \\

 \item  Our goal  is to show that $\Omega_\mu$ is a singular K\"ahler-Einstein metric on $\ol{\Mod}_{0,n}$. For this purpose, we first construct trivializing holomorphic sections of $\Lcal$ in the neighborhood of every point $m\in \partial\ol{\Mod}_{0,n}$.
 In Section~\ref{sec:coord:bdry}, we recall the construction of local coordinates of $\ol\Mcal_{0,n}$ near $m$  by  plumbing families.
 In Section~\ref{sec:sect:L:near:bdry}, we consider the case where $m$ is contained in a stratum of codimension one in $\ol{\Mod}_{0,n}$, which means that $m$ represents a stable curve having two genus zero components, denoted by $C^0$ and $C^1$, joined at a node. In each component, we assign a positive weight to the point corresponding to the node of $m$ such that the weights associated to all the marked points add up to $2$. We have on $C^i$  a rank one local system $\Lb_i$  and  a  section $\omega_i$ of $\Omega^1(\Lb_i)$ in the same way as we had $\Lb$ and $\omega$ above.
 The sections $\omega_0$ and $\omega_1$ will be used as data for the construction of a plumbing family representing a neighborhood $\Ucal$ of $m$ in $\ol{\Mod}_{0,n}$.
 As a by-product, we get a holomorphic non-vanishing section $\Phi$ of $\Lcal$ in $\Ucal\cap\Mod_{0,n}$.
 In Section~\ref{sec:sect:L:bdry:gen}, we generalize this construction to the case where $m$ is contained in a stratum of codimension $r$ with $r>1$.\\

\item Section~\ref{sec:metric:coord} is devoted to the proof of a formula for the Hermitian norm of the section $\Phi$ (see Proposition~\ref{prop:norm:bdry:gen}). The idea of the proof is to use the flat metric approach of Thurston. We start by relating the  point of views of Deligne-Mostow and Thurston. Each holomorphic section of $\Omega^1(\Lb)$ on $\CP^1\smin\{x_1,\dots,x_n\}$ with valuation $-\mu_s$ at $x_s$ defines a flat metric on $\CP^1$ with cone singularities at $x_1,\dots,x_n$. Its Hermitian norm  with respect to $((.,.))$ is precisely the area of this flat surface.
In~\cite{Thurston98}, Thurston introduced a method to compute this area by performing some  surgeries on the flat surface, and obtained in particular an alternative proof that the signature of $((.,.))$ is $(1,n-3)$. We will use the same method to compute the Hermitian norm of  $\omega'=\Phi(m')$, where $\Phi$ is the section of $\Lcal$ mentioned above and  $m' \in \Ucal\cap\Mod_{0,n}$.
As a direct consequence, we obtain a rather explicit formula for the metric $\Omega_\mu$ near the boundary of $\ol{\Mod}_{0,n}$ (see Proposition~\ref{prop:Chern:form}).\\

 \item In Section~\ref{section:SKE}, we recall some basic facts about singular K\"ahler-Einstein metrics. It follows immediately from Proposition~\ref{prop:Chern:form} that $\Omega_\mu$  is a singular K\"ahler-Einstein metric associated with the pair $(\ol{\Mod}_{0,n},D_\mu)$. Theorem~\ref{theorem:main} is then a straightforward consequence of this fact. Comparing $\Omega_\mu$ with the complex hyperbolic metric considered by McMullen in \cite{McMullen}, we get Corollary~\ref{coro:mcmullen}.\\

\item In Section~\ref{sec:mu:rat:ext:sect}, following an idea of Kawamata~\cite{kaw}, we construct an extension  $\hat\Lcal$ of $\Lcal$ to $\ol{\Mod}_{0,n}$ in the case when all weights $\mu_s$ are rational. This extension is the pushforward of a rank one locally free sheaf on the universal curve $\ol{\UC}_{0,n}$. By construction, $\Phi$ extends naturally to a trivializing holomorphic section of $\hat{\Lcal}$ on $\Ucal$,  and $\Omega_\mu$ is a representative (in the sense of currents) of the first Chern class of $\hat{\Lcal}$.
This leads to an alternative method to compute the volume of $\Mod_{0,n}$ with respect to $\Omega_\mu$ by using sections of $\hat{\Lcal}$ (see Theorem~\ref{thm:kawa:form}). Simplifying  a construction by Kawamata,  we construct some explicit holomorphic global sections of $\hat{\Lcal}$, for which one can easily determine the zero divisor. By the continuity of the volume with respect to $\mu$ (which can be derived from Theorem~\ref{theorem:main}), we obtain Theorem~\ref{thm:kawa:formpart}. This approach also allows us to calculate $c_1(\hat\Lcal)$ by the Grothendieck-Riemann-Roch formula and to recover formula~\eqref{eq:main:thm}.\\

\item In the appendix we explain an algorithm computing the intersection numbers of boundary divisors, which is necessary if one wants to compute the volumes explicitly.  We then  give the explicit results for $\Mod_{0,5}$ and a special case for $\Mod_{0,6}$ with the aim to help interested readers to see how concrete computations can be carried out.
\end{enumerate}

\subsection*{Acknowledgments:} We thank S.~Boucksom, J.-P.~Demailly and P.~Eyssidieux for very useful conversations about positive currents and singular K\"ahler-Einstein metrics.
We are very indebted to D.~Zvonkine for the helpful and enlightening discussions.

We would also like to thank C.~Faber who shared with us his program computing intersection numbers in $\ol{\Mod}_{g,n}$, and L. Pirio for useful comments on an earlier version of this paper.

\section{Background on rank one local systems on the punctured sphere.}\label{sec:loc:sys:on:P1}
In this section, we summarize the settings and some results in \cite[Sec. 2,3]{DeligneMostow86} relevant to our purpose.
\subsection{Cohomology of a rank one local system on the punctured sphere}
Let $n$ be a positive integer such that $n\geq 3$.  Let us fix the following data
\begin{itemize}
 \item[.] $\Sig=(x_1,x_2,\dots,x_n)$ is a $n$-uple of distinct points on the sphere $\S^2\simeq \CP^1$.
 \item[.] $\mu=(\mu_1,\dots,\mu_n)$ is a $n$-uple of positive real numbers such that
 $$
 \mu_1+\dots+\mu_n=2.
 $$
 \item[.] $\alpha_i=\exp(2\pi\imath\mu_i)$, $i=1,\dots,n$.
 \item[.] $\Lb$ is a complex rank one local system on $\CP^1\smin \Sig$ with monodromy around $x_i$ given by $\alpha_i$. Note that up to isomorphism $\Lb$ is unique.
\end{itemize}

Using the $C^\infty$-de Rham description, we can identify  $H^\bullet(\CP^1\smin\Sig, \Lb)$  with the cohomology of the de Rham complex of  $\Lb$-valued $C^\infty$ differential forms on $\CP^1\smin\Sig$, and $H^\bullet_c(\CP^1\smin\Sig, \Lb)$ with the cohomology of the subcomplex of compactly supported forms.

Let $\Lb^\vee$ be the dual local system of $\Lb$. This is the local system with monodromy $\alpha_i^{-1}$ around $x_i$. The Poincar\'e duality pairing by integration on $\CP^1\smin\Sig$, that is
$$
\begin{array}{ccc}
 H^i(\CP^1\smin\Sig,\Lb)\otimes H_c^{2-i}(\CP^1\smin\Sig, \Lb^\vee) & \lra & \C \\
 (\alpha,\beta) & \mapsto & \int_{\CP^1\smin\Sig}\alpha\wedge\beta\\
\end{array}
$$
is then a perfect pairing.

\begin{Proposition}[Deligne-Mostow]~\label{prop:dim:cohom}
 If one of the $\alpha_s, \, s\in \{1,\dots,n\}$ is not $1$, then $H^i(\CP^1\smin\Sig,\Lb)$ and $H^i_c(\CP^1\smin\Sig,\Lb)$ vanish for $i\neq 1$, and
 $$
 \dim H^1(\CP^1\smin\Sig,\Lb)=\dim H^1_c(\CP^1\smin\Sig,\Lb) = n-2.
 $$
\end{Proposition}

There are several ways to describe the homology and cohomology of $\Lb$ and $\Lb^\vee$. For instance, one can use a triangulation $\Ts$ of $\CP^1\smin\Sig$ to construct chain complexes giving $H^\bullet(\CP^1\smin\Sig,\Lb)$ and $H_\bullet(\CP^1\smin\Sig,\Lb)$ as follows: an  $i$-chain with coefficients in $\Lb$ is a formal sum $\sum e_\sig\cdot\sig$, where $\sig$ is an $i$-simplex of the triangulation, and $e_\sig$ is a horizontal section of the restriction of $\Lb$ to $\sig$. An $\Lb$-valued $i$-cochain associates to each $i$-simplex $\sig$ of the triangulation a horizontal section of $\Lb$ over $\sig$. Note that the complex of $\Lb$-valued cochains is dual to the complex of chains with coefficients in $\Lb^\vee$.

The cohomology with compact support $H_c^\bullet(\CP^1\smin\Sig,\Lb)$ is also the cohomology of the complex of $\Lb$-valued cochains compactly supported on $\Ts$. Its dual complex is the complex of locally finite chains with coefficients in  $\Lb^\vee$, the  homology of which will be denoted by $H^{\rm lf}_\bullet(\CP^1\smin\Sig,\Lb^\vee)$.

One can also use currents to define $H^\bullet(\CP^1\smin\Sig,\Lb)$. For any chain $C$  with  coefficients in $\Lb^\vee$, there exists a unique $\Lb^\vee$-valued current $(C)$ such that
$$
\int_C\omega =\int_{\CP^1\smin\Sig}(C)\wedge\omega
$$

\noindent for all  $\Lb$-valued $C^\infty$ form $\omega$. The map $C \mapsto (C)$ provides the isomorphisms $H_i(\CP^1\smin\Sig, \Lb^\vee) \simeq H_c^{2-i}(\CP^1\smin\Sig,\Lb^\vee)$ and $H^{\rm lf}_i(\CP^1\smin\Sig,\Lb^\vee)\simeq H^{2-i}(\CP^1\smin\Sig,\Lb^\vee)$.

If $\beta$ is a rectifiable proper map from an open, semi-open, or closed  interval $I$ to $\CP^1\smin\Sig$, and $e \in H^0(I,\beta^*\Lb^\vee)$, we let $(e\cdot\beta)$ be the $\Lb^\vee$-valued current for which
$$
\int (e\cdot\beta)\wedge \omega=\int_I\langle e,\beta^*\omega \rangle.
$$

\noindent If $\beta: [0,1] \ra \CP^1$ maps $0$ and $1$ to $\Sig$ and $(0,1)$ into $\CP^1\smin\Sig$, then for any $e\in H^0((0,1),\beta^*\Lb^\vee)$, $e\cdot\beta$ is a cycle and hence defines an homology class in $H_1^{\rm lf}(\CP^1\smin\Sig,\Lb^\vee)\simeq H^1(\CP^1\smin\Sig,\Lb^\vee)$.

Let us fix a partition of $\Sig$ into two subsets $\Sig_0$ and $\Sig_1$. Let $\T_0, \T_1$ be two trees (graphs with no cycles) where the number of vertices of $\T_i$  is $|\Sig_i|$, and $\beta: \T_0\sqcup\T_1 \ra \CP^1$ be an embedding such that the vertex set of $\T_i$ is mapped to $\Sig_i$. We choose for any open edge $a$ of $\T_0\sqcup\T_1$ an orientation, and a non vanishing section $e(a) \in H^0(a,\beta^*\Lb^\vee)$. For each edge $a$, $e(a)\cdot\beta_{|a}$ is then a locally finite cycle on $\CP^1\smin\Sig$, with coefficients in $\Lb^\vee$. Let $I_0\sqcup I_1$ be the partition of $\{1,\dots,n\}$ corresponding to the partition $\Sig=\Sig_0\sqcup\Sig_1$.

\begin{Proposition}[\cite{DeligneMostow86}, Prop. 2.5.1]\label{prop:base:hom}
 If $\prod_{i\in I_0} \alpha_i \neq 0$, then the family
 $$
 \{e(a)\cdot\beta_{|a}, a \text{ is an edge of } \T_0\sqcup\T_1\}
 $$

 \noindent is a basis of $H_1^{\rm lf}(\CP^1\smin\Sig,\Lb^\vee)$.
\end{Proposition}

\subsection{Sheaf cohomology}
Another way to compute the cohomology of $\CP^1\smin\Sig$ with coefficients in  $\Lb$ is to use the sheaf cohomology. For this purpose, we will identify $\Lb$ with its sheaf of locally constant sections. Let $j : \CP^1\smin\Sig \ra \CP^1$ be the natural inclusion, and let $j_{!}\Lb$ be the extension of $\Lb$ by $0$ to $\CP^1$.  In this setting, $H_c^\bullet(\CP^1\smin\Sig,\Lb)$ is the cohomology on $\CP^1$ with coefficients in $j_{!}\Lb$. It is by definition, the hypercohomology on $\CP^1$ of any complex of sheaves $K^\bullet$ with $\Hs^0(K^\bullet)=j_{!}\Lb$, and $\Hs^i(K^\bullet)=0$, for $i\neq 0$.  On the other hand, if $\Lb^\bullet$ is a resolution of $\Lb$, whose components are acyclic for $j_*$ (that is $R^qj_*\Lb^k=0$ for $q >0$), then $H^\bullet(\CP^1\smin\Sig,\Lb)$ is the hypercohomology on $\CP^1$ of $j_*\Lb$. We have the 

\begin{Proposition}[\cite{DeligneMostow86}, Prop. 2.6.1]\label{prop:isom:cohom}
 If $\alpha_i \neq 1$ for all $i \in \{1,\dots,n\}$, then $H^\bullet_c(\CP^1\smin\Sig,\Lb) \simeq H^\bullet(\CP^1\smin\Sig,\Lb)$.
\end{Proposition}

The holomorphic $\Lb$-valued de Rham complex $\Omega^\bullet(\Lb): \Ocal(\Lb) \ra \Omega^1(\Lb)$ is a resolution of $\Lb$ on $\CP^1\smin\Sig$. Hence, we can interpret $H^\bullet(\CP^1\smin\Sig,\Lb)$  as the hypercohomology on $\CP^1\smin\Sig$ of $\Omega^\bullet(\Lb)$. Since $H^q(\CP^1\smin\Sig,\Omega^p(\Lb))=0$, for $q>0$ (because $\CP^1\smin\Sig$ is  Stein), this gives
$$
H^\bullet(\CP^1\smin\Sig,\Lb)=H^\bullet \Gamma(\CP^1\smin\Sig,\Omega^\bullet(\Lb)).
$$
On the other hand, since we have $R^qj_*\Omega^p(\Lb)=0$ for $q >0$, it follows that

$$
H^\bullet(\CP^1\smin\Sig,\Lb)=\mathbb{H}^\bullet(\CP^1,j_*\Omega^\bullet(\Lb)).
$$
\subsection{The de Rham meromorphic description of the cohomology of $\Lb$}
We will describe a section of $\Ocal(\Lb)$ on an open set $U\subset \CP^1\smin \Sig$ as the product of a multivalued function and a multivalued section of $\Lb$. Those objects are defined as follows: $U$ is provided with a base point $o$,  a multivalued section of a sheaf $\mathscr{F}$ on $U$ is actually a section of the pullback of  $\mathscr{F}$ on the universal cover $(\hat{U},\hat{o})$ of $(U,o)$.
A section of $\Lb$ at $o$ extends to a unique horizontal  multivalued section. A multivalued section of $\Ocal$ is uniquely determined by its germ at $o$.

Fix an $x_s\in \Sig$, and let $z$ be a local coordinate which identifies a neighborhood of $x_s$ with a disc $\Dis$ in $\C$ centered at $z(x_s)=0$. Let $\Dis^*=\Dis\smin\{0\}$.  If the monodromy of $\Lb$ around $x_s$ is $\alpha_s=\exp(2\pi\imath\mu_s)$, then the monodromy of $z^{-\mu_s}$ is the inverse of that of a horizontal section of $\Lb$. Therefore, any section of $\Ocal(\Lb)$ (resp. $\Omega^1(\Lb)$) on $\Dis^*$ can be written as $u=z^{-\mu_s}\cdot e\cdot f$ (resp. $u=z^{-\mu_s}\cdot e\cdot fdz$), where $e$ is a non-zero (horizontal) multivalued section of $\Lb$, and $f$ is a holomorphic function on $\Dis^*$.  We define $u$ to be {\em meromorphic} at $x_s$ if $f$ is, and define its valuation at $x_s$ to be
$$
v_{x_s}(u)=v_{x_s}(f)-\mu_s.
$$
\noindent Note that these definitions are independent of the choice of the local coordinate.

Let us write $j_*^m\Omega^\bullet(\Lb)$ for the sheaf complex consisting of meromorphic forms in $\Omega^\bullet(\Lb)$. The inclusion of $j_*^m\Omega^\bullet(\Lb)$ into $j_*\Omega^\bullet(\Lb)$ induces an isomorphism on the cohomology sheaves. This implies
$$
\mathbb{H}^\bullet(\CP^1,j_*^m\Omega^\bullet(\Lb)) \simeq \mathbb{H}^\bullet(\CP^1, j_*\Omega^\bullet(\Lb))=H^\bullet(\CP^1\smin\Sig,\Lb).
$$

\noindent Since we have $H^q(\CP^1, j_*^m\Omega^p(\Lb))=0$ for $q>0$, $\mathbb{H}^\bullet(\CP^1,j_*^m\Omega^\bullet(\Lb))$ is simply $H^\bullet\Gamma(\CP^1,j_*^m\Omega^\bullet(\Lb))$, that is the cohomology of the complex of $\Lb$-valued forms holomorphic on $\CP^1\smin\Sig$ and meromorphic at $\Sig$. To sum up, we have
$$
H^\bullet(\CP^1\smin\Sig, \Lb)\simeq H^\bullet\Gamma(\CP^1,j_*^m\Omega^\bullet(\Lb)).
$$

\begin{Proposition}[\cite{DeligneMostow86}, Cor. 2.12]\label{prop:mer:form:uniq}
There is, up to a constant factor, a unique non-zero $\omega \in \Gamma(\CP^1,j_*^m\Omega^1(\Lb))$ whose valuation at $x_s$ is at least $-\mu_s$. Actually, we have $v_{x_s}(\omega)=-\mu_s$, and $\omega$ is invertible on $\CP^1\smin\Sig$. If $\infty \not\in\Sig$, then, up to a constant factor, $\omega=e\cdot \prod_{x_s\in \Sig} (z-x_s)^{-\mu_s}dz$, and if  $\infty \in \Sig$, then $\omega=e\cdot \prod_{x_s\neq \infty}(z-x_s)^{-\mu_s}dz$.
\end{Proposition}

Moreover, we have

\begin{Proposition}[\cite{DeligneMostow86}, Prop. 2.13]\label{prop:form:n:trivial}
Assume that $\alpha_s\neq 1$ for all $s\in \{1,\dots,n\}$, that is none of the $\mu_s$ is an integer, then the cohomology class of the form $\omega$ in the previous proposition is not zero.
\end{Proposition}

Let us assume that none of the $\alpha_s$ is $1$. Let $[\omega]$ denote the cohomology class of $\omega$ in $H^1(\CP^1\smin\Sig,\Lb)$. Since we have $H^1(\CP^1\smin\Sig, \Lb)\simeq H^1_c(\CP^1\smin\Sig,\Lb)$ (cf. Proposition~\ref{prop:isom:cohom}), $\omega$ also gives a cohomology class,  denoted again by $[\omega]$, in $H^1_c(\CP^1\smin\Sig,\Lb)$. Thus, for any locally finite cycle $C$ in $H_1^{\rm lf}(\CP^1\smin\Sig, \Lb^\vee)$, $\langle [C],[\omega]\rangle$  is well-defined. If $C$ is represented by a compactly supported cycle, then
$$
\langle [C],[\omega] \rangle =\int_C\omega.
$$
If $C=e'\cdot\beta$ is a cycle where $\beta: [0,1] \ra \CP^1$ is such that $\beta(0),\beta(1)\in \Sig$,  $\beta((0,1)) \subset \CP^1\smin\Sig$, and $e'$ is a horizontal section of $\beta^*\Lb^\vee$ on $(0,1)$. We can define a finite cycle $C'$ with coefficients in $\Lb^\vee$ homologous to $C$ as follows: let $x_{s_0}=\beta(0), x_{s_1}=\beta(1)$, and $\Dis_i$ a small disc centered at $x_{s_i}$ such that $\Dis_i\cap \Sig=\{x_{s_i}\}$, and $\Dis_0\cap\Dis_1=\vide$. Let $C_i, \,  i=0,1$, be a circle centered at $x_{s_i}$ and contained in $\Dis_i$. Let $I$ denote the  interval $[0,1]$.  Let $y_0=\beta(\eps_0)$ be the first intersection of $\beta(I)$ and $C_0$, and $y_1=\beta(1-\eps_1)$ be the last intersection of $\beta(I)$ with $C_1$. We consider $y_0$ and $y_1$ as base points of $C_0$ and $C_1$ respectively, and parametrize those circles counter-clockwise by the maps $\gamma_i:[0,1]\ra C_i$. Let $I':=[\eps_0,1-\eps_1]$, and $\beta'$ be the restriction of $\beta$  to $I'$.
Let $e'_i:=e'(y_i)/(\alpha_{s_i}^{-1}-1)$. We also denote  by $e'$ the unique horizontal section of $\gamma_i^*\Lb^\vee$ determined by this vector. Consider the 1-chain $e'_i\cdot\gamma_i$ with coefficients in $\Lb^\vee$. Since the monodromy of $\Lb^\vee$ at $x_{s_i}$ is $\alpha^{-1}_{s_i}$, we get $ d(e'_i\cdot\gamma_i)=e'\cdot\{y_i\}$. Let $C'$ denote  the 1-cycle $e'_0\cdot\gamma_0+e'\cdot\beta'-e'_1\cdot\gamma_1$. One can easily check that $dC'=0$, and $[C']=[C] \in H_1^{\rm lf}(\CP^1\smin\Sig,\Lb^\vee)$. Since $C'$ is compactly supported, we have
$$
\langle [C],[\omega] \rangle =\langle [C'], [\omega] \rangle =\int_{C'} \omega.
$$

\begin{Remark}\label{rk:cal:by:int}
If $\omega=e\cdot \prod_{x_s\in \Sig}(z-x_s)^{-\mu_s}dz$, where $0< \mu_s < 1$ for all $s\in\{1,\dots,n\}$, and $\beta$ is a path from $x_{s_1}$ to $x_{s_2}$ without passing through any point in $\Sig$, then we also have
\begin{equation}\label{eq:pairing:by:int}
\langle [C],[\omega] \rangle = \langle e',e \rangle \int_\beta \prod_{x_s\in \Sig}(z-x_s)^{-\mu_s}dz.
\end{equation}
\end{Remark}

\subsection{Hermitian structure}\label{sec:hermstruct} Since all the $\alpha_s$ have modulus equal to $1$, $\Lb$ admits a horizontal Hermitian metric $(.,.)$. We can use this metric to define a perfect pairing
$$
\psi_0:  H^1_c(\CP^1\smin\Sig,\Lb)\otimes_{\C}H^1_c(\CP^1\smin\Sig,\bar{\Lb}) \ra H^2_c(\CP^1\smin\Sig,\C)\simeq \C
$$
where $\bar{\Lb}$ is the complex conjugate local system of $\Lb$.
The vector space $H^1_c(\CP^1\smin\Sig,\bar{\Lb})$ is the complex conjugate of $H^1_c(\CP^1\smin\Sig,\Lb)$.   By setting
$$
((u,v)):=\frac{-1}{2\pi\imath}\psi_0(u,\bar{v})
$$
\noindent we get a Hermitian form on $H^1_c(\CP^1\smin\Sig,\Lb)$.

A section $\omega$ of $j_*^m\Omega^1(\Lb)$  is said to be  {\em of the first kind} if $v_{x_s}(\omega)>-1$ for all $x_s\in \Sig$. For such a form,
we have $|\int_{\CP^1\smin\Sig}\omega\wedge\ol{\omega}| < \infty $. We define $H^{1,0}(\CP^1\smin\Sig,\Lb)$ to be the vector space of forms of the first kind in $\Gamma(\CP^1,j_*^m\Omega^1(\Lb))$, and $H^{0,1}(\CP^1\smin\Sig,\Lb)$ as the complex conjugate of $H^{1,0}(\CP^1\smin\Sig,\bar{\Lb})$. The latter is the space of anti-holomorphic $\Lb$-valued $1$-forms, whose complex conjugate is of the first kind. As usual, such a form $\omega$ defines a cohomology class $[\omega]\in H^1(\CP^1\smin\Sig,\Lb)\simeq H_c^1(\CP^1\smin\Sig,\Lb)$.

\begin{Proposition}\label{prop:herm:coh:class}[\cite{DeligneMostow86} Prop. 2.19]
 If $\omega_1$ and $\omega_2$ are in $H^{1,0}(\CP^1\smin\Sig,\Lb)\cup H^{0,1}(\CP^1\smin\Sig,\Lb)$ then
 $$
 (([\omega_1],[\omega_2]))=\frac{-1}{2\pi\imath}\int_{\CP^1\smin\Sig}\omega_1\wedge\ol{\omega}_2.
 $$
\end{Proposition}

\begin{Proposition}\label{prop:herm:sign}[\cite{DeligneMostow86} Prop. 2.20]
 Assume that $0< \mu_s < 1$ for all $s\in \{1,\dots,n\}$, then the natural map
 $$
 H^{1,0}(\CP^1\smin\Sig,\Lb)\oplus H^{0,1}(\CP^1\smin\Sig,\Lb) \ra H^1(\CP^1\smin\Sig,\Lb)\simeq H^1_c(\CP^1\smin\Sig,\Lb)
 $$
 \noindent is an isomorphism. The Hermitian form $((.,.))$ is positive definite on $H^{1,0}$, negative definite on $H^{0,1}$, and the decomposition is orthogonal.  Since $\dim_\C H^{1,0}=1$, and $\dim_\C H^{0,1}=n-3$, the signature of $((.,.))$ is $(1,n-3)$.
\end{Proposition}

\subsection{Local system and the line bundle $\Lcal$ over $\Mod_{0,n}$}\label{sec:def:L}
Recall that $\Mod_{0,n}$ is the moduli space parametrizing Riemann surfaces of genus zero and $n$ marked points (punctures). Since every Riemann surface of genus zero is isomorphic to $\CP^1$, we can also view $\Mod_{0,n}$ as the space of configurations of $n$ distinct points on $\CP^1$ up to action of ${\rm PGL}(2,\C)$.  If $\Sig=\{x_1,\dots,x_n\}, \, n\geq 3$, is a set of $n$ points in $\CP^1$, then up to action of ${\rm PGL}(2,\C)$, we can always assume that $x_{n-2}=0, x_{n-1}=1, x_n=\infty$. Thus, $\Mod_{0,n}$ can be identified with the subset of $(\CP^1)^{n-3}$ consisting of $(n-3)$-tuples $(x_1,\dots,x_{n-3})$ such that $x_s\neq x_{s'}$ if $s\neq s'$, and $x_s\not\in \{0,1,\infty\}$.

Over $\Mod_{0,n}$ we have  a fibration $\pi : \UC_{0,n} \ra \Mod_{0,n}$ whose fiber over a point $m\in \Mod_{0,n}$ is the $n$-punctured sphere represented by $m$. Let $\ol{\Mod}_{0,n}$ be the Deligne-Mumford-Knudsen compactification of $\Mod_{0,n}$. We also have a fibration $\pi: \ol{\UC}_{0,n} \ra \ol{\Mod}_{0,n}$ extending the projection from $\UC_{0,n}$ to $\Mod_{0,n}$, where $\ol{\UC}_{0,n}$ is the {\em universal curve} which is a compact space containing $\UC_{0,n}$ as an open dense subset.   It is well known that $\pi$ is a flat proper morphism, and there exist by construction $n$ sections $\sig_1\dots,\sig_n$ of $\pi$ such that $\sig_s(m)$ is the $s^{\rm th}$ marked point on the stable curve $\pi^{-1}(m)$. Note that $\ol{\UC}_{0,n}$ is actually isomorphic to $\ol{\Mod}_{0,n+1}$, and  $\ol{\Mod}_{0,n}$ is a smooth projective variety.

Fix a vector $\mu:=(\mu_1,\dots,\mu_n)\in (\R_{>0})^n$ such that $\mu_1+\dots+\mu_n=2$, and $\mu_s \not\in \N$ for all $s\in \{1,\dots,n\}$.  By~\cite{DeligneMostow86}, Section 3.13, there exists a rank one local system $\Lb^\mu$ on $\UC_{0,n}$
such that, for any $m\in\Mod_{0,n}$, the induced local system $\Lb^\mu_m$ on $\pi^{-1}(m)\simeq (\CP^1,(x_1,\dots,x_n))$ has monodromy given by $\alpha_s=\exp(2\pi\imath\mu_s)$ at each puncture $x_s$.
Since the projection $\pi:\UC_{0,n}\ra \Mod_{0,n}$ is locally topologically trivial,
setting $\Hb^{\mu}:=R^1\pi_*\Lb^\mu$ we get a local system of rank $n-2$ over $\Mod_{0,n}$ whose fiber over $m$ is $H^1(\CP^1\smin\Sig,\Lb_m^\mu)\simeq H^1_c(\CP^1\smin\Sig,\Lb_m^\mu)$. Associated to this local system is a flat projective space bundle $\mathbb{P}\Hb^\mu$ whose fiber over $m$  is $\mathbb{P}\Hb_m^\mu \simeq \CP^{n-3}$.

We have seen that for each $m\in \Mod_{0,n}$, up to a constant factor, there is a unique $\Lb_m^\mu$-valued meromorphic $1$-form $\omega_m\in \Gamma(\CP^1,j^m_*\Omega^1(\Lb_m^\mu))$ such that the valuation of $\omega_m$ at the puncture $x_s$ is exactly $-\mu_s$. By Proposition~\ref{prop:form:n:trivial}, we know that $\omega_m$ represents a non-trivial cohomology class in $\Hb_m^\mu$. Thus $\omega_m$ provides us with a section of the flat projective space bundle $\mathbb{P}\Hb^\mu$. Let us denote this section by $\Xi_\mu$.

Since  the pull-back of the bundle $\mathbb{P}\Hb^\mu$ to the universal cover $\widetilde{\Mod}_{0,n}$ is isomorphic to the trivial bundle $\widetilde{\Mod}_{0,n}\times\CP^{n-3}$, the section $\Xi_\mu$ gives rise to a map  $\widetilde{\Xi}_\mu: \widetilde{\Mod}_{0,n} \ra \CP^{n-3}$. We have the following crucial result

\begin{Proposition}[\cite{DeligneMostow86}, Lem. 3.5, Prop. 3.9]\label{prop:section:hol:etale}
 The section $\Xi_\mu$ is holomorphic, and the map $\widetilde{\Xi}_\mu: \widetilde{\Mod}_{0,n} \ra \CP^{n-3}$ is \'etale.
\end{Proposition}

A direct consequence of Proposition~\ref{prop:section:hol:etale} is that we have a holomorphic line bundle $\Lcal$ over $\Mod_{0,n}$ whose fiber over $m$ is the line $\C\cdot[\omega_m] \subset H^1(\CP^1\smin\Sig,\Lb_m^\mu)$.

Assume moreover that $0< \mu_s< 1$, for all $s \in \{1,\dots,n\}$. We have seen that in this case, the fiber $\Hb_m^\mu\simeq H^1(\CP^1\smin\Sig,\Lb_m^\mu)$ of the local system (flat bundle) $\Hb^\mu$ carries a Hermitian form of signature $(1,n-3)$. This Hermitian form then gives rise to a horizontal Hermitian metric on $\Hb^\mu$. Therefore, we have a flat bundle over $\Mod_{0,n}$ whose fiber over $m$ is the ball $\Bb_m \subset \mathbb{P}\Hb^\mu_m$ which is defined by
$$
\Bb_m:=\{\C\cdot v \in \mathbb{P}\Hb_m^\mu, \; ((v,v))_m >0\}.
$$
\noindent By Proposition~\ref{prop:herm:sign}, the line $\C\cdot[\omega_m]$ belongs to $\Bb_m$. Thus, the map $\widetilde{\Xi}_\mu$ actually takes values in a fixed ball $\Bb\subset \CP^{n-3}$. As a consequence, we see that $\Lcal$ is locally the pull-back by $\Xi_\mu$ of the restriction of the tautological line bundle of $\CP^{n-3}$ to $\Bb$.  Remark that $\Lcal$ carries  naturally a Hermitian metric induced by the Hermitian metric on $\Hb^\mu$. The line bundle $\Lcal$ and its Chern form will be our main focus in the rest of this paper.


\section{Local coordinates at boundary points of $\ol{\Mod}_{0,n}$}\label{sec:coord:bdry}
It is well-known that the complement of $\Mod_{0,n}$ in $\ol{\Mod}_{0,n}$ is the union of finitely many divisors called {\it vital divisors}, each of which uniquely corresponds to a partition of $\{1,\dots,n\}$ into two subsets $I_0\sqcup I_1$ such that $\min\{|I_0|,|I_1|\} \geq 2$. Let $\Pcal$ be the set of partitions satisfying this condition. For each partition $\Scal:=\{I_0,I_1\} \in \Pcal$, we denote by $D_\Scal$ the corresponding divisor in  $\ol{\Mod}_{0,n}$. Here below, we collect some classical facts on those divisors which are  relevant for our purpose (see \cite{Keel92}).

\begin{itemize}
 \item[(i)]  The family $\{D_\Scal, \, \Scal\in \Pcal\}$ consists of smooth divisors with normal crossings.

 \item[(ii)] If  $\Scal=\{I_0,I_1\}$, then $D_\Scal$ is isomorphic to $\ol{\Mod}_{0,|I_0|+1}\times \ol{\Mod}_{0,|I_1|+1}$.

 \item[(iii)] Let $\Scal=\{I_0,I_1\}$ and $\Scal'=\{J_0,J_1\}$ be two partitions in $\Pcal$.  Then $D_\Scal\cap D_{\Scal'}=\varnothing$ unless one of the following occurs:
 $$
 I_0 \subset J_0, \quad I_0\subset J_1, \quad I_1 \subset J_0, \quad I_1\subset J_1.
 $$
 \end{itemize}

We first  need to describe a neighborhood of a point $m$ in $\partial \Mod_{0,n}$. Fix a partition $\Scal=\{I_0,I_1\}\in \Pcal$. Let $n_0=|I_0|$, and $n_1=|I_1|=n-n_0$. Without loss of generality, we can assume that $I_0=\{1,\dots,{n_0}\}$ and $I_1=\{{n_0+1},\dots,n\}$.   From (ii) we know that $D_\Scal$ is isomorphic to $\ol{\Mod}_{0,n_0+1}\times\ol{\Mod}_{0,n_1+1}$. Let $m$ be a point in $D_\Scal$. We will only focus on the case when $m \in D_\Scal$ is a generic point, that is the fiber $C_m$ of $\pi$ over $m$ is a nodal curve having two irreducible components of genus zero intersecting at a simple node.

The normalization of $C_m$ consists of two Riemann surfaces of genus zero denoted by $C_m^0$ and $C_m^1$, where $C_m^0$ (resp. $C^1_m$) contains the marked points $x_1,\dots, x_{n_0}$  (resp.  $x_{n_0+1},\dots,x_n$).  Let $\Sig_0:=\{x_1,\dots,x_{n_0}\}$ and $\Sig_1:=\{x_{n_0+1},\dots,x_n\}$. There are two points $\hat{y}_0 \in C^0_m\smin\Sig_0$ and $\hat{y}_1 \in C^1_m\smin\Sig_1$ that correspond to the unique node of $C_m$. The marked curves $(C^0_m,(\hat{y}_0,x_1,\dots,x_{n_0}))$ and $(C^1_m,(\hat{y}_1,x_{n_0+1},\dots,x_n))$ represent respectively two points $m_0\in \Mod_{0,n_0+1}$ and $m_1\in \Mod_{0,n_1+1}$.

We will now describe how one can embed holomorphically a small disc $\Dis \subset \C$ centered at  $0$ into $\ol{\Mod}_{0,n}$  and transversely to $D_\Scal$ such that $0$ is mapped to $m$.  For this, let us fix the following data:
\begin{itemize}
 \item[.] $U$ is neighborhood of $\hat{y}_0$ in $C^0_m$ such that $U\cap\{x_1,\dots,x_{n_0}\}=\vide$, $F:U \ra \C$ is a coordinate mapping such that $F(\hat{y}_0)=0$,

 \item[.] $V$ is neighborhood of $\hat{y}_1$ in $C^1_m$ such that $V\cap\{x_{n_0+1},\dots,x_n\}=\vide$, $G:V \ra \C$ is a coordinate mapping such that $G(\hat{y}_1)=0$.
\end{itemize}

\noindent Pick a constant $c\in \R_{>0}$ such that the disc $\Dis_c:=\{|z|<c\}\subset \C$ is contained in both $F(U)$ and $G(V)$. For any $t \in \C$ such that $|t|< c^2$, set $C^0_{m,t}:=C^0_m\smin \{p \in U, \, |F(p)| \leq |t|/c\}$, and $C^1_{m,t}:=C^1_m\smin \{q \in V, \, |G(q)|\leq |t|/c\}$. Let $A_t$ denote the annulus $\{|t|/c < |z| < c\} \subset \Dis_c$.
We then define a compact Riemann surface by gluing $C^0_{m,t}$ and $C^1_{m,t}$ via the identification:  $p \in F^{-1}(A_t)$ is identified with $q \in G^{-1}(A_t)$ if and only if  $F(p)G(q)=t$. Let us denote the surface obtained from this construction by $C_{(m,t)}$.

It is easy to see that the marked curve $(C_{(m,t)}, (x_1,\dots,x_n))$ represents a point in $\Mod_{0,n}$.
We thus have  a map $\varphi: \Dis_{c^2}=\{t \in \C, |t| < c^2\} \ra \ol{\Mod}_{0,n}$, which is defined by $\varphi(0)=(C_m,\Sig)$ and $\varphi(t)=(C_{(m,t)},\Sig)$, for $t\neq 0$. This map is well known to be a holomorphic embedding of $\Dis_{c^2}$ into $\ol{\Mod}_{0,n}$.
The construction above is called a {\em plumbing}, and the image of $\Dis_{c^2}$ by $\varphi$ is called a {\em plumbing family} (see \cite[Sec. 2]{Wolpert90}).

Recall that  $m$ is identified with $(m_0,m_1)$ by the isomorphism between $D_\Scal$ and $ \Mod_{0,n_0+1}\times\Mod_{0,n_1+1}$. Therefore, we can identify a neighborhood $\Vcal$ of $m$ in $D_\Scal$ with a product space $\Vcal_0\times\Vcal_1$, where $\Vcal_i$ is a neighborhood of $m_i$ in  $\Mod_{0,n_i+1}$. For any $m'=(m'_0,m'_1) \in \Vcal$, let $C_{m'_0}$ and $C_{m'_1}$ be the curves represented by $m'_0$ and $m'_1$ respectively. On $C_{m'_i}$ we have a distinguished marked point $\hat{y}'_i$ which corresponds to the node of the curve $C_{m'}$  represented by $m'$. We can always identify $C_{m'_i}$ with $\CP^1$ such that $\hat{y}'_i=0$.  In conclusion, we get the following well known result (see \cite[Sec. 2]{Wolpert90}, \cite[Chap. 11]{ACG11}).

\begin{Proposition}\label{prop:neigh:bound:pt}
 Assume that for all $m'=(m'_0,m'_1) \in \Vcal_0\times\Vcal_1$ we have some plumbing data  $(U,V,F,G,c)$ as above, where $F$ and $G$ depend holomorphically on $m'$. Then there exists a system of holomorphic local coordinates at $m$ which identifies a neighborhood $\Ucal$ of $m$ in $\ol{\Mod}_{0,n}$ with  $\Vcal_0\times\Vcal_1\times\Dis_{c^2}$. The point in $\ol{\Mod}_{0,n}$ corresponding to $(m'_0,m'_1,t)$ represents the surface obtained by applying the $t$-plumbing construction to the nodal surface represented by $(m'_0,m'_1)$. In particular, $\Ucal\cap \Mod_{0,n}$ is identified with $\Vcal_0\times\Vcal_1\times\Dis^*_{c^2}$ in those coordinates.
\end{Proposition}

\section{Sections of $\Lcal$ near the boundary: generic points} \label{sec:sect:L:near:bdry}
In Section~\ref{sec:def:L}, we defined a holomorphic line bundle $\Lcal$ over ${\Mod}_{0,n}$ by providing local trivializations (see Proposition~\ref{prop:section:hol:etale}).
In this section, we  investigate  $\Lcal$ near the boundary of $\ol{\Mod}_{0,n}$. Our goal is to exhibit holomorphic sections of $\Lcal$ in a neighborhood  of every point $m \in \partial\ol{\Mod}_{0,n}$. Assume that $m$ is a generic point of a divisor $D_\Scal$. Let $\Scal=\{I_0,I_1\}, C^0_m, C^1_m,\Sig_0,\Sig_1, \hat{y}_0,\hat{y}_1$ be as in the previous section.  We will identify $C^0_m$ (resp. $C^1_m$) with $\CP^1$ in such a way that $\hat{y}_0=0$  and $\infty \not\in\Sig_0$ (resp. $\hat{y}_1=0$, and $\infty\not\in\Sig_1$). Set $\hat{\Sig}_i=\Sig_i\sqcup\{\hat{y}_i\}, \; i=0,1$.

Let $\hat{\mu}_0=\sum_{n_0+1\leq s \leq n} \mu_s$, $\hat{\mu}_1=\sum_{1 \leq s \leq n_0} \mu_s$ and $\hat{\alpha}_i=\exp(2\pi\imath\hat{\mu}_i)$.  We assume that $\hat{\mu}_0 < 1$. Denote by $\Lb_i$ the rank one local system on $C^i_m\smin\hat{\Sig}_i$ with  monodromy $\alpha_s$ at $x_s$, and $\hat{\alpha}_i$  at $\hat{y}_i$.  Let $\ee_i$ be a horizontal multivalued section of $\Lb_i$.  Set
$$
\omega_0=\ee_0\cdot z^{-\hat{\mu}_0}\prod_{1\leq s \leq n_0}(z-x_s)^{-\mu_s}dz \quad \text{ and } \quad \omega_1 := \ee_1\cdot z^{-\hat{\mu}_1}\prod_{n_0+1\leq s \leq n}(z-x_s)^{-\mu_s}dz.
$$
Observe that $\omega_i$ is a well defined section in $\Gamma(C^i_m,j^m_*\Omega^1(\Lb_i))$.

We are going to construct a plumbing family starting from $m$ and a section of $\Lcal$ over the corresponding (punctured) family. For this, we first need to fix the plumbing data.

\begin{Lemma}\label{lm:plumb:data}
Let $r$ be a real number different from $-1$. For any holomorphic function $f$ defined on a disc $\Dis$ in $\C$ centered at $0$ and satisfying $f(0)\neq 0$, there exists a coordinate change $z\mapsto w$ preserving $0$ such that
$$
z^{r}f(z)dz=w^rdw
$$
\noindent on a neighborhood of $0$ in $\Dis$, with suitable determinations of $z^r$ and $w^r$.
\end{Lemma}
\begin{proof}
Let us fix a determination of $z^r$. We will look for a coordinate change of the form $w =zh(z)$. It suffices to find a holomorphic function $h$ defined on a neighborhood of $0$  such that $h(0)\neq 0$ and
$$
z^rh(z)^{r}\left( h(z)+zh'(z) \right)=z^rf(z) \Leftrightarrow h(z)^{r}(h(z)+zh'(z))=f(z)
$$
where $h(z)^r$ is a determination defined near $h(0)\not=0$.
Setting $g(z):=h^{r+1}(z)$, we must have
$$
g(z)+\frac{1}{r+1}zg'(z)=f(z).
$$
Let $f(z)=\sum_{k\geq 0}c_kz^k$, with $c_0\neq 0$, be the expansion of $f$ at $0$. Assuming that $g$  admits an expansion $g(z)=\sum_{k\geq 0}d_kz^k$, we see that the sequence $(d_k)_{k\geq 0}$ must satisfy
$$
d_k\Bigl(1+\frac{k}{r+1}\Bigr)=c_k \Leftrightarrow d_k=\frac{r+1}{r+1+k}c_k.
$$
\noindent In particular, we see that $d_0=c_0\neq 0$, and the power series $\sum_{k\geq 0}d_kz^k$ has the same convergence radius as $\sum_{k\geq 0}c_kz^k$. Thus $g(z)$ is a well defined holomorphic function on $\Dis$ which satisfies $g(0)\neq0$. It follows that $h(z):=g(z)^{1/(r+1)}$ is  well defined in a neighborhood of $0$ for any choice of a determination. Then we choose the determination $h(z)^r$ in such a way that $h(0)^{r+1} =g(0)$, and if we define $w^r=(zh(z))^r:=z^r h(z)^r$, the lemma is proved.
\end{proof}

Now, choosing a determination for  $\prod_{1\leq s\leq n_0}(z-x_s)^{-\mu_s}$  and $\prod_{n_0+1\leq s\leq n}(z-x_s)^{-\mu_s}$ in a neighborhood of $0$, we then get two holomorphic functions $f$ and $g$ which do not vanish at $0$. Applying Lemma~\ref{lm:plumb:data} to the forms $z^{-\hat{\mu}_0}f(z)dz$ and $z^{-\hat{\mu}_1}g(z)dz$, we see that there exist two holomorphic functions $F: U \ra \C$ and $G: V \ra \C$, where $U$ and $V$ are some neighborhoods of $0$, such that $F(0)=G(0)=0,F'(0)\neq 0, G'(0)\neq 0$ and
\begin{equation}\label{eq:plumb:cond}
z^{-\hat{\mu}_0}f(z)dz= F^{-\hat{\mu}_0}(z)dF(z), \qquad z^{-\hat{\mu}_1}g(z)dz= G^{-\hat{\mu}_1}(z)dG(z)
\end{equation}

Let $c$ be a positive real number such that $\Dis_c$ is contained in both $F(U)$ and $G(V)$. We can now  use the tuple $(F,U,G,V,c)$ to construct the plumbing family associated to $m$.  For any  $t\in \Dis_{c^2}$, let $C_{(m,t)}$ be the $n$-punctured sphere obtained by the construction described in the previous section. Recall that $C_{(m,t)}$  is obtained from  $C^0_{m,t}$ and $C^1_{m,t}$ by the gluing rule  $w_1=t/w_0$ in the coordinates $w_0=F(z)$ and $w_1=G(z)$.  By the definition of $F$ and $G$, the expressions of $\omega_0$ and $\omega_1$ in those local coordinates are respectively
\begin{equation}\label{eq:plumb:gluing:rule}
\omega_0=\ee_1\cdot w_0^{-\hat{\mu}_0}dw_0, \qquad  \omega_1=\ee_2\cdot w_1^{-\hat{\mu}_1}dw_1.
\end{equation}

\begin{Lemma}\label{lm:loc:sys:L:glue}
There exists a rank one local system $\Lb$ on $C_{(m,t)}\smin\Sig$ whose restriction to $C^i_{m,t}\smin\Sig_i$ is $\Lb_i$. We also have a multivalued horizontal section $\ee$ of $\Lb$ whose restriction to $C^i_{m,t}$ is identified with $\ee_i$.
\end{Lemma}
\begin{proof}
We first remark that $C^i_{m,t}$ is  biholomorphic to a disc with $n_i$ punctures, and the annulus $A_t$ is homotopy equivalent  to the boundary of $C^i_{m,t}$. By definition, the monodromy of $\Lb_i$ along the boundary of $C^i_{m,t}$ (with the counterclockwise orientation) is given by $\exp(-2\pi\imath\hat{\mu}_i)$. Observe that the transition map identifies a circle homotopic to the boundary of $C^0_{m,t}$ with a circle homotopic to the boundary of $C^1_{m,t}$ with the inverse orientation. Since we have
$$
\exp(-2\pi\imath\hat{\mu}_1)=\exp(-2\pi\imath(2-\hat{\mu}_0))=\exp(2\pi\imath\hat{\mu}_0),
$$
\noindent the restriction of $\Lb_0$ on $A_t$ is isomorphic to the restriction of $\Lb_1$.  We can then identify $\Lb_0$ with $\Lb_1$ on $A_t$ by setting $\ee_0\simeq \ee_1$. Therefore, we have a well defined rank one local system $\Lb$ on $C_{(m,t)}$ with the desired monodromies at the punctures and a multivalued horizontal section, denoted by $\ee$, whose restriction to $C^i_{m,t}$ is $\ee_i$.
\end{proof}

\begin{Lemma}\label{lm:1-form:glue}
There exists a unique $\Lb$-valued meromorphic $1$-form $\omega\in \Gamma(C_{(m,t)},j^m_*\Omega^1(\Lb))$, whose restriction to $C^0_{m,t}$ is equal to $\omega_0$. Its restriction to $C^1_{m,t}$ is equal to $-t^{1-\hat{\mu}_0}\omega_1$ for some determination of $t^{\hat{\mu}_0}$.
\end{Lemma}
\begin{proof}
By definition, $\omega_i$ is a section of $\Omega^1(\Lb)$ on $C^i_{m,t}$, meromorphic at the punctures. All we need to show is that
\begin{equation}\label{eq:w1:w2:equal}
\omega_0=-t^{1-\hat{\mu}_0}\omega_1 \hbox{ on $A_t$,}
\end{equation}
the uniqueness being clear by analytic continuation. Using the local coordinates $w_0$ and $w_1$, we have (see \eqref{eq:plumb:gluing:rule})
$$
\omega_0=\ee_0\cdot w_0^{-\hat{\mu}_0}dw_0 \text{ and } \omega_1= \ee_1\cdot w_1^{-\hat{\mu}_1}dw_1
$$
for some choices of the determinations $w_0^{\hat\mu_0}$ and $w_1^{\hat\mu_1}$.
Recall that the changes of trivializations on $A_t$ satisfy  $w_0 \mapsto t/w_1$ and $\ee_0 \mapsto \ee_1$. Thus
$$
\omega_0=\ee_0\cdot w_0^{-\hat{\mu}_0}dw_0= \ee_1\cdot (w_1/t)^{\hat{\mu}_0}(-t/w_1^2)dw_1
= -t^{1-\hat{\mu}_0}\ee_1\cdot w_1^{-\hat{\mu}_1}dw_1=-t^{1-\hat{\mu}_0}\omega_1
$$
where $t^{\hat{\mu}_0}$ is chosen in such a way that $t^{\hat{\mu}_0}=(t/w_1)^{\hat{\mu}_0}w_1^{2-\hat{\mu}_1}$, which is possible since $\hat\mu_0+\hat\mu_1=2$. Observe that as $t$ completes a turn around $0$, the determination of $t^{\hat\mu_0}$ is multiplied by $e^{2\imath\pi\hat\mu_0}$.
\end{proof}

\medskip

Let $\T_0$ be an embedded tree in $C^0_m$ whose vertex set consists of $n_0$ points in  $\hat{\Sig}_0$. Let $\T_1$ be an embedded tree in $C^1_m$ whose vertex set is exactly $\Sig_1$. Let $a_1,\dots,a_{n_0-1}$ denote the edges of $\T_0$, and $b_1,\dots,b_{n_1-1}$ the edges of $\T_1$. Let $\ee'_i$ be an $\ol{\Lb}_i$-multivalued horizontal section on $C^i_m\smin\hat{\Sig}_i$.  By Proposition~\ref{prop:base:hom}, we know that $\{\ee'_0\cdot a_j,\, j=1,\dots,n_0-1\}$ (resp. $\{\ee'_1\cdot b_j,\, j=1,\dots,n_1-1\}$) is a basis of $H_1^{\rm lf}(C^0_m\smin\hat{\Sig}_0,\ol{\Lb}_0)$ (resp. a  basis of $H_1^{\rm lf}(C^1_m\smin\hat{\Sig}_1,\ol{\Lb}_1)$). Set
$$
\eta_j:=\langle [\ee'_0\cdot a_j], [\omega_0]\rangle, \quad \xi_j:=\langle [\ee'_1\cdot b_j], [\omega_1]\rangle.
$$
\noindent Since $[\omega_0]$ and $[\omega_1] $ are not zero (see Proposition~\ref{prop:form:n:trivial}), we have
$$
\eta:=(\eta_1,\dots,\eta_{n_0-1})\neq 0 \in \C^{n_0-1} \text{ and } \xi:= (\xi_1,\dots,\xi_{n_1-1}) \neq 0 \in \C^{n_1-1}.
$$

\begin{Lemma}\label{lm:pairings:glue}
Let $\omega$ be as in Lemma~\ref{lm:1-form:glue}. Then there exists a basis of $H_1^{\rm lf}(C_{(m,t)}\smin\Sig,\ol{\Lb})$ such that the coordinates of $\omega$ in the dual basis are given by $(\eta,-t^{1-\hat{\mu}_0}\xi) \in \C^{n-2}$.
\end{Lemma}
\begin{proof}
We first consider the case when $\T_0$ does not contain $\hat{y}_0$. The tree  $\T_i$ can always be chosen  to be  contained entirely in $C^i_{m,t}$.  It follows that $\{\ee'_0\cdot a_j\}$ and $\{\ee'_1\cdot b_j\}$ can be considered as homology classes in $H_1^{\rm lf}(C_{(m,t)}\smin\Sig, \ol{\Lb})$. Moreover, by Proposition~\ref{prop:base:hom}, the union of those classes makes up a basis of $H_1^{\rm lf}(C_{(m,t)}\smin\Sig,\ol{\Lb})$. Since the restrictions of $\omega$ to $C^0_{m,t}$ and $C^1_{m,t}$ are respectively  $\omega_0$ and $-t^{1-\hat{\mu}_0}\omega_1$, we get
\begin{eqnarray}
\label{eq:pair:coord:1} \langle [\ee'_0\cdot a_j], [\omega] \rangle & = & \langle [\ee'_0\cdot a_j], [\omega_0] \rangle = \eta_j, \, j=1,\dots,n_0-1,\\
\label{eq:pair:coord:2} \langle [\ee'_1\cdot b_j], [\omega] \rangle  & = & \langle [\ee'_1\cdot b_j], -t^{1-\hat{\mu}_0}[\omega_1]\rangle= -t^{1-\hat{\mu}_0}\xi_j, \, j=1,\dots,n_1-1.
\end{eqnarray}
\noindent Thus the lemma is proven in this case.

Consider now  the case where $\T_0$ contains $\hat{y}_0$. Remark that in this case, there is a point in $\Sig_0$, say $x_{n_0}$, which is not contained in $\T_0$. Up to a renumbering, we can assume that the set of edges containing $\hat{y}_0$ as an end is $\{a_j, \; j=1,\dots,k\}$, with $k \leq n_0-1$. We can also assume that $x_j$ is the other end of $a_j$,  for $j=1,\dots,k$.

Recall that the plumbing construction is carried out in a neighborhood $U$ of $\hat{y}_0$. Let $\Dis_0$ be an embedded disc in $C^0_m$ that contains $U$. For $j=1,\dots,k$, let $y_j$ be  the first intersection of $a_j$ with $\partial \Dis_0$, and $a'_j$ be the subarc of $a_j$ from $x_j$ to $y_j$.  Let $a''_j$ denote the boundary of $\Dis_0$ considered as a loop based at $y_j$. Since $\hat{\mu}_0\not\in \N$, there exists a constant $\varepsilon$ such that $[\ee'_0\cdot a_j]=[\ee'_0\cdot a'_j+ \varepsilon\ee'_0\cdot a''_j]$ in $H_1^{\rm lf}(C^0_m\smin\hat{\Sig}_0,\ol{\Lb}_0)$. Therefore,
\begin{equation*}
\eta_j=\langle [\ee'_0\cdot a'_j+ \varepsilon\ee'_0\cdot a''_j],[\omega_0]\rangle =\int_{a'_j}(\ee'_0,\omega_0)+\varepsilon\int_{a''_j}(\ee'_0,\omega_0), \quad j=1,\dots,k.
\end{equation*}

\noindent We construct a new tree $\T$ in $C_{(m,t)}$ from $\T_0$ and $\T_1$ by removing $a_1,\dots,a_k$ from $\T_0$, and adding the edges $c_j$ joining $x_j$ to some vertex of $\T_1$ for $j=1,\dots,k$. Note that the vertex set of $\T$ is $\Sig\smin\{x_{n_0}\}$.  Let $\ee'$ be an $\Lb$-multivalued horizontal section on $C_{(m,t)}\smin\Sig$. Then $\{[\ee'\cdot c_1], \dots, [\ee'\cdot c_k], [\ee'_0\cdot a_{k+1}],\dots, [\ee'_0\cdot a_{n_0-1}], [\ee'_1\cdot b_1],\dots,[\ee'_1\cdot b_{n_1-1}]\}$ is a basis of $H_1^{\rm lf}(C_{(m,t)}\smin\Sig,\ol{\Lb})$ by Proposition~\ref{prop:base:hom}.

For $j=1,\dots,k$, since $a'_j$ and $a''_j$ are entirely contained in $C^0_{m,t}$, we can consider $\ee'_0\cdot a'_j+\varepsilon\ee'_0 \cdot a''_j$ as elements of $H_1^{\rm lf}(C_{(m,t)}\smin\Sig,\ol{\Lb})$. Since the union of $\{a'_j,a''_j ,c_j , b_1,\dots,b_{n_1-1}\}$ is homotopic to the boundary of an open disc disjoint from $\Sig$, we deduce that
$[\ee'_0\cdot a'_j+\varepsilon\ee'_0 \cdot a''_j]$ is a linear combination of $[\ee'\cdot c_j]$ and $[\ee'_1\cdot b_1],\dots,[\ee'_1\cdot b_{n_1-1}]$. Therefore, $\{[\ee'_0\cdot a'_1+\varepsilon\ee'_0 \cdot a''_1],\dots, [\ee'_0\cdot a'_k+\varepsilon\ee'_0 \cdot a''_k], [\ee'_0\cdot a_{k+1}],\dots, [\ee'_0\cdot a_{n_0-1}], [\ee'_1\cdot b_1],\dots,[\ee'_1\cdot b_{n_1-1}] \}$ is also a basis of  $H_1^{\rm lf}(C_{(m,t)}\smin\Sig,\ol{\Lb})$. Since we have
$$
\langle [\ee'_0\cdot a'_j+\varepsilon\ee'_0 \cdot a''_j],[\omega]\rangle =\int_{a'_j}(\ee'_0,\omega)+\varepsilon\int_{a''_j}(\ee'_0,\omega)=\int_{a'_j}(\ee'_0,\omega_0)+\varepsilon\int_{a''_j}(\ee'_0,\omega_0)=\eta_j, \quad j=1,\dots,k,
$$
the coordinates of $[\omega]$ in the dual basis are given by $(\eta,-t^{1-\hat{\mu}_0}\xi)$.
\end{proof}

\begin{Remark} In the proof of Lemma~\ref{lm:pairings:glue}, we could have chosen $\T_0$ such that the node $\hat{y}_0$ is not contained in $\T_0$, and the proof would have been more direct. However, in the next section where we will treat the case where $m$ belongs to several divisors $D_\Scal$, we will be forced to deal with trees containing points corresponding to nodes and we will use a method which is similar to the one above (see Lemma~\ref{lm:pairings:glue:gen}).
\end{Remark}

\begin{Remark}\label{rk:monod:Dehn:tw}
  Let $\gamma$ be a small loop around $0$ in $\Dis_{c^2}$.   The element of the mapping class group ${\rm Mod}_{0,n}$ corresponding to $\gamma$ is a  Dehn twist around a closed curve on $\CP^1$ separating $\Sig_0$ and $\Sig_1$. It can be shown that  the monodromy of the local system $\Hb^\mu$ around such a loop is given by the matrix $\left(\begin{smallmatrix} \Id_{n_0-1} & 0 \\ 0 &  e^{2\pi\imath(1-\hat{\mu}_0)}\Id_{n_1-1} \end{smallmatrix}\right)$ (see \cite[Prop. 9.2]{DeligneMostow86} for the case $n_1=2$).
\end{Remark}


Recall that we can write $m=(m_0,m_1)$, where $m_i\in \Mod_{0,n_i+1}$ represents $(C^i_m,\hat{\Sig}_i)$. Fix a constant $c>0$. There exist some neighborhoods $\Vcal_i$ of $m_i$ in $\Mod_{0,n_i+1}$ such that for any $m'=(m'_0,m'_1)\in \Vcal_0\times\Vcal_1$ and $t\in \Dis^*_{c^2}$, we can apply the same plumbing construction with parameter $t$  as above to the curve $C_{m'}$. Let $C_{(m',t)}$ denote the resulting surface in $\Mod_{0,n}$. By  Proposition~\ref{prop:neigh:bound:pt}, this construction  identifies $\Vcal_0\times\Vcal_1\times \Dis_{c^2}$ with a neighborhood of $m$ in  $\ol{\Mod}_{0,n}$.

Let $C^i_{m'}$ be the component of  $C_{m'}$ containing $\Sig_i$. We define the sections  $\omega_{m'_i} \in \Gamma(C^i_{m'},j^m_*\Omega^1(\Lb_i))$ in the same manner as $\omega_i$.  Since $C_{(m',t)}$ is defined  by the same plumbing construction as $C_{(m,t)}$, by Lemma~\ref{lm:1-form:glue} we get an element $\omega_{(m',t)} \in \Gamma(C_{(m',t)},j^m_*\Omega^1(\Lb))$ constructed from $\omega_{m'_0}$ and $\omega_{m'_1}$. Since $\omega_{(m',t)}$ is a vector in the fiber of $\Lcal$ over $(m',t)$, the assignment   $\Phi: (m',t) \mapsto \omega_{(m',t)}$  is a section of  $\Lcal$ on $\Vcal_0\times\Vcal_1\times\Dis^*_{c^2}$.

\begin{Lemma}\label{lm:sect:L:near:bdry}
 $\Phi$ is a holomorphic section of  $\Lcal$ on $\Vcal_0\times\Vcal_1\times\Dis^*_{c^2}$.
\end{Lemma}
\begin{proof}
To see that $\Phi$ is holomorphic section of $\Lcal$, it is enough to show that the pairings of $\omega_{(m',t)}$ with a basis of $H_1^{\rm lf}(C_{(m',t)}\smin\Sig,\ol{\Lb})$ are holomorphic functions of $(m'_0,m'_1,t)$.  Since $\{\eta_j,\, j=1,\dots,n_0-1\}$ and $\{\xi_j,\, j=1,\dots,n_1-1\}$ are holomorphic functions of $m'_0$ and $m'_1$ respectively, the lemma is a direct consequence of Lemma~\ref{lm:pairings:glue}, see also~\cite[Sec. 3]{DeligneMostow86}.
\end{proof}

\section{Sections of $\Lcal$ near the boundary: general case}\label{sec:sect:L:bdry:gen}

\subsection{Principal component}\label{sec:princ:component}
Each point $m$ in  $\ol{\Mod}_{0,n}$ represents a  nodal curve $C_m$ with $n$ marked points $(x_1,\dots,x_n)$. Let   $C^0_m,\dots,C_m^{r}$ be the irreducible components of $C_m$.  The topological type of $C_m$ is encoded by a tree $\Tb$ whose vertex set is in bijection with the set of irreducible components. Each edge of this tree corresponds to a node of $C_m$.

The point $m$ belongs  to the intersection of $r$ boundary divisors, each of them being associated with one of the $r$ nodes $p_1,\dots,p_r$ as follows: splitting a node $p_j$ into two points, we get two connected components $C^{(0)}_{m,p_j}$ and $C_{m,p_j}^{(1)}$ from $C_m$. For $i=0,1$, we define the set $I^j_i\subset \{1,\dots,n\}$ as follows: $s\in I^j_i$ if and only if $x_s\in C^{(i)}_{m,p_j}$. Exchanging $C^{(0)}_{m,p_j}$ and $C_{m,p_j}^{(1)}$ if necessary, we will always assume that $\sum_{s\in I_1^j}\mu_s<1$. Set $\Scal_j=\{I_0^j,I_1^j\}$, we have $m\in\cap_{j=1}^r D_{\Scal_j}$.


Let $\Sig:=\{x_1,\dots,x_n\}$. For each component $C^j_m$,  set $\Sig_j := \Sig\cap C^j_m$. Note that $\Sig_j$  can be empty. We also have on $C^j_m$ some other marked points denoted by $\{y_1,\dots,y_{s_j}\}$ that correspond to nodes of $C_m$. Set $\hat{\Sig}_j:=\Sig_j\sqcup\{y_1,\dots,y_{s_j}\}$. We now assign to every point $y$ in $\hat{\Sig}_j$ a weight $\hat{\mu}(y)$ as follows: if $y=x_s \in \Sig_j$ then $\hat{\mu}(y)=\mu_s$. If $y \in \{y_1,\dots,y_{s_j}\}$, we have a corresponding node of $C_m$. Splitting this node into two points, we get two connected components of $C_m$. Let $\mathring{C}^j_{m,y}$ denote the component that {\bf does not} contain $C^j_m$. The weight associated to $y$ is then
\begin{equation}\label{eq:weight:node}
\hat{\mu}(y):= \sum_{x_s\in \mathring{C}^j_{m,y}}\mu_s.
\end{equation}

Since $y$ corresponds to a node, there exists another marked point $y'$ that is identified with $y$. Let $C^{j'}_m$ be the irreducible component that contains $y'$. Since the genus of $C_m$ is zero, we must have $j'\neq j$. As a consequence, we get
\begin{equation}\label{eq:weight:node:paired}
\hat{\mu}(y')=2-\hat{\mu}(y).
\end{equation}

Let $\hat{\mu}^j$ be the vector recording the weights of the points in $\hat{\Sig}_j$.

\begin{Lemma}\label{lm:exist:princ:comp}\hfill
 \begin{itemize}
  \item[a)] The sum of the weights in $\hat{\mu}^j$ is $2$.

  \item[b)] There exists a unique component $C^j_m$ such that all the weights in $\hat{\mu}^j$ are smaller than $1$.
 \end{itemize}
\end{Lemma}
\begin{proof}
The first assertion follows immediately from the definition of the weights at the points corresponding to nodes of $C_m$. We will prove the second assertion by induction on the number of vertices of $\Tb$.

If $\Tb$ has only one vertex, then  b) is trivially true. Suppose that $\Tb$ has $r+1$ vertices, with $r \geq 1$. Pick a component $C^j_m$ corresponding to a leaf of $\Tb$, that is a vertex which is connected to the rest of $\Tb$ by only one edge. Suppose that $C^j_m$ satisfies the property of the lemma ({\em i.e.} all the weights in $\hat{\mu}^j$ are smaller than one). Let us show that $C^j_m$ is the unique component satisfying this condition. Let $y$ be the unique point in $\hat{\Sig}_j$ that corresponds to a node in $C_m$. Since the weight $\hat{\mu}(y)$ is less than $1$, from a) we have $\sum_{x_s\in C^j_m}\mu_s >1$.

Consider  another irreducible component $C^k_m$ of $C_m$. There is a point $y_k \in C^k_m$ which corresponds to the node separating $C^k_m$ from $C^j_m$. Let $\mathring{C}^{k}_{m,y_k}$ be the component containing $C^j_m$ which is obtained after splitting $y_k$ into two points.  By definition, the weight of $y_k$ is
$$
\hat{\mu}(y_k)=\sum_{x_s\in \mathring{C}^{k}_{m,y_k}} \mu_s \geq \sum_{x_s\in C^j_m}\mu_s >1.
$$
Therefore, $C^k_m$ cannot satisfy the condition in b). We can then conclude that $C^j_m$ is the unique component that satisfies this condition.

Assume now that $C^j_m$ does not satisfy the condition of the lemma, which means that $\hat{\mu}(y)>1$. Let $C^k_m$ be the unique component of $C_m$ that is adjacent to $C^j_m$, and $y'$ be the point in $C^k_m$ that is identified with $y$. Note that the  weight of $y'$ is given by
$$
\hat{\mu}(y')=\sum_{x_s\in C^j_m}\mu_s =2-\hat{\mu}(y) <1.
$$
Set $\Sig':=(\Sig \setminus \Sig_j)\sqcup\{y'\}$. We see that each point in $\Sig'$ has a weight strictly smaller than $1$, and the total weight of the points in $\Sig'$ is $2$. Let $C'_m$ be the stable curve obtained by removing $C^j_m$ from $C_m$. Since the tree corresponding to $C'_m$ has a vertex less than $\Tb$, we can apply the induction hypothesis to conclude that there is a unique component of $C'_m$ that satisfies the desired condition.
\end{proof}

\begin{Definition}\label{def:princ:comp}
We  call the unique component $C^j_m$ that satisfies the condition that all the weights in $\hat{\mu}^j$ are smaller than $1$ the {\em $\mu$-principal component} of $C_m$.
\end{Definition}

In what follows, we will always assume that $C^0_m$ is the principal component of $C_m$. Let $\vv_j$ be the vertex of $\Tb$ corresponding to $C^j_m$. We consider $\vv_0$ as the root of $\Tb$, and set the length of every edge of $\Tb$ to be one. We define the level $\Lev_j$ of the component  $C^j_m$ to be the distance in $\Tb$ from $\vv_j$ to $\vv_0$. Observe that we can always choose a numbering of the components of $C_m$ such that $\Lev_j \leq \Lev_{j+1}$ for $j=0,\dots,r-1$.

If $C^j_m$ is not the principal component of $C_m$, then there is a unique point $\hat{y}_j \in \hat{\Sig}_j$ which corresponds to the node separating $C^j_m$ from $C^0_m$. Remark that we have  $\hat{\mu}(\hat{y}_j)>1$, and $\hat{y}_j$ is the unique point in $\hat{\Sig}_j$ whose weight is greater than $1$. We will call $\hat{y}_j$ the {\em principal node} of $C^j_m$, and define the {\em weight} of $C^j_m$ to be   $\nu_j=\hat{\mu}(\hat{y}_j)-1$. The following lemma provides some basic properties of the weights $\nu_j$. Its proof is straightforward from the definition of $\hat{\mu}$ and the fact that $\Tb$ is a tree.

\begin{Lemma}\label{lm:weight:of:comp} Let $C^j$ be an irreducible component of $C_m$ which is not the principal one. Then we have
\begin{itemize}
 \item[a)]  $0< \nu_j < 1$.

 \item[b)]  Let $\hat{C}^j_m$ be the component containing $C^j_m$ which is obtained by splitting $C_m$ at the principal node of $C^j_m$, that is the node separating $C^j_m$ from $C^0_m$. Then we have
 $$
 \nu_j=1-\sum_{x_s\in \hat{C}^j_m} \mu_s
 $$

 \item[c)] If $\vv_{k}$ is a vertex in the path from $\vv_0$ to $\vv_j$ and $k\neq j$, then $\nu_k < \nu_j$.
\end{itemize}
\end{Lemma}

\begin{Remark}\label{rk:pair:node}
 Every node of $C_m$ is the principal node of a unique component. This is because each node of $C_m$ corresponds to a pair of points $\{y,y'\}$ that are contained in two different components, and we have $\hat{\mu}(y)+\hat{\mu}(y')=2$ (cf. \eqref{eq:weight:node:paired}).
\end{Remark}

\subsection{Construction of sections of $\Lcal$ in a neighborhood of $m$}\label{sec:const:sect:L:gen}
Set $k_j:=|\hat{\Sig}_j|, \,  j=0,\dots,r$. For each $j\in \{0,\dots,r\}$, let $\Lb_j$ be a rank one local system on $C^j_m\smin\hat{\Sig}_j$ with monodromy $\exp(2\pi\imath\hat{\mu}(y))$ at any point $y\in \hat{\Sig}_j$. We will  fix an $\Lb_j$-multivalued horizontal section $\ee_j$, and a meromorphic  section $\omega_j$ of $\Gamma(C^j_m,j^m_*\Omega^1(\Lb_j))$ with valuation $-\hat{\mu}(y)$ at every point $y\in \hat{\Sig}_j$.

Let $\{p,q\}$ be a pair of points in the normalization of $C_m$ that correspond to a node, and $C^j_m$ and $C^{j'}_m$ be respectively the components that contain  $p$ and $q$.  Using Lemma~\ref{lm:plumb:data}, we can find some neighborhoods $U$ of $p$, and $V$ of $q$, together with local coordinates $F$ on $U$, $G$ on $V$   such that
$$
\omega_j=\ee_j \cdot F^{-\hat{\mu}(p)}dF, \quad \omega_{j'}=\ee_{j'}\cdot G^{-\hat{\mu}(q)}dG.
$$
Choose a constant $c>0$ small enough such that for any $t=(t_1,\dots,t_r)\in (\Dis_{c^2})^r$, the plumbing construction with plumbing data $(F,U,G,V)$ and parameter $t_i$ as above  can be carried out at all the nodes simultaneously. For any $j \in \{1,\dots,r\}$, we can assume that $t_j$ is the plumbing parameter at the principal node of $C^j_m$.

Let $C_{(m,t)}$ denote the resulting surface in $\Mod_{0,n}$.  On $C_{(m,t)}$ we have $n$ marked points $(x_1,\dots,x_n)$ with associated weights $(\mu_1,\dots,\mu_n)$, we will also denote by $\Sig$ this finite subset of $C_{(m,t)}$. Let $U^j_m$ be an open subset of $C^j_m$ containing all the points in $\Sig_j\subset \Sig$ and disjoint from the regions affected by the plumbing construction. We can consider $U^j_m$ as an open subset of $C_{(m,t)}$. As usual, let $\Lb$ be a rank one local system on $C_{(m,t)}\smin\Sig$, with monodromy $\exp(2\pi\imath\mu_s)$ at $x_s$.

\begin{Lemma}\label{lm:1-form:glue:general}
Let $\jj$ be the natural embedding of $C_{(m,t)}\smin\Sig$ into $C_{(m,t)}$.  Then there exists a unique element $\omega$ of $\Gamma(C_{(m,t)},\jj^m_*\Omega^1(\Lb))$ such that
  \begin{itemize}
  \item[$\bullet$] the restriction of $\omega$ to $U^0_m$ is equal to $\omega_0$,
  \item[$\bullet$] for $j=1,\dots,r$, the restriction of $\omega$    to $U^j_m$ is equal to $P_j(t)\omega_j$, where $P_j(t)$ is a function of $t$ which is defined as follows:  let  $0<i_1<\dots<i_{L_j}=j$ be the indices of the vertices of $\Tb$ that are contained in the unique path from $\vv_0$ to $\vv_j$, and $\nu_{i_s}$ is the weight  of $C^{i_s}_m$, then
  $$
  P_j(t_1,\dots,t_r)=(-1)^{L_j}\prod_{s=1}^{L_j}t_{i_s}^{\nu_{i_s}}.
  $$
  \end{itemize}
\end{Lemma}
\begin{proof}
Let $\tilde{C}^j_m$ be the subsurface of $C_m$ which is the union of the components $C^0_m,\dots,C^j_m$.  Given $t=(t_1,\dots,t_r)$, we define $\tilde{C}^j_{(m,t_1,\dots,t_j)}$ from $\tilde{C}^j_m$ by applying successively  the plumbing constructions with parameter $t_i$ at the principal node of $C^i_m$, for $i=1,\dots,j$. Note that $\tilde{C}^r_{(m,t_1,\dots,t_r)}=C_{(m,t)}$. By induction, this lemma is a direct consequence of Lemma~\ref{lm:1-form:glue}.
\end{proof}

Let us denote by $\omega_{(m,t)}$ the $\Lb$-valued meromorphic one form given by Lemma~\ref{lm:1-form:glue:general}. By construction, $\omega_{(m,t)}$ has valuation $-\mu_s$ at $x_s$, thus it is an element of the fiber of $\Lcal$ over $C_{(m,t)}$. We would like now to show that the assignment $(m,t) \mapsto \omega_{(m,t)}$ is a holomorphic section of $\Lcal$ on $\Ucal\cap{\Mod}_{0,n}$, where $\Ucal$ is a neighborhood of $m$ in $\ol{\Mod}_{0,n}$.

We first specify an appropriate basis of $H_1^{\rm lf}(C_{(m,t)}\smin\Sig,\ol{\Lb})$.  Let $\T_0$ be an embedded topological tree in $C^0_m$, whose vertex set is $\hat{\Sig}_0$ minus one point. For $j=1,\dots,r$, let $\T_j$ be a an embedded topological tree in $C^j_m$ whose vertex set is $\hat{\Sig}_j$ minus the principal node. Let $a^j_i, \, i=1,\dots,k_j-2$, denote the edges of $\T_j$. Fix an $\ol{\Lb}_j$-multivalued horizontal section $\ee'_j$ on $C^j_m\smin\hat{\Sig}_j$. By Proposition~\ref{prop:base:hom}, the family $\{[\ee'_j\cdot a^j_i], \, i=1,\dots,k_j-2\}$ is a basis of $H_1^{\rm lf}(C^j_m\smin\hat{\Sig}_j,\ol{\Lb}_j)$. Set
$$
\xi^{(j)}_i:= \langle [\ee'_j\cdot a^j_i],[\omega_j]\rangle= \int_{a^j_i} (\ee'_j,\omega_j), \, i=1,\dots,k_j-2.
$$
Let $\xi^{(j)}$ denote the vector $(\xi^{(j)}_1,\dots,\xi^{(j)}_{k_j-2})$.  We have

\begin{Lemma}\label{lm:pairings:glue:gen}
 Let $\omega$ and $P_j, \, j=1,\dots,r$, be as in Lemma~\ref{lm:1-form:glue:general}. Then there exists a basis of $H_1^{\rm lf}(C_{(m,t)}\smin\Sig,\ol{\Lb})$ such that the coordinates of $[\omega]$ in the dual basis are given by $(\xi^{(0)},P_1(t)\xi^{(1)},\dots,P_r(t)\xi^{(r)})\in \C^{n-2}$.
\end{Lemma}
\begin{proof}
 Let $\tilde{C}^j_m$ and $\tilde{C}^j_{(m,t_1,\dots,t_j)}$ be as in the proof of Lemma~\ref{lm:1-form:glue:general}. Recall that $C_{(m,t)}=\tilde{C}^r_{(m,t_1,\dots,t_r)}$, and $\tilde{C}^j_{(m,t_1,\dots,t_j)}$ is obtained from $\tilde{C}^{j-1}_{(m,t_1,\dots,t_{j-1})}$ and $C^j_m$ by the plumbing construction at the principal node of $C^j_m$. The lemma then follows from Lemma~\ref{lm:pairings:glue} and Lemma~\ref{lm:1-form:glue:general} by induction.
\end{proof}

Each pair $(C^j_m,\hat{\Sig}_j)$ represents a point $m_j$ in $\Mod_{0,k_j}$. Hence the point $m$ is contained in a stratum of $\ol{\Mod}_{0,n}$ which is isomorphic to $\Mod_{0,k_0}\times\dots\times\Mod_{0,k_{r}}$. Let $\Vcal_j$ be a neighborhood of $m_j$ in $\Mod_{0,k_j}$ and set $\Vcal:=\Vcal_0\times\dots\times\Vcal_r$. Let $\Lcal_j$ denote the line bundle over $\Mod_{0,k_j}$ whose fiber over $m_j$ is $\C\cdot[\omega_j] \subset H^1(C^j_m\smin\hat{\Sig}_j,\Lb_j)$. We  extend $\omega_j$ to a holomorphic section of $\Lcal_j$ on $\Vcal_j$. Since the plumbing data $(F,U,G,V)$ depend analytically on $(m_0,\dots,m_r)$, the plumbing construction $(m,t) \mapsto C_{(m,t)}$ identifies a neighborhood of $m$ in $\ol{\Mod}_{0,n}$ with $\Vcal\times (\Dis_{c^2})^r$.

Let $\omega_{(m,t)}$ denote the $\Lb$-valued meromorphic one form on $C_{(m,t)}$ defined in Lemma~\ref{lm:1-form:glue:general}. The assignment $\Phi: (m,t) \mapsto \omega_{(m,t)}$ provides us with a section  of $\Lcal$ on $\Vcal\times(\Dis^*_{c^2})^r$. Since $\omega_j$ is a holomorphic section of $\Lcal_j$, $\xi^{(j)}$ depends analytically on $m_j$. It follows that the  coordinates of $\omega_{(m,t)}$ in a basis of $H^1(C_{(m,t)}\smin\Sig,\Lb)$ are given by holomorphic functions of $(m_0,\dots,m_r,t_1,\dots,t_r)$. Thus we have shown
\begin{Proposition}\label{prop:sect:L:bdry:hol}
 The section $\Phi$ is holomorphic.
\end{Proposition}

\section{Flat metrics on punctured spheres and Hermitian metric on the line bundle $\Lcal$} \label{sec:metric:coord}
\subsection{Hermitian norm of the section $\Phi$}

Let $m$ be now a point in a stratum $\Mod:=\Mod_{0,k_0}\times\dots\times\Mod_{0,k_r}$ of codimension $r$ in $\ol{\Mod}_{0,n}$. Let $(C_m,\Sig)$ be the stable curve represented by $m$, and $C^0_m,\dots,C^r_m$ its irreducible components.  In what follows we will use the notations of Section~\ref{sec:sect:L:bdry:gen}.  Our goal in this section is to prove a formula (cf. \eqref{eq:norm:bdry:gen}) for the Hermitian norm of $\Phi(m,t)$ in $H^1(C_{(m,t)}\smin\Sig,\Lb)$, where $\Phi$ is the section in Proposition~\ref{prop:sect:L:bdry:hol}.

On each irreducible component $C^j_m$ of $C_m$, we have a finite subset $\hat{\Sig}_j$ consisting of points in $\Sig\cap C^j_m$ and the nodes of $C^j_m$. The pair $(C^j_m,\hat{\Sig}_j)$ represents a point $m_j \in \Mod_{0,k_j}$, where $k_j=|\hat{\Sig}_j|$. We identified a neighborhood of $m$ in $\ol{\Mod}_{0,n}$ with $\Vcal_0\times\dots\times\Vcal_r\times(\Dis_{c^2})^r$, where $\Vcal_j \subset \Mod_{0,k_j}$ is a neighborhood of $m_j:=(C^j_m,\hat{\Sig}_j)$, and $c$ is a positive real constant small enough.

Let $z^j \in \C^{k_j-3}, \, j=0,\dots,r$, be the coordinates on $\Vcal_j$, and $t=(t_1,\dots,t_r)$ the coordinates on $(\Dis_{c^2})^r$. In these local coordinates, $m$ is identified with the point $(z^0(m_0),\dots,z^r(m_r),0\dots,0)$, and  we have $\Vcal\times(\Dis^*_{c^2})^r=\Vcal\times(\Dis_{c^2})^r\cap\Mod_{0,n}$, where $\Vcal=\Vcal_0\times\dots\times\Vcal_r$.

Remark that for each $j\in \{1,\dots,r\}$, the subset of $\Vcal\times(\Dis_{c^2})^r$ defined by $\{t_j=0\}$  is the intersection of $\Vcal\times(\Dis_{c^2})^r$ with a boundary divisor $D_{\Scal_j}$ in $\ol{\Mod}_{0,n}$. This divisor corresponds to the partition $\Scal_j=\{I^j_0,I^j_1\}$ of $\{1,\dots,n\}$ that is induced by the splitting of the $j$-th node of $C_m$ into two points. Thus, we see that the stratum $\Mod$ of $m$ is precisely the intersection $\cap_{1\leq j \leq r}D_{\Scal_j}$.

Recall that $\hat{\mu}^j$ is the vector recording the weights of marked points in $C^j_m$, and the component $C^0_m$ of $C_m$ is characterized by the property that all the weights in $\hat{\mu}^0$ are strictly smaller than $1$ (see Lemma~\ref{lm:exist:princ:comp}). Thus the Hermitian form on $H^1(C^0_m\smin\hat{\Sig}_0,\Lb_0)$ has signature $(1,k_0-3)$. Let us denote this Hermitian form by $((.,.))_0$. Recall that we have defined a section $\Phi: (m,t) \mapsto \omega_{(m,t)}$ of $\Lcal$ on $\Vcal\times(\Dis^*_{c^2})^r$ (see Proposition~\ref{prop:sect:L:bdry:hol}). We will prove the following

\begin{Proposition}\label{prop:norm:bdry:gen}
For $j=1,\dots,r$, let $P_j(t)$ be as in Lemma~\ref{lm:1-form:glue:general}, and $\xi^{(j)} \in \C^{k_j-2}$ be as in Lemma~\ref{lm:pairings:glue:gen}.   For each  $j=1,\dots,r$, there exists a positive definite  Hermitian form $((.,.))_j$ on $\C^{k_j-2}$ depending only on $\mu$ such that the norm of $[\omega_{(m,t)}]$ in $H^1(C_{(m,t)}\smin\Sig,\Lb)$ is given by
\begin{equation}\label{eq:norm:bdry:gen}
 (([\omega_{(m,t)}],[\omega_{(m,t)}]))=((\xi^{(0)},\xi^{(0)}))_0-\sum_{j=1}^r|P_j(t)|^2((\xi^{(j)},\xi^{(j)}))_j
\end{equation}
Here we identify $((.,.))_0$ with a Hermitian form on $\C^{k_0-2}$.
\end{Proposition}


As a consequence of Proposition~\ref{prop:norm:bdry:gen}, we get

\begin{Proposition}\label{prop:Chern:form}
There exist some neighborhood $\Vcal_j$ of $m_j$ and holomorphic local coordinates $z^{j}: \Vcal_j \ra \C^{k_j-3}$ such that if  $m=(z^{0},\dots,z^{r},\underset{r}{\underbrace{0,\dots,0}})$, and  $t=(t_1,\dots,t_r) \in (\Dis^*_{c^2})^r$, then we have
\begin{equation}\label{eq:norm:Phi:normalized}
(([\omega_{(m,t)}],[\omega_{(m,t)}]))=1-||z^{0}||^2-\sum_{j=1}^r|P_j(t)|^2(1+||z^{j}||^2).
\end{equation}
and the Chern form of $\Lcal$ on $\Vcal\times(\Dis^*_{c^2})$ is given by
\begin{equation}\label{eq:Chern:form}
 \Omega_\mu:=dd^c\log\left(1-||z^{0}||^2-\sum_{j=1}^r|P_j(t)|^2(1+||z^{j}||^2)\right).
\end{equation}
In other words, locally at $m$, $\Omega_\mu$ is the pullback of the complex hyperbolic metric $dd^c\log(1-\|w\|^2)$ on $\C^{n-3}$ by the multivalued map
$$
(z^0,t_1,\dots,t_r,z^1,\dots, z^r)\mapsto w=(z^0,P_1(t),\dots,P_r(t),P_1(t)z^1,\dots,P_r(t) z^r)
$$
(note that even if the map is multivalued, the metric is well defined).
\end{Proposition}

\begin{Remark}\label{rk:Omega:cont}
 Recall from Lemma~\ref{lm:1-form:glue:general} that we have $P_j(t)=(-1)^{L_j}\prod_{s=1}^{L_j}t_{i_s}^{\nu_{i_s}}$,  where the family of indices $\{i_s, \, s=1,\dots,L_j\}$ records the components of $C_m$ between the principal one, {\em i.e.} $C^0_m$, and $C^j_m$. Since all the exponents $\nu_{i_s}$ are positive, the function $|P_j(t)|$  extends by continuity to  $(\Dis_{c^2})^r$. Thus, the function
 $$
 \varphi: (z^0,\dots,z^r,t) \mapsto 1-||z^0||^2-\sum_{j=1}^r|P_j(t)|^2(1+||z^{j}||^2)
 $$
 is a continuous function on $\Vcal\times(\Dis_{c^2})^r$.

 As we will see in the sequel, the Hermitian norm  $(([\omega_{(m,t)}],[\omega_{(m,t)}]))$ can be interpreted as the area of the flat surface defined by $\omega_{(m,t)}$. From this viewpoint, the continuity of $\varphi$ on $(\Dis_{c^2})^r$ reflects the fact that as $t$ converges to $0\in (\Dis_{c^2})^r$, the metric defined by $\omega_{(m,t)}$ ``converges'' to the metric defined by $\omega_0$ on the principal component of $C_m$. This convergence in the space of flat metrics is the key point in the construction of the (metric) completion of $\Mod_{0,n}$ introduced by Thurston~\cite{Thurston98}.
\end{Remark}

\begin{proof}[Proof of Prop.~\ref{prop:Chern:form} assuming Prop.~\ref{prop:norm:bdry:gen}]
 Recall that we have  a rank one local system $\Lb_j$ on $C^j_m\smin \hat{\Sig}_j$ whose monodromy at the points in $\hat{\Sig}$ are given by $\hat{\mu}^j$. This local system gives rise to a local system $\Hb_j$ of rank $k_j-2$ and a holomorphic line  bundle $\Lcal_j$ on $\Mod_{0,k_j}$.  Let $\Xi_{\mu_j}$ denote the section of the bundle $\mathbb{P}\Hb_j$ defined in Proposition~\ref{prop:section:hol:etale}.  By construction, $\omega_j$ is a vector in the line $\Xi_{\mu_j}(m_j)$. Let $\xi^{(j)}=(\xi^{(j)}_1,\dots,\xi^{(j)}_{k_j-2})$ be the coordinates of $\omega_j$ in some basis of $H^1(C^j\smin\hat{\Sig}_j,\Lb_j)$.  We can choose the basis of $H^1(C^j_m\smin\hat{\Sig}_j,\Lb_j)$ such that
 $$
 ((\xi^{(0)},\xi^{(0)}))_0=-\sum_{i=1}^{k_0-3}|\xi^{(0)}_i|^2 + |\xi^{(0)}_{k_0-2}|^2,
 $$
 and
 $$
 ((\xi^{(j)},\xi^{(j)}))_j=\sum_{i=1}^{k_j-2}|\xi^{(j)}_i|^2, \; \text{ for } j=1,\dots,r.
 $$
Recall that $\omega_j\in H^1(C^j\smin\hat{\Sig}_j,\Lb_j)$ is not trivial so that we can normalize our coordinates in such a way that $\xi^{(j)}_{k_j-2}=1$. Hence $(\xi^{(j)}_1,\dots,\xi^{(j)}_{k_j-3})$ are the coordinates of $\Xi_{\mu_j}(m_j)$ is some local chart of $\mathbb{P}H^1(C^j_m\smin\hat{\Sig}_j,\Lb_j)$. Since $\Xi_{\mu_j}$ is \'etale by Proposition~\ref{prop:section:hol:etale}, we can use $z^j_i=\xi^{(j)}_i, \, i=1,\dots,k_j-3$, to define local coordinates in a neighborhood of  $m_j$.   The proposition then follows from \eqref{eq:norm:bdry:gen}.
\end{proof}

We will spend the rest of this section  to prove Proposition~\ref{prop:norm:bdry:gen}. For this purpose, we will make use of the flat metric approach introduced by Thurston~\cite{Thurston98}.

\subsection{Thurston's coordinates}\label{sec:Thurston:coord}
Let us first recall Thurston's coordinates on the moduli space of flat metrics on the sphere with prescribed cone angles at singularities (see \cite[Prop. 3.2]{Thurston98}).
Fix a vector $(\theta_1,\dots,\theta_n)$, with $ 0< \theta_s < 2\pi$, such that $\theta_1+\dots+\theta_n=2\pi(n-2)$. Let $M$ denote a flat surface homeomorphic to $\S^2$ with conical singularities  denoted by $x_1,\dots,x_n$, and the cone angle at $x_s$ being $\theta_s$. Let $\T$ be a tree whose vertex set consists  of $n-1$  points in $\{x_1,\dots,x_n\}$ and all the edges are geodesics (it is not difficult to show that such a tree always exists).   Choosing an orientation for every edge of $\T$, then using a  developing map, one can associate to each edge of $\T$ a complex number (see~\cite[pp. 525-526]{Thurston98}). We  then get a vector $Z(M)$ in $\C^{n-2}$ associated to $M$.

For any flat metric (with the same prescribed cone angles at the singularities) close to $M$, one can also find a geodesic tree isomorphic to $\T$.
Hence, we also get an associated vector in $\C^{n-2}$ in the same way. It turns out that this correspondence defines a local chart for the space of flat metrics (with prescribed cone angles) on the sphere. Up to homothety,  this  space  can be identified with $\Mod_{0,n}$. Therefore, this construction also yields a local coordinate system for $\Mod_{0,n}$.

Let $m=(\CP^1,\{x_1,\dots,x_n\})\in \Mod_{0,n}$ be the point corresponding to the homothety class of $M$.  Assume that all the cone angles at the singularities are smaller than $2\pi$.  In \cite{Thurston98}, it was proved that the area of $M$ can be expressed as a Hermitian form $\Ar$ of signature $(1,n-3)$ in the coordinates of $Z(M)$, that is
$$
\Aa(M)={}^t\ol{Z(M)}\cdot \Ar \cdot Z(M).
$$
Consequently, the induced local chart on $\Mod_{0,n}$ identifies a neighborhood of $m$ with an open in the ball $\Bb:=\{\langle v \rangle, \; { }^t\bar{v}\Ar v >0\} \subset \CP^{n-3}$. By a classical construction, $\Ar$ induces a complex hyperbolic metric on $\Bb$. Since the area is an invariant of the flat metric, this complex hyperbolic metric is invariant by the coordinate changes. Therefore, we get a well defined complex hyperbolic metric structure on $\Mod_{0,n}$. \medskip

Set $\mu_s:= 1-\theta_s/(2\pi)$, and $\mu=(\mu_1,\dots,\mu_n)$. By definition, $M$ is isometric to $(\CP^1\smin\Sig,\gb)$, where $\Sig=\{x_1,\dots,x_n\}$, and $\gb$ is a flat metric  on $\CP^1\smin\Sig$ such that each $x_s$ has a neighborhood isometric to an Euclidean cone of angle $\theta_s$. Without loss of generality, we can assume that $\infty \not\in \Sig$. Remark  that $\prod_{1\leq s \leq n} |z-x_s|^{-2\mu_s}|dz|^2$ is a flat metric with the same singularities and the same cone angles as $\gb$. Therefore, we must have
$\gb(z)=\lambda^2\prod_{1\leq s \leq n} |z-x_s|^{-2\mu_s}|dz|^2$, where $\lambda$ is a positive real number.

Let $\Lb$ be the rank one local system on $\CP^1\smin\Sig$ with monodromy $\exp(2\pi\imath\mu_s)$  at $x_s$. Choose a horizontal Hermitian metric for $\Lb$, and let $\ee$ be an $\Lb$-multivalued horizontal section such that the norm of $\ee$ is $1$. Let
$$
\omega=\lambda \ee \cdot\prod_{1\leq s \leq n} (z-x_s)^{-\mu_s}dz.
$$
\noindent Then $\gb(z)$ is the metric associated to the $(1,1)$-form $\omega\wedge\ol{\omega}$. Recall that we have  a Hermitian form $((.,.))$ on $H_c^1(\CP^1\smin\Sig,\Lb)\simeq H^1(\CP^1\smin\Sig,\Lb)$ of signature $(1,n-3)$. By Proposition~\ref{prop:herm:coh:class}, we have
\begin{equation}\label{eq:norm:area}
(([\omega],[\omega]))=\int_{\CP^1\smin\Sig}\omega\wedge\ol{\omega}=\Aa(M).
\end{equation}

\medskip

Fix a base point  $p\in \CP^1\smin\Sig$ and consider the universal cover $(\Delta,\tilde{p})$ of $(\CP^1\smin\Sig,p)$. Let $f$ be a determination of the multivalued function  $\lambda(z-x_1)^{-\mu_1}\dots(z-x_n)^{-\mu_n}$ in a neighborhood $U$ of $p$. We also  denote by $f$ its pullback to a neighborhood $\tilde{U}$ of $\tilde{p}$. Let $\varphi$ be a holomorphic function on $\tilde{U}$ such that $f=\varphi'$. Let $z$ be the coordinate on $\Delta$, and set $w=\varphi(z)$. Observe that we have
$$
\varphi^*dw =f(z)dz, \text{ and } \varphi^*|dw|^2=|f(z)|^2|dz|^2,
$$
\noindent which means that $\varphi$ realizes an isometry between a neighborhood of $\tilde{p}$ (with the metric $\gb$)  and an open subset of $\C$ with the standard Euclidean metric $|dw|^2$. In other words,  $\varphi$ is a developing map for $\gb$. Therefore, we can  extend $\varphi$ to a locally isometric map from $(\Delta,\tilde{\gb})$ to $(\C,|dw|^2)$.

Now let $a$ be an oriented edge of the tree $\T$. The complex number associated to $a$ is given by $\int_{\varphi(\hat{a})}dw=\int_{\hat{a}}f(z)dz$,  where $\hat{a}$ is a component of the pre-image of $a$ in $\Delta$. We can consider $\ee\cdot a$ as an element of $H_1(\CP^1\smin\Sig,\Lb)$, therefore, we can  write
$$
\int_{\hat{a}}f(z)dz=([\ee\cdot a],[\ee\cdot f(z)dz]).
$$
\noindent From Proposition~\ref{prop:base:hom}, we know that the set $\{[\ee\cdot a ], \; a \text{ is an edge of } \T\}$ is a basis of $H_1(\CP^1\smin\Sig,\Lb)$. Since the pairing
$H_1^{\rm lf}(\CP^1\smin\Sig,\Lb)\otimes H^1(\CP^1\smin\Sig,\Lb) \ra \C$ is perfect, it follows that the cohomology class of $\omega$ is (locally) uniquely determined by the vector $Z(M) \in \C^{n-2}$.

By definition, the hyperbolic metric on $\Mod_{0,n}$ is the pullback of the complex hyperbolic metric on the ball $\Bb\subset\CP^{n-3}$. This metric is  defined by the Chern form of the  tautological line bundle over  $\Bb\subset \CP^{n-3}$. Recall that $\C\cdot[\omega]$ is the fiber of $\Lcal$ over $m$, and $\Lcal$ is actually the pullback of the tautological bundle on $\Bb$ by the map $\Xi_\mu$ (see Proposition~\ref{prop:section:hol:etale}).  Thus we have proved the following

\begin{Proposition}\label{prop:Thurston:coord:equiv}
The Thurston local coordinates  on $\Mod_{0,n}$ are defined by the section $\Xi_\mu$, and the Hermitian form $\Ar$ on $\C^{n-2}$ is induced by the Hermitian form $((.,.))$  on $H^1(\CP^1\smin\Sig,\Lb)$. Moreover, the complex hyperbolic metric on $\Mod_{0,n}$ is the one induced by the Chern form of $(\Lcal,((.,.)))$.
\end{Proposition}

\subsection{Thurston's surgery on flat surfaces}\label{sec:surgery}
We now describe the {\em cone adding} construction introduced in \cite[pp. 520-521]{Thurston98}, which is the key idea of the proof that the signature of $\Ar$ is $(1,n-3)$. Let $M$ be a flat surface homeomorphic to the sphere which has $n$ conical singular points as above. Recall that $\mu_s$ is the curvature at the cone point $x_s$.  Suppose now that we are given a geodesic arc $e$ on $M$ joining $x_i$ to $x_j$ and $\mu_i+\mu_j < 1$.

We first construct an Euclidean cone whose apex angle  is $2\pi(1-\mu_i-\mu_j)$ as follows:
Let $(ABC)$  be a triangle in $\R^2$ whose interior angles at $A,B,C$ are given by $((1-\mu_i-\mu_j)\pi, \mu_i\pi,\mu_j\pi )$ respectively, and the length of $\ol{BC}$ is equal to the length of $e$.
Let $(A'B'C')$ be the image of $(ABC)$ by the mirror symmetry. We now glue $\ol{AC}$ to $\ol{A'C'}$, and $\ol{AB}$ to $\ol{A'B'}$ by identifications respecting the order of endpoints.
We then obtain a flat surface homeomorphic to a disc, which has a singular point $\hat{x}$ with cone angle  $2\pi(1-\mu_i-\mu_j)$ in the interior. The boundary of this disc is the union of two geodesic segments corresponding to $\ol{BC}$ and $\ol{B'C'}$. Let $y_1$ (resp. $y_j$) denote the identification of $B$ and $B'$ (resp. of $C$ and $C'$). The interior angles  at $y_i$ and $y_j$ are respectively  $2\pi\mu_i, 2\pi\mu_j$.

We now slit open $M$ along $e$ and glue the cone constructed above to  this surface  in such a way that $y_i$ (resp. $y_j$) is identified with $x_i$ (resp. with $x_j$). Since $e$ and $\ol{BC}$ have the same length, the gluings are realized by isometries. We thus have a flat surface $\hat{M}$ homeomorphic to $\S^2$. By construction, the cone angles at $x_i$ and $x_j$ in $\hat{M}$ are now equal to $2\pi$, which means that $x_i$ and $x_j$ are regular points in $\hat{M}$. Therefore, $\hat{M}$ has exactly $n-1$ singularities: $x_s$ with $s \not\in \{i,j\}$, and $\hat{x}$. Remark that  $e$  corresponds to a loop on $\hat{M}$ consisting of two geodesic arcs,  we will call $e$ and the corresponding loop the {\em base} of the added cone. We record here below some key properties  of this construction.
\begin{itemize}
 \item[$\bullet$] The triangle $(ABC)$ is uniquely determined up to isometry, since its angles are  determined by $\mu_i$ and $\mu_j$, and the length of $\ol{BC}$ is equal to the length of $e$. It follows that there exists a positive  constant $\kappa(\mu_i,\mu_j)$ such that $\Aa((ABC))=\kappa(\mu_i,\mu_j)|e|^2$, where $|e|$ is the length of $e$.

 \item[$\bullet$]We have
 $$
 \Aa(\hat{M})-\Aa(M)=2\Aa((ABC))=2\kappa(\mu_i,\mu_j)|e|^2.
 $$

 \item[$\bullet$] The sides of $(ABC)$ can be considered as geodesic segments in $\hat{M}$.  Thus, given a developing map of $\hat{M}$, we can associate to those segments the complex numbers $z(\ol{BC}), z(\ol{CA}), z(\ol{AB})$. There exist some complex  numbers $c_1,c_2$ depending only on $(\mu_1,\mu_2)$  such that
 $$
 z(\ol{AB})=c_1z(\ol{BC}), \text{ and }  z(\ol{AC})=c_2z(\ol{BC}).
 $$
 \item[$\bullet$] We can apply similar constructions to $\hat{M}$ to get other surfaces with less singularities as long as there are two singular points such that the sum of the corresponding curvatures is less than $1$.
\end{itemize}

\subsection{Flat surfaces with convex boundary}\label{sec:convex:bdry}

For our purpose, we will need to consider flat surfaces with boundary.  In what follows, by a {\em flat surface with convex boundary} we will mean a topological surface with boundary $M$ equipped with a flat metric structure with conical singularities  satisfying the following property: for any point $x \in \partial M$, there is a  neighborhood of $x$ which is isometric to a convex domain in $\R^2$.   For such a surface, any path of minimal length (in a fixed homotopy class) joining two points in the interior does not intersect the boundary.

Let $\Sig$ denote the set of cone singularities in $\mathrm{int}(M)$.  We will also need a generalized notion of homotopy on $M$. A pair of arcs $\gamma_0,\gamma_1: [0,1] \ra M$ are said to be {\em homotopic in $M\smin\Sig$ with fixed endpoints}  if we have $\gamma_0(0)=\gamma_1(0)=x, \gamma_0(1)=\gamma_1(1)=y$, and there exists a continuous map $H: [0,1]\times[0,1] \ra M$ such that $H(.,0)=\gamma_0, H(.,1)=\gamma_1$,  $H(0,.)=\{x\}, H(1,.)=\{y\}$, and $H((0,1)\times(0,1)) \subset M\smin\Sig$.
With this definition, a path with two endpoints in $\Sig$  not passing through any other point in $\Sig$ may be homotopic to the union of some arcs  with endpoints in $\Sig$.
Remark that given any developing map for the flat metric, the complex numbers associated to  two homotopic paths (that is the difference in $\C$ of the two endpoints) must be the same.

\medskip

We now suppose that $M$ is a flat surface with convex boundary. Let $\Sig=\{x_1,\dots,x_n\}$ denote the set of cone points of $M$, and  assume that $\Sig$ is contained in the interior of $M$.  All the cone angles $\theta_s$ at  $x_s$  are supposed to be smaller than $2\pi$,  and
\begin{equation}\label{eq:sum:curv:sm:1}
\sum_{1\leq s\leq n} (2\pi-\theta_s) < 2\pi.
\end{equation}

Set $\mu_s=1-\theta_s/(2\pi)$. The condition \eqref{eq:sum:curv:sm:1} is equivalent to $\mu_1+\dots+\mu_s < 1$. Let $e$ be the path of minimal length from $x_1$ to $x_2$. Note that $e$ is contained in the interior of $M$ (since $M$ has convex boundary), and $e$ does not pass through any other point in $\Sig$. Let $M'$ be the flat surface obtained by  slitting open $M$ along $e$. One of the boundary component of $M'$  consists of two copies of $e$, which will be denoted by $e_1$ and $e_2$.  Since $\mu_1+\mu_2<1$, we can glue an Euclidean cone $\Cb$ of apex curvature $1-\mu_1-\mu_2$ to $M'$ along this boundary component.
Let $\hat{M}$ denote the new surface. Remark that $\hat{M}$ also has convex boundary.   We consider $M'$ as a subsurface of $\hat{M}$.  The singular points $x_1,x_2$ of $M$ now correspond to two regular points in $\hat{M}$, we denote those points by the same notation. Let $\hat{x}$ be the apex of $\Cb$, and  set $\hat{\Sig}=\{x_3,\dots,x_n\}\sqcup\{\hat{x}\}$.
It is worth noticing that given a path in $M$ which does not cross $e$, then its image by a developing map for $\hat{M}$ is also the image of a developing map for $M$.  In view of the proof of Proposition~\ref{prop:norm:bdry:gen},  we will need the following lemma.

\begin{Lemma}\label{lm:add:cone:coord:funct}
 Let $a$ be a geodesic segment in $\hat{M}$ with endpoints in $\hat{\Sig}$. We assume that the two endpoints of $a$ are distinct, and $a$ does not contain any point in $\hat{\Sig}$ in its interior. Then there exist a piecewise geodesic path $b$  in $M'$ connecting two points in $\Sig$,  and a constant $\kappa \in \C$ such that, for a fixed choice of the developing map on the universal cover of  $\hat{M}$,  we have
 $$
 z(a)=\kappa z(e_1)+z(b),
 $$
\noindent where $z(a),z(e_1),z(b)$ are the complex numbers associated to $a,e_1,b$ respectively.
\end{Lemma}
\begin{proof}
 We have two cases:
 \begin{itemize}
 \item[$\bullet$] Case 1: $a$ does not contain $\hat{x}$.   Since $\Cb\smin\{\hat{x}\}$ is homeomorphic to a punctured disc, $M'$ is a deformation retract of $\hat{M}\smin\{\hat{x}\}$.  Let $b$ be the image of $a$ by this retraction, then $b$ is homotopic in $\hat{M}$ to $a$. Thus we have $ z(a)=z(b)$.

 \item[$\bullet$] Case 2: $a$ contains $\hat{x}$. By assumption, we can consider $a$ as a ray starting from $\hat{x}$ and ending at a point $x_s \in \Sig$. Let $y$ be the first intersection of $a$ with $\partial \Cb=e_1\cup e_2$.  Denote by $a_0$ (resp. $a_1$) the subsegment of $a$ between $\hat{x}$ and $y$ (resp. between $y$ and $x_s$). Let $e'_1$ be the geodesic segment in $\partial \Cb$ from $x_1$ to $y$.  Set $b_0=e'_1*a_0$ and $b_1:=a_1*e'_1$.    Observe that $a$ is homotopic (in $\hat{M}$) to the path $b_1*b_0$. Since $b_1$ does not contain $\hat{x}$, from Case 1, we know that it is homotopic to a piecewise geodesic path $b$ from $x_1$ to $x_s$. We thus have $ z(a)=z(b_0)+z(b)$.  But we have seen that $z(b_0)=\kappa z(e_1)$, where $\kappa$ is a complex number determined by $(\mu_1,\mu_2)$. Hence the lemma is proved for this case.
\end{itemize}
\end{proof}

\subsection{Infinite flat metric structures}\label{sec:infinite:surf}
Let $k\geq 1$ and fix a vector $\mu=(\mu_0,\dots,\mu_k)\in \R_{>0}^{k+1}$ such that $\mu_0+\dots+\mu_k=2$, where $\mu_0>1$,  but $\mu_i< 1$, for $i=1,\dots,k$. Let $\Sig=\{x_0,\dots,x_k\}$ be a set of $k+1$ points in $\CP^1$ and  $\Lb$ be the rank one local system on $\CP^1\smin\Sig$ whose monodromy at $x_s$ is $\exp(2\pi\imath\mu_s)$. Fix a horizontal multivalued section $\ee$ of $\Lb$ with Hermitian norm equal to $1$. Let $\omega$ be a meromorphic section of $\Omega^1(\Lb)$ with valuation $-\mu_s$ at $x_s$. We can write
$$
\omega=\lambda\ee\cdot\prod_{0\leq s \leq k}(z-x_s)^{-\mu_s}dz
$$
\noindent with $\lambda\in\C^*$. Note that  since $\mu_0>1$, we have $\int_{\CP^1\smin\Sig}\omega\wedge\ol{\omega} =\infty$, so $\omega$ is not of the first kind.

Let $\T$ be an embedded tree in $\CP^1$ whose vertex set is $\Sig_0=\{x_1,\dots,x_k\}$. Let $a_j, \, j=1,\dots,k-1,$ denote the edges of $\T$. From Proposition~\ref{prop:base:hom}, we know that the family $\{[\ee\cdot a_j], \, j=1,\dots,k-1\}$ is a basis of $H_1^{\rm lf}(\CP^1\smin\Sig,\Lb)$. Recall that $\omega$ represents a cohomology class in $H^1(\CP^1\smin\Sig,\Lb)\simeq H^1_c(\CP^1\smin\Sig,\Lb)$.
Set $z_j:=\langle [\ee\cdot a_j],[\omega]\rangle$, and  $Z:=(z_1,\dots,z_{k-1})\in \C^{k-1}$. Since the valuations of $\omega$ at the endpoints of $a_j$ are all greater than $-1$, we can write
$$
z_j=\lambda\int_{a_j}(z-x_0)^{-\mu_0}\dots(z-x_k)^{-\mu_k}dz.
$$

Observe that the $(1,1)$-form $\omega\wedge\ol{\omega}$ defines a flat metric structure on $\CP^1\smin\Sig$, with conical singularity at $x_s$ for $s=1,\dots,k$.  The cone angle at $x_s$ is $\theta_s=2\pi(1-\mu_s)$. Note that this is an infinite metric structure since any geodesic ray cannot reach $x_0$ in finite time. Let $M$ denote the corresponding flat surface.

Let $e_1$ be a path of minimal length in  $M$ joining $x_1$ and $x_2$, such a path must be a geodesic segment which does not contain any singularity in its interior. By assumption, we have  $\mu_1+\mu_2 <1$. Thus we can add a cone $\Cb_1$ over  $e_1$ to get a surface  with $k-1$ singularities.  By construction, the curvature at the new singularity  is $\mu_1+\mu_2$.
One can continue adding $k-2$ cones $\Cb_2,\dots,\Cb_{k-1}$ to obtain successively the surfaces $M_2,\dots, M_{k-1}$, where $M_i$ has a cone singularity with curvature $\mu_1+\dots+\mu_{i+1}$ at some point denoted by $\hat{x}_i$, and $M_{i+1}$ is obtained from $M_i$ by adding the cone $\Cb_{i+1}$ whose base is a geodesic arc, denoted by $e_{i+1}$, joining $\hat{x}_{i}$ and $x_{i+1}$.
Note that there exists a positive real constant $c_i$ depending only on $(\mu_1,\dots,\mu_k)$ such that $\Aa(\Cb_i)=c_i|e_i|^2$.
Remark also that $M_i$ has $k-i$ singularities, and $M_{k-1}$ is an infinite Euclidean cone with apex angle equal to $2\pi(1-(\mu_1+\dots+\mu_k))$.

Choosing a developing map for $M_0=M$, we get a complex number $w_1$ associated to $e_1$. We can extend the developing map of $M_0$ to get a developing map of $\Cb_1$.
By construction, $e_2$ is a geodesic ray starting from $\hat{x}_1$ (the apex of $\Cb_1$), therefore we can extend this developing map to get a complex number $w_2$ associated to $e_2$.
Continuing this process, we get a vector $W=(w_1,\dots,w_{k-1})\in \C^{k-1}$, where $w_i$ is the complex number associated to $e_i$.
We have the following lemma,  which is implicit  in the proof of \cite[Prop. 3.3]{Thurston98}.
\begin{Lemma}\label{lm:add:cone:linear:coord}
The complex number $w_i$ is a linear function of $Z$ for $i=1,\dots,k-1$.
\end{Lemma}
\begin{proof}
We will prove this lemma by induction. Recall that $w_1$ is the complex number associated to the geodesic arc $e_1$ on $M_0=M$. But this number can be interpreted as the pairing of the homology class $[\ee\cdot e_1]$ with $[\omega]$, hence it is a linear function of $Z$. Note also that by the same argument, the complex number associated to any path in $M_0$ with endpoints in $\{x_1,\dots,x_k\}$ is a linear function of $Z$.

Consider the flat surface $M_1$. As a Riemann surface, $M_1$ can be identified with $\CP^1$. Set  $\hat{\Sig}_1:=\{x_0,\hat{x}_1,x_3,\dots,x_k\}$. Let $\Lb_1$ be the rank one local system on $\CP^1\smin\hat{\Sig}_1$ with monodromy $\exp(2\imath\pi\mu_s)$ at $x_s$, for $s=0,3,\dots,k$, and $\exp(2\imath\pi(\mu_1+\mu_2))$ at $\hat{x}_1$.
The flat metric of $M_1$ is thus induced by a $\Lb_1$-valued meromorphic $1$-form $\omega_1 \in \Gamma(\CP^1,\jj_*^m\Omega(\Lb_1))$ with valuation $-\mu_s$ at $x_s$, for $s=0,3,\dots,k$, and $-(\mu_1+\mu_2)$ at $\hat{x}_1$.

Let $\T_1$ be an embedded tree in $M_1$ whose vertex set is equal to $\{\hat{x}_1,x_3,\dots,x_k\}$, and whose edges are geodesic segments in $M_1$. One can construct such a tree by seeking for instance the paths of minimal length joining $\hat{x}_1$ to the other cone points $x_3,\dots,x_k$. Let $a^1_1,\dots,a^1_{k-2}$ denote the edges of $\T_1$. Fix a multivalued horizontal section $\ee_1$ of $\Lb_1$.  Then $\{[\ee_1\cdot a^1_1],\dots, [\ee_1\cdot a^1_{k-2}]\}$ is a basis of $H^{\rm lf}_1(\CP^1\smin\hat{\Sig},\Lb_1) \simeq H_1(\CP^1\smin\hat{\Sig},\Lb_1)$.

Set $z^1_j:=\langle [\ee_1\cdot a^1_j],[\omega_1]\rangle$, for $j=1,\dots,k-2$.  Since the complex number associated to any path with endpoints in $\hat{\Sig}_1\smin\{x_0\}$ can be also interpreted as the pairing of $[\omega_1]$ with a homology class in $H_1^{\rm lf}(\CP^1\smin\hat{\Sig}_1,\Lb_1)$, it follows that such a number is a linear function of $Z^1:=(z^1_1,\dots,z^1_{k-2})$. From Lemma~\ref{lm:add:cone:coord:funct}, we deduce that the $z^1_j$'s are linear functions of the vector $Z$. Therefore, the complex number associated to any path in $M_1$ with endpoints in $\hat{\Sig}_1\smin\{x_0\}$ is a linear function of $Z$. In particular $w_2$ is a linear function of $Z$. The rest of the proof follows from an induction argument.
\end{proof}

Lemma~\ref{lm:add:cone:linear:coord} implies that the correspondence  $\Psi: \C^{k-1} \ra \C^{k-1}, (z_1,\dots,z_{k-1}) \mapsto (w_1,\dots,w_{k-1})$ is a linear map. Our goal now  is to show the following.
\begin{Proposition}\label{prop:inf:surf:coord:change}
The linear map $\Psi$ is an isomorphism.
\end{Proposition}
\begin{proof}
Let $\Lcal$ be the holomorphic line bundle over $\Mod_{0,{k+1}}$ associated to the weight vector $\mu$ (see Section~\ref{sec:def:L}).
To show that $\Psi$ is an isomorphism, we will show that $\Psi$ is injective in a neighborhood of $Z$.
For this, we consider $\omega$ as an element in the fiber of the line bundle $\Lcal$ over  the point $m=(\CP^1,\{x_0,x_1,\dots,x_k\}) \in \Mod_{0,k+1}$, and identify a neighborhood $\Vcal$ of $Z$ in $\C^{k-1}$  with a neighborhood of $\omega$ in the total space of $\Lcal$.

We can always assume that $x_0=\infty, x_1=0,x_2=1$. A point $m'$ in $\Mod_{0,k+1}$ close to $m$ corresponds to a tuple $(\CP^1,\{\infty,0,1,x'_3,\dots,x'_k\})$, with $x'_i$ close to $x_i$. Hence an element of $\Lcal$ close to $\omega$ can be written as
$$
\omega'=\lambda'\ee\cdot z^{-\mu_1}(z-1)^{-\mu_2}\prod_{i=3}^k(z-x'_i)^{-\mu_i}dz
$$
\noindent where $\lambda'\in\C$ is close to $\lambda$, and $\ee$ is considered as a  horizontal section of $\Lb$ on $m'$.

Assume that we have $Z'$ and $Z''$ in $\Vcal$ such that $\Psi(Z')=\Psi(Z'')=W'=(w'_1,\dots,w'_{k-1})$.  Let $\omega'$ and $\omega''$ be the points in $\Lcal$ corresponding to $Z'$ and $Z''$. The projections of $\omega'$ and $\omega''$ in $\Mod_{0,k+1}$ are denoted by $m'$ and $m''$.

Let $M'$ and $M''$ denote the flat surfaces defined by $\omega'$ and $\omega''$.  By definition, the vector $W'$ records the complex numbers associated to the bases of $k-1$ cones added to $M'$ (resp. to $M''$) to obtain a flat surface $M'_{k-1}$  (resp. $M''_{k-1}$) with a single singularity. Observe that the surfaces $M'_{k-1}$ and $M''_{k-1}$ are both  isometric to a standard infinite Euclidian  cone $\Cb$ with apex angle $2\pi(1-(\mu_1+\dots+\mu_k))$. For the sake of concreteness, $\Cb$ is defined by the flat metric $|z|^{-2(\mu_1+\dots+\mu_k)}|dz|^2$ on $\C$. Note also that $\Cb$ is also isometric to $M_{k-1}$.


Given $\Cb\simeq M_{k-1}$, we can recover $M$ from $W=(w_1,\dots,w_{k-1})$ as follows:  since $\Cb_{k-1}$ is a neighborhood of the apex of $M_{k-1}$, we can choose a developing map for $M_{k-1}$ such that the complex number associated to one of the geodesic segments in the base of $\Cb_{k-1}$ is $w_{k-1}$.
Cut off the cone $\Cb_{k-1}$,  and glue the two geodesic segments in the base of $\Cb_{k-1}$, we  obtain the flat surface $M_{k-2}$ having two singularities.
By construction, the cone $\Cb_{k-2}$ is a neighborhood of one of these singularities. The complex number $w_{k-2}$ determines the embedding of $\Cb_{k-2}$ into $M_{k-2}$.
Therefore, we can then continue  the cutting-regluing operation to remove  the remaining $k-2$ cones and  get back to the surface $M'$.
Note that along this process, one needs to keep track of the  developing map chosen for $\Cb\simeq M_{k-1}$.

Clearly, we can recover $M'$ and $M''$ from $W'$ and $W''$ in the same way. Since $M'_{k-1}$ and $M''_{k-1}$ are isometric, and $W'=W''$, we can conclude that $M'$ and $M''$ are isometric.
The isometry between $M'$ and $M''$ induces an isomorphism between $m'$ and $m''$. Therefore, we have $m'=m''$, which means that  $\omega'$ and $\omega''$ belong to the same fiber of $\Lcal$. Hence, there is a complex number $\lambda$ such that $\omega'=\lambda\omega''$, or equivalently  $Z'=\lambda Z''$.  Since $\Psi$ is a linear map, we have $\Psi(Z')=\lambda\Psi(Z'') \Leftrightarrow W'=\lambda W'$. Recall that by construction, all the coordinates of $W'$ are non-zero, thus we must have $\lambda=1$, and $Z'=Z''$. The proposition is then proved.
\end{proof}

\subsection{Proof of Proposition~\ref{prop:norm:bdry:gen}: case of codimension one}\label{sec:pf:norm:bdry:codim1}
We now give the proof of Proposition~\ref{prop:norm:bdry:gen}  in  the case $r=1$, that is $m$ is a generic point in a divisor $D_\Scal$, where  $\Scal=\{I_0,I_1\}\in\Pcal$ (see Section~\ref{sec:coord:bdry}).

We can assume that $I_0=\{1,\dots,n_0\}$ and $I_1=\{n_0+1,\dots,n\}$. Let $C^0_m,C^1_m$ be the corresponding irreducible components of $C_m$. For $i=0,1$, let $\hat{\mu}_i, \hat{y}_i, \hat{\Sig}_i, \Lb_i, \omega_i, \T_i$  be as in Section~\ref{sec:sect:L:near:bdry}.

Let $((.,.))_0$ be the Hermitian form on $H^1(C^0_m\smin\hat{\Sig}_0,\Lb_0)$. By Proposition~\ref{prop:herm:sign}, we know that $((.,.))_0$ has signature $(1,n_0-2)$.
Let $\xi^{(0)}=(\xi^{(0)}_1,\dots,\xi^{(0)}_{n_0-1}) \in \C^{n_0-1}$ (resp. $\xi^{(1)}=(\xi^{(1)}_1,\dots,\xi^{(1)}_{n_1-1}) \in \C^{n_1-1}$) be the vector recording the pairings of $[\omega_0]$ (resp. $[\omega_1]$) with the basis of $H_1^{\rm lf}(C^0_m\smin\hat{\Sig}_0,\ol{\Lb}_0)$ associated to $\T_0$ (resp. the basis of $H_1^{\rm lf}(C^1_m\smin\hat{\Sig}_1,\ol{\Lb}_1)$ associated to $\T_1$).
Let $C_{(m,t)}$ be the stable curve obtained from the plumbing construction in Section~\ref{sec:sect:L:near:bdry}, where $t \in \Dis_{c^2}$.  We need to show the following

\begin{Proposition}\label{prop:area:form:bdry:r1}
There exists a positive definite Hermitian form $((.,.))_1$ on $\C^{n_1-1}$ depending only on $(\mu_{n_0+1},\dots,\mu_n)$ such that, if $\omega_{(m,t)}$ is the element of $H^{1,0}(C_{(m,t)}\smin\Sig,\Lb)$ defined in Lemma~\ref{lm:1-form:glue}, then we have
\begin{equation}\label{eq:area:near:bdry}
 (([\omega_{(m,t)}],[\omega_{(m,t)}]))=((\xi^{(0)},\xi^{(0)}))_0-|t|^{2(1-\hat{\mu}_0)}((\xi^{(1)},\xi^{(1)}))_1.
 \end{equation}
\end{Proposition}
\begin{proof}
Let $M_0$ be the flat surface defined by $\omega_0$ on $C^0_m$, $M_1$ the surface defined by $-t^{1-\hat{\mu}_0}\omega_1$ on $C_m^1$, and $M$ the flat surface defined by $\omega_{(m,t)}$ on $C_{(m,t)}$. Let $(F,U,G,V,c)$ be the plumbing data as in Section~\ref{sec:sect:L:near:bdry}.
Choose a constant $\tau \in (|t|/c,c)$, and let $\gamma_\tau$ be the curve in $C_{(m,t)}$ which corresponds to  the set $\{p \in U, |F(p)|=\tau\} \simeq \{q\in V, |G(q)|=|t|/\tau\}$.
Since the metric defined by $\omega_0$ in $U$ is the pullback of $|F(z)|^{-2\hat{\mu}_0}|dF(z)|^2$, we deduce  that $\gamma_\tau$ is the set of points  whose distance in $M_0$ to $\hat{y}_0$ is  $R(\tau):=\frac{\tau^{1-\hat{\mu}_0}}{1-\hat{\mu}_0}$. In particular, $\gamma_\tau$ has constant curvature $1/R(\tau)$ and length equal to $2\pi\tau^{1-\hat{\mu}_0}$.

The curve $\gamma_\tau$ cuts $M$ into two subsurfaces, the one that contains $\Sig_i$ is denoted by $M_i^\tau$.   The observation above implies that $\partial M_1^\tau$ is convex.
Remark that  $M_i^\tau$ can be viewed as a subsurface of $M_i$.
By definition, $M_1^\tau$ contains $n_1$ cone singularities corresponding to  the points in $\Sig_1$  in its interior.
Since the sum of the curvatures at those points is smaller than $2\pi$, one can add  $n_1-1$ cones $\Cb_1,\dots,\Cb_{n_1-1}$ to $M_1^\tau$ to get a  flat surface $\hat{M}_1^\tau$ having a single singularity with cone angle $2\pi(1-\hat{\mu}_0)$.

Let $W=(w_1,\dots,w_{n_1-1})$ be the vector recording the complex numbers associated to the bases of the cones $\Cb_1,\dots,\Cb_{n-1}$.  By Proposition~\ref{prop:Thurston:coord:equiv}, there is a linear isomorphism $\Psi$ of $\C^{n_1-1}$ such that $W=-t^{1-\hat{\mu}_0}\Psi(\xi^{(1)})$.

Recall that the total area of the added cones is equal to $\sum_{i=1}^{n_1-1}c_i|w_i|^2$, where the $c_i$'s are real positive constants determined by the weight vector $(\mu_{n_0+1},\dots,\mu_n)$. Therefore, there is a positive definite Hermitian form $((.,.))_1$  on $\C^{n_1-1}$ such that
 $$
 \sum_{i=1}^{n_1-1}\Aa(\Cb_i)= |t|^{2(1-\hat{\mu}_0)}((\xi^{(1)},\xi^{(1)}))_1.
 $$
We now remark that $\hat{M}_1^\tau$ is isometric to a subset of the Euclidean cone $\Cb$ defined by the metric $|z|^{-2\hat{\mu}_0}|dz|^2$ on $\C$. Since $\partial \hat{M}_1^\tau \simeq \gamma_\tau$ has constant curvature $1/R(\tau)$, $\gamma_\tau$ corresponds to  the set of points in $\Cb$ whose distance to the apex is $R(\tau)$.
Hence $\hat{M}_1^\tau$ is isometric to the flat surface defined by $|z|^{-2\hat{\mu}_0}|dz|^2$ on the disc $\Dis_\tau$.
It follows that $\hat{M}_1^\tau$ is isometric to the flat metric defined by $\omega_0$ on the set $\{p \in U, |F(p)| \leq  \tau\}$.

Let $\hat{M}^\tau$ be the flat surface obtained by gluing $\hat{M}_1^\tau$ to $M^\tau_0$ along $\gamma_\tau$. From the argument above, we conclude that $\hat{M}^\tau$ is isometric to $M_0$. Since $\Aa(M)=\Aa(\hat{M}^\tau)-\sum_{1 \leq i \leq n_1-1} \Aa(\Cb_i)=\Aa(M_0)-\sum_{1\leq i \leq n_1}\Aa(\Cb_i)$, we have
$$
(([\omega_{(m,t)}],[\omega_{(m,t)}])):=(([\omega_0],[\omega_0]))-|t|^{2(1-\hat{\mu}_0)}((\xi^{(1)},\xi^{(1)}))_1=((\xi^{(0)},\xi^{(0)}))_0-|t|^{2(1-\hat{\mu}_0)}((\xi^{(1)},\xi^{(1)}))_1.
$$
\end{proof}

\begin{Remark}\label{rk:Th:cone:angle}
 In \cite{Thurston98}, Thurston introduced a completion  $\ol{\Mod}^\mu_{0,n}$ of $\Mod_{0,n}$ with respect to the complex  hyperbolic metric induced by $\Lcal$. The space $\ol{\Mod}^\mu_{0,n}$ is equipped with a {\em cone-manifold} structure. In this setting, $m$  corresponds to a point in a stratum of codimension $n_1-1$ representing the flat surfaces on which all the cone points in $\Sig_1$ collide. The quantity $1-\hat{\mu}_0$ can be interpreted as the {\em scalar cone angle} at $m$ (see \cite[Sec. 3]{Thurston98}). Note also that if $m=(m_0,m_1)$ with $m_i\in \Mod_{0,n_i+1}$, then the flat surface corresponding to $m$ is uniquely determined (up to a rescaling) by $m_0$. Thus  for all $m'_1 \in \Mod_{0,n_1+1}$, the the point $m'=(m_0,m'_1)$ represents the same element of $\ol{\Mod}^\mu_{0,n}$.
\end{Remark}

\subsection{Proof of Proposition~\ref{prop:norm:bdry:gen}: general case}
\begin{proof}
 Let $\tilde{C}^j_{(m,t_1,\dots,t_j)}$ be as in the proof of Lemma~\ref{lm:1-form:glue:general}, where  $\tilde{C}^r_{(m,t_1,\dots,t_r)}=C_{(m,t)}$. Recall that $\tilde{C}^{j+1}_{(m,t_1,\dots,t_{j+1})}$ is obtained from $\tilde{C}^j_{(m,t_1,\dots,t_j)}$ and $C^{j+1}_m$ by a plumbing construction at the principal node of $C^{j+1}_m$. Therefore Proposition \ref{prop:norm:bdry:gen} follows from Proposition~\ref{prop:area:form:bdry:r1} and Lemma~\ref{lm:1-form:glue:general} by induction.
\end{proof}

\section{Singular K\"ahler-Einstein metrics}\label{section:SKE}

Our aim now is to explain that the metric constructed in the previous section on $\Mod_{0,n}$ is actually a singular K\"ahler-Einstein metric on $\ol{\Mod}_{0,n}$, and that this fact will enable us to compute the volume of $\Mod_{0,n}$ endowed with this metric.

\subsection{General setting}\label{general:section}

We first recall some basic facts about singular K\"ahler-Einstein metrics on projective varieties in a simplified setting as we will not need a very high degree of generality (for instance, see~\cite{BEGZ} and \cite{Guenancia} and the references therein for a more general exposition).  In general, we will identify a Hermitian metric on a complex manifold with its associated $(1,1)$-form. Given a divisor $D$, we will often use the same notation for the $(1,1)$-cohomology class it defines and, by a slight abuse of notation, we will denote both the associated holomorphic line bundle and the rank $1$ locally free sheaf of holomorphic sections by ${\cal O}(D)$. 

Let $X$ be a complex projective manifold of dimension $N$ and $D=\sum_{i=1}^k \la_i D_i$ an $\R$-divisor, i.e. for each $i$, $D_i$ is an irreducible and reduced subvariety of codimension 1 and $\la_i\in\R^*$. Such a $D$ can be endowed with a ``metric'' which writes locally on a suitable covering $(V_j)$ of $X$ as $e^{-\phi_j}$, where  the real functions $\phi_j$  satisfy compatibility properties analogous to the classical case of a $\Z$-divisor, i.e. a line bundle in the usual sense.  The regularity of $\phi_j$ will be discussed later on, but let us say that they are in $L^1_{\rm loc}$.  In general, abusing notation, we write $h_D=e^{-\phi_D}$ for this metric. The ``curvature'' of $h_D$ is the globally defined closed $(1,1)$-current $\imath\Theta(h_D)=\frac{\imath}{2\pi}\partial\bar\partial\phi_D= dd^c\phi_D$ and is a representative of the cohomology class $\{D\}\in H^{1,1}(X,\R)$. Here $d=\partial+\bar\partial$ and $d^c=\frac 1{2\pi\imath}(\partial-\bar\partial)$ which are both real operators.

For instance,  if we have some  section of ${\cal O}_X(D_i)$  whose zero divisor is $D_i$ and which is given by a holomorphic function $f_i$ in local coordinates, then we can take $\phi_D=\sum_{i=1}^k \la_i\log |f_i|^{2}$. By the Lelong-Poincar\'e formula we have $\imath\Theta(h_D)=\sum_{i=1}^k \la_i [D_i]=[D]$ where $[D_i]$ is the current of integration over $D_i$. It will be more convenient to choose an arbitrary smooth metric $h_0$ on ${\cal O}(D)$ and to write the previous metric $h_D=e^{-\varphi_D}h_0$ for some function $\varphi_D:X\ra [-\infty,+\infty)$ which is smooth on $X\backslash D$. If $\Theta_0$ is the curvature of $h_0$ then $[D]=\Theta_0+dd^c\varphi_D$. In particular, we have $\Theta_0=-dd^c\varphi_D$ on $X\backslash D$.

From now on, we assume that $D$ is a $\R$-divisor with simple normal crossings and that
the pair $(X,D)$ is klt (for Kawamata log terminal), which will just mean for us that $\la_i<1$. In particular, $D$ is not necessarily effective.

 Let us fix a smooth volume form $dV$ on $X$ which is the same as a smooth metric on the anti-canonical line bundle $-K_X:=\Lambda^N T_X$. The opposite of the $(1,1)$-form associated to the curvature of this metric, that we will denote in a standard way by $\Theta_{K_X}:=dd^c\log(dV)$, is a representative of the first Chern class $c_1(K_X)$. Assume moreover that we have a smooth metric $\Omega$ on the restriction of the tangent bundle $T_X$ to $X\backslash D$ which satisfies
\begin{itemize}
\item[(i)] $\Ric(\Omega)=-c\,\Omega$ on $X\backslash D$, where $\Ric(\Omega)=-dd^c\log(\Omega^N)$ is the Ricci form of $\Omega$ and $c$ is a positive real number;
\item[(ii)] $\Omega^N=e^{\varphi-\varphi_D}dV$, where $\varphi_D$ is as above, and  $\varphi$ is a continuous function on $X$ and smooth on $X\backslash D$.
\end{itemize}
The first condition means that $\Omega$ is a K\"ahler-Einstein metric on $X\backslash D$ with negative Einstein constant $-c$ and the second condition imposes some control on the behavior of $\Omega^N$ at infinity i.e. near the boundary divisor $D$.
We will need the following simple
\begin{Lemma}\label{zeroextension:lemma}
Let $\Omega$ be a smooth closed positive $(1,1)$-form on $X\backslash D$ and assume that $\Omega$ has continuous local potentials on $X$, i.e. for any $x\in X$, there exists a neighborhood $U$ of $x$ in $X$ and a function $\varphi_U:U\ra\R$ which is continuous on $U$ and smooth on $U\backslash D$ such that $\Omega=dd^c\varphi_U$ on $X\backslash U$. Then the extension $\tilde\Omega$ by $0$ of $\Omega$ to $X$ is a  well defined closed positive current on $X$, and for any $x$ and $U$, $\tilde\Omega_{|U}=dd^c\varphi_U$ in the sense of currents.
\end{Lemma}
\begin{proof} By assumption, $\varphi_U$ is psh on $U\backslash D$ and by standard arguments (see~\cite{Dembook} for instance), it is known that ${\varphi_U}_{|U\backslash D}$ can be extended in a unique way as a psh function on the whole of $U$. In particular, the extension belongs to $L^1_{\rm Loc}(U)$. But  as $\varphi_U$ is continuous, this extension is actually $\varphi_U$. Moreover, still as $\varphi_U$ is continuous, its Lelong numbers along $D$ vanish hence $\tilde\Omega$ is well defined and coincides with $dd^c\varphi_U$ on $U$.
\end{proof}

 Set $\Theta:=\Theta_{K_X}+\Theta_0$. Remark that $\Theta$ is a smooth form on $X$. Since $\Theta_0=-dd^c\varphi_D$ on $X\backslash D$, (i) and (ii) imply that $c\,\Omega=\Theta+dd^c\varphi$ on $X\backslash D$ and in particular $\Theta+dd^c\varphi$ is a positive current on $X\backslash D$ (we also say that $\varphi$ is $\Theta$-psh). By the previous lemma (applied locally to $\varphi_U=\psi_U+\varphi_{|U}$ where $\psi_U$ is a local potential of the smooth form $\Theta$), the equality $c\,\Omega=\Theta+dd^c\varphi$ is valid on $X$ i.e. $c\,\Omega$ and $\Theta$ both are representatives of $c_1(K_X+D)$. Notice that since $dd^c\varphi$ puts no mass on $D$, we now obtain from (ii) that
 $$\Ric(\Omega)=-c\,\Omega+[D],$$
 namely $\Omega$ is a {\em singular K\"ahler-Einstein metric} attached to the pair $(X,D)$. As $\Ric(\Omega)$ is representative of $c_1(K_X)$, we have $c\{\Omega\}=c_1(K_X+D)$.

In general, if $T$ is a closed positive $(1,1)$-current on $X$, it is not possible to define $T^N$ in a reasonable way. However, if $T=\Theta+dd^c\varphi$ with $\Theta$ smooth and $\varphi$ locally bounded then, following the work of Bedford-Taylor~\cite{BT}, one can define $T^p$ for any $p\geq 1$ and moreover in our case where $T=\Omega$, $\{\Omega\}^N=\frac1{c^N}(K_X+D)^N$ (\cite{Demaillytrento}, Cor. 9.3). Finally, as the wedge product in the sense of Bedford-Taylor puts no mass on pluripolar sets (as a consequence of the Chern-Levine-Nirenberg inequality, see~\cite{Demaillytrento}, Prop. 1.11), we conclude that the volume of $X\backslash D$ endowed with the smooth metric $\Omega$ satisfies
\begin{equation}\label{eq:int:volume:KE}
\int_{X\backslash D}\Omega^N= \frac1{c^N}(K_X+D)^N.
\end{equation}

\subsection{Singular K\"ahler-Einstein metrics on $\ol{\Mod}_{0,n}$}
We shall apply now the formalism of the previous section to the situation where $X=\ol{\Mod}_{0,n}$,  and $X\backslash D=\Mod_{0,n}$ (here $N=n-3$). Recall that we defined $D_\mu:=\sum_\Scal \lambda_\Scal \,D_\Scal$ where
$$\lambda_\Scal=(|I_1|-1)(\mu_\Scal-1)+1$$
if $\Scal=\{I_0,I_1\}$  and, exchanging $I_0$ and $I_1$ if necessary, $\mu_\Scal:=\sum_{s\in I_1}\mu_s<1$ ($\mu_\Scal=\hat\mu_0$ in the notation of Section~\ref{sec:sect:L:near:bdry}). Observe that each $\lambda_\Scal$ is smaller than $1$. Here and in the sequel, the sums are always taken over all the (unordered) partitions $\Scal\in\Pcal$ i.e. satisfying $\min\{|I_0|,|I_1|\}\geq 2$.

\begin{Proposition}\label{prop:s-KE:metric:M0n} The  extension by $0$ of the Chern form $\Omega_\mu$ defined in Proposition~\ref{prop:Chern:form} is a singular K\"ahler-Einstein metric on $(\ol{\Mod}_{0,n},D_\mu)$. More precisely, $\Ric(\Omega_\mu)=-(N+1)\,\Omega_\mu+[D_\mu]$.
\end{Proposition}
\begin{proof}
We first recall a few basic facts about complex hyperbolic $N$-space:  it can be seen as the unit ball $\Bb^N\subset\C^N\subset\CP^N$ and we can identify its group of biholomorphisms with ${\rm PU}(1,N)$. We restrict to $\Bb^N$ the exact sequence of vector bundles
$$0\ra L\ra\underline{\C}^{N+1}\ra Q\ra 0,
$$
where $L$ is the tautological line subbundle of the trivial bundle $\underline{\C}^{N+1}=\CP^N\times\C^{N+1}$ and $Q$ is the quotient bundle. The group ${\rm U}(1,N)$ acts on this exact sequence and preserves the constant Hermitian metric of signature $(1,N)$ on $\underline{\C}^{N+1}$. In restriction to $L$, this metric is positive definite and hence defines a Hermitian metric on the line bundle $L$. The Chern form $c_1(L)$ associated with this metric is a positive $(1,1)$-form on $\Bb^N$, and the corresponding metric has constant holomorphic sectional curvature: it is K\"ahler-Einstein and $\Ric(c_1(L))=-(N+1)\,c_1(L)$. That the Einstein constant $-c$ is equal to $-(N+1)$ is due to the fact that on $\Bb^N$, the tangent bundle is naturally isomorphic to ${\rm Hom}(L,Q)$ hence the canonical bundle can be identified with $L^{N+1}$.

As $\Lcal$ is the pullback of $L$ by an immersion, this proves that $\Omega_\mu$ defines a metric on $\Mod_{0,n}$ and that  $\Ric(\Omega_\mu)=-(N+1)\,\Omega_\mu$ on $\Mod_{0,n}$. Moreover, by Proposition~\ref{prop:Chern:form} and Remark~\ref{rk:Omega:cont}, the metric on $\Lcal$ is locally defined by continuous positive functions (even on the neighborhood of $D_\mu$) and so by Lemma~\ref{zeroextension:lemma}, $\Omega_\mu$ defines a closed positive $(1,1)$-current on $\ol{\Mod}_{0,n}$.

Now, we are going to see that $\Omega_\mu^N=e^{\varphi-\varphi_{D_\mu}}dV$ where $dV$ is a smooth volume form on $\ol{\Mod}_{0,n}$, $\varphi_{D_\mu}$ is associated with the divisor $D_\mu$ and $\varphi$ is continuous on $\ol{\Mod}_{0,n}$. Proposition~\ref{prop:Chern:form} gives the expression of the metric $\Omega_\mu$ in local coordinates centered at a point $m$ of $\ol{\Mod}_{0,n}$. More precisely, recall that locally it is the pullback of the complex hyperbolic metric on $\Bb^N$
by the multivalued map
$$(z^0,t_1,\dots,t_r,z^1,\dots, z^r)\mapsto(z^0,P_1(t),\dots,P_r(t),P_1(t)z^1,\dots,P_r(t) z^r)$$
where $r$ is the number of vital divisors crossing at $m$, $z^j\in\C^{k_j-3}$ and $P_j(t)$ is described in Lemma~\ref{lm:1-form:glue:general}.

As the volume form associated with the complex hyperbolic metric on $\Bb^N$ is
$$\left(\frac{\imath}{2\pi}\right)^N\frac1 {(1-\|w\|^2)^{N+1}} dw_1\wedge d\bar w_1\wedge\dots\wedge dw_N\wedge d\bar w_N
$$
a straightforward computation shows that in the above coordinates
$$\begin{array}{rcl}
\Omega_\mu^N&=&\DS\left(\frac{\imath}{2\pi}\right)^N\frac{dz^0\wedge d\bar z^0\wedge\bigwedge_{j=1}^r \nu_j^2|t_j|^{-2}|P_j(t)|^{2}dt_j\wedge d\bar t_j\wedge\bigwedge_{j=1}^r |P_j(t)|^{2(k_j-3)} dz^j\wedge d\bar z^j}{\bigl(1-||z^{0}||^2-\sum_{j=1}^r|P_j(t)|^2(1+||z^{j}||^2)\bigr)^{N+1}}\\
&=&\DS\left(\frac{\imath}{2\pi}\right)^N\frac{\prod_{j=1}^r \left(\nu_j^2|t_j|^{-2}|P_j(t)|^{2(k_j-2)}\right)\,  \bigwedge_{j=0}^r dz^j\wedge d\bar z^j\wedge \bigwedge_{j=1}^r dt_j\wedge d\bar t_j}{\bigl(1-||z^{0}||^2-\sum_{j=1}^r|P_j(t)|^2(1+||z^{j}||^2)\bigr)^{N+1}}
\end{array}
$$
where for any $j=0,\dots,r$,  $dz^j\wedge d\bar z^j$ stands for $\bigwedge_{i=1}^{k_j-3}dz^j_i\wedge d\bar z^j_i$.

If $r=0$, i.e. if $m\in{\Mod}_{0,n}$, then $\Omega_\mu$ is smooth in a neighborhood of $m$. Assume now that $r\geq 1$. The divisors $D_{\Scal_j}$ passing through $m$ are given by $t_j=0$, $1\leq j\leq r$, and they correspond to partitions $\Scal_j=\{I_0^j,I_1^j\}$ (see Section~\ref{sec:sect:L:bdry:gen}). We have to determine the power of $|t_j|^2$ in the numerator of $\Omega_\mu^N$. From the combinatorial description in Section~\ref{sec:sect:L:bdry:gen}, we see that $|t_j|^{2\nu_j}$ appears $(|I_1^j|-1)$ times (which is the dimension of the stratum ${\Mod}_{0,|I_1^j|+1}$ plus $1$) in total in the product of the $|P_i|^{2(k_i-2)}$ so that the power of $|t_j|^2$ is $(|I_1^j|-1)\nu_j-1=-(|I_1^j|-1)(\mu_{\Scal_j}-1)-1=-\lambda_{\Scal_j}$ hence
$$\Omega_\mu^N=\left(\frac{\imath}{2\pi}\right)^N\frac{\prod_{j=1}^r \nu_j^2\, \bigwedge_{j=1}^r dt_j\wedge d\bar t_j\wedge \bigwedge_{j=0}^r dz^j\wedge d\bar z^j}{\prod_{j=1}^r |t_j|^{2\lambda_{\Scal_j}}\bigl(1-||z^{0}||^2-\sum_{j=1}^r|P_j(t)|^2(1+||z^{j}||^2)\bigr)^{N+1}}=e^{\varphi-\varphi_{D_\mu}}dV
$$
where, up to the multiplication by smooth functions, $\varphi=-\log\bigl(1-||z^{0}||^2-\sum_{j=1}^r|P_j(t)|^2(1+||z^{j}||^2)\bigr)^{N+1}$, $\varphi_{D_\mu}=\sum_{j=1}^r \lambda_{\Scal_j}\log |t_j|^2$ and $dV=\bigwedge_{j=1}^r dt_j\wedge d\bar t_j\wedge \bigwedge_{j=0}^r dz^j\wedge d\bar z^j$.
The proof of the proposition now follows from the discussion in Section~\ref{general:section} and Remark~\ref{rk:Omega:cont}.
\end{proof}

\subsection{Proof of Theorem~\ref{theorem:main}}\label{sec:proof:main}
From Proposition~\ref{prop:s-KE:metric:M0n}, we have 
$$
\{\Omega_\mu\}=\frac{1}{N+1}\left(K_{\ol{\Mod}_{0,n}}+\sum_{\Scal} \lambda_\Scal D_\Scal \right).
$$
We first notice that the canonical divisor $K_{\ol{\Mod}_{0,n}}$ can be expressed in terms of the vital divisors $D_\Scal$.  Indeed,
$$K_{\ol{\Mod}_{0,n}}\sim\psi-2\delta$$
where $\psi=\sum_{s=1}^{n} \psi_s$ is the $\psi$-divisor class (see~\cite[p.~335]{ACG11}), $\delta=\sum_{\Scal} D_\Scal$ is the boundary divisor, and $\sim$ stands for the linear equivalence of divisors (see~\cite[p.~386]{ACG11}; here we use that $\ol{\Mod}_{0,n}$ is a fine moduli space and that the Hodge bundle is trivial on $\ol{\Mod}_{0,n}$).

For pairwise distinct $i, j, k\in\{1,\dots,n\}$, denote by $\delta_{i|jk}$ the divisor in $\ol{\Mod}_{0,n}$ corresponding to curves with a node separating the $i$-th marked point from the $j$-th and $k$-th marked points. It is well-known (see~\cite{zvonkine} or \cite[Chap. 17]{ACG11}) that for any such choice of $i,j,k$ we have $\psi_i\sim\delta_{i|jk}$ hence
$$(n-1)(n-2)\psi\sim\sum_{i,j,k}\delta_{i|jk}=\sum_\Scal\big(|I_0|\,|I_1|(|I_1|-1)+|I_1|\,|I_0|(|I_0|-1)\bigr)D_\Scal=(n-2)\sum_\Scal |I_0|\,|I_1| \,D_\Scal
$$
and substituting $\psi$ in the above expression of $K_{\ol{\Mod}_{0,n}}$ we get
$$K_{\ol{\Mod}_{0,n}}\sim_\Q\sum_\Scal \frac{\bigl(|I_0|-2\bigr) \bigl(|I_1|-2\bigr)-2}{n-1} \,D_\Scal
$$
where $\sim_\Q$ stands for the $\Q$-linear equivalence of divisors. Thus, we obtain
$$K_{\ol{\Mod}_{0,n}}+\sum_\Scal\lambda_\Scal\,D_\Scal\sim_\Q\sum_\Scal \Bigl(|I_1|-1\Bigr) \biggl(\mu_\Scal-\frac{|I_1|}{N+2}\biggr) \,D_\Scal.
$$
Now, recall that $N=n-3$ and notice that
$$\begin{array}{rcl}
\displaystyle 2(n-1)\Bigl(|I_1|-1\Bigr) \biggl(\mu_\Scal-\frac{|I_1|}{N+2}\biggr) & = & 2\bigl(|I_1|-1\bigr)\bigl((n-1)\mu_\Scal-|I_1|\bigr)\\
&=&\bigl(n-2+|I_1|-|I_0|\bigr)(n-1)\mu_\Scal-2|I_1|\bigl(|I_1|-1\bigr)\\
&=&(n-2)(n-1)\mu_\Scal +\bigl(|I_1|-|I_0|\bigr)\bigl(|I_1|+|I_0|-1\bigr)\mu_\Scal-2 |I_1|\bigl(|I_1|-1\bigr)\\
&=&(n-2)(n-1)\mu_\Scal +\Bigl(|I_1|\bigl(|I_1|-1\bigr)-|I_0|\bigl(|I_0|-1\bigr) \Bigr)\mu_\Scal-2 |I_1|\bigl(|I_1|-1\bigr)\\
&=&(n-2)(n-1)\mu_\Scal - |I_0|\bigl(|I_0|-1\bigr)\mu_\Scal-(2-\mu_\Scal)|I_1|\bigl(|I_1|-1\bigr).
\end{array}
$$
By a similar computation as above we get\footnote{we are grateful to D.~Zvonkine for explaining this trick to us}
$$\begin{array}{rcl}
\displaystyle(n-1)(n-2)\sum_i\mu_i\,\psi_i&\sim&\displaystyle\sum_{i,j,k}\mu_i\,\delta_{i|jk}\\
&=&\displaystyle\sum_\Scal\Bigl( |I_1|\bigl(|I_1|-1\bigr)\sum_{i\in I_0}\mu_i+|I_0|\bigl(|I_0|-1\bigr)\sum_{i\in I_1}\mu_i\Bigr)D_\Scal\\
&=&\displaystyle\sum_\Scal\Bigl( |I_1|\bigl(|I_1|-1\bigr)(2-\mu_\Scal)+|I_0|\bigl(|I_0|-1\bigr)\mu_\Scal\Bigr)D_\Scal\\
\end{array}
$$
and therefore
$$\frac2{(N+1)} \biggl(K_{\ol{\Mod}_{0,n}}+\sum_\Scal\lambda_\Scal\,D_\Scal\biggr)\sim_\Q -\sum_s\mu_s\,\psi_s+\sum_\Scal \mu_\Scal\,D_\Scal.
$$
\begin{Remark}
If we define $\Pcal'$ to be the set of unordered partitions of $\{1,\dots,n\}$ into two non-empty subsets $I_0\sqcup I_1$ then with the convention $D_{\{s\},\{s\}^c}=-\psi_s$, we get the expression
$$
\{\Omega_\mu\}=\frac{1}{2}\sum_{\Scal\in \Pcal'}\mu_\Scal D_\Scal.
$$
\end{Remark}
Finally, using  \eqref{eq:int:volume:KE} we obtain the formula stated in Theorem~\ref{theorem:main} for the volume of $\Mod_{0,n}$ endowed with the metric $\Omega_\mu$:
$$
\int_{\Mod_{0,n}}\Omega_\mu^{N}=\frac{1}{(N+1)^N}\left(K_{\ol{\Mod}_{0,n}}+\sum_\Scal\lambda_\Scal\,D_\Scal\right)^N=\frac1{2^N}\left(\sum_{\Scal\in\Pcal'} \mu_\Scal\,D_{\Scal}\right)^N.
$$

\subsection{Comparison with McMullen's formula}\label{sec:proof:cor:McM}

In~\cite{McMullen}, C. McMullen proves a Gauss-Bonnet formula for cone manifolds. Using in particular the fact that on the unit $N$-ball there exists a ${\rm PU}(1,N)$-invariant metric with constant holomorphic sectional curvature, the formula enables him to calculate the volume of $\Mod_{0,n}$ endowed with the metric $\Omega_\mu$ by computing the orbifold Euler characteristic of $\ol{\Mod}_{0,n}$ (here, ``orbifold'' has to be taken in a very general sense).
From this point of view, our strategy is somewhat  similar to this approach,  since we can interpret $c_1\bigl(K_{\ol{\Mod}_{0,n}}+D_\mu\bigr)$ as an orbifold first Chern class.
If $X$ is a smooth $N$-ball quotient then by the Hirzebruch proportionality theorem, the total Chern class of $X$ is given by
$$
c(X)=\left(1+\frac{c_1(X)}{N+1}\right)^{N+1}=\left(1-\frac{c_1(K_X)}{N+1}\right)^{N+1}
$$
\noindent hence we have the following equality
$$c_1^N(K_X)=(-1)^N(N+1)^{N-1}\chi(X)$$
(where $\chi(X)=c_N(X)$ is the Euler characteristic of $X$). In general, the above equalities make sense at the level of the ${\rm PU}(1,N)$-invariant forms on $\Bb^N$ which represent the respective cohomology classes. As the metric $g_\mu$ of McMullen is normalized in order to have constant holomorphic sectional curvature $-1$, we have ${\rm Ric}(g_\mu)=-\frac 1{2\pi}\frac{N+1}2 g_\mu$ hence,  if $\omega_\mu$ is the K\"ahler form associated with $g_\mu$, $c_1(X)=-\frac {N+1}{4\pi}\omega_\mu$  and $c_N(X)=(-1)^N\frac{(N+1)}{(4\pi)^N}\omega_\mu^N$, as pullback of ${\rm PU}(1,N)$-invariant forms on $\Bb^N$. As a consequence, the volume we compute and the one computed by McMullen are related by
$$\int_{\Mod_{0,n}}\Omega_\mu^{N}=\frac{1}{(4\pi)^N}\int_{\Mod_{0,n}}\omega_\mu^N=\frac{N!}{(4\pi)^N}\,{\rm vol}(\Mod_{0,n},g_\mu)=\frac{(-1)^{N}}{N+1}\,\chi\bigl(P(\mu)\bigr)
$$
where $P(\mu)$ is the orbifold associated with the weights $\mu$ considered by McMullen.  In summary, we get the
\begin{Corollary}\label{coro:mcmullen}
$$
\left(-\sum_s\mu_s\,\psi_s+\sum_\Scal \mu_\Scal\,D_\Scal\right)^N=\frac{(-2)^{N}}{(N+1)}\sum_\Qcal (-1)^{|\Qcal|+1}(|\Qcal|-3)!\prod_{B\in\Qcal}\max\Bigl(0,1-\sum_{i\in B}\mu_i\Bigr)^{|B|-1}.
$$
where $\Qcal$ ranges over all partitions of the indices $\{1,\dots,n\}$ into blocks $B$.
\end{Corollary}

\section{A more algebro-geometric approach}\label{sec:mu:rat:ext:sect}

\subsection{Kawamata's extension}\label{sec:kawa:ext:def}
Through out this section, we will assume that all the weights $(\mu_s)$ are rational numbers.  If $d\in\N^*$ is such that $d\mu_s\in\N$ for all $s$, then the local system $\Lb^{\otimes d}$ on $\CP^1\smin \Sig$ is trivial. Y.~Kawamata proves in~\cite{kaw} that the line bundle $\Lcal^{\otimes d}$ has a natural extension to $\ol{\Mod}_{0,n}$ that we denote abusively by $\hat{\Lcal}^{\otimes d}$ (i.e. $\hat{\Lcal}$ is only a $\Q$-divisor).
This extension is constructed in the following way:  it follows immediately from the description in Section~\ref{sec:def:L}  that $\Lcal^{\otimes d}$ is isomorphic to $\pi_*\Ocal\bigl(d(K_{{\UC}_{0,n}/{\Mod}_{0,n}}+\sum_s \mu_s\,\Gamma_s)\bigr)$, where $\Gamma_s$ is the divisor given by the $s$-th section of the universal curve, and $K_{{\UC}_{0,n}/{\Mod}_{0,n}}:=K_{{\UC}_{0,n}}\otimes K^{\vee}_{{\Mod}_{0,n}}$ is the relative canonical bundle of the
fibration $\pi_{|{\UC}_{0,n}}:{\UC}_{0,n}\ra {\Mod}_{0,n}$. 
Observe that for any $m\in{\Mod}_{0,n}$, ${\rm deg}\Bigl({K_{{\UC}_{0,n}/{\Mod}_{0,n}}}_{|\pi^{-1}(m)}\Bigr)={\rm deg}\bigl(K_{\CP^1}\bigr)=-2$ and since $\sum_s\mu_s=2$, the restriction of ${\cal O}\bigl(d(K_{{\UC}_{0,n}/{\Mod}_{0,n}}+\sum_s \mu_s\,\Gamma_s)\bigr)$ to any fiber of $\pi_{|{\UC}_{0,n}}$ is trivial.
Therefore, $\pi_*\Ocal\bigl(d(K_{{\UC}_{0,n}/{\Mod}_{0,n}}+\sum_s \mu_s\,\Gamma_s)\bigr)$ is indeed a rank $1$ invertible sheaf on ${\Mod}_{0,n}$.

The first task is to extend $\Ocal\bigl(d(K_{{\UC}_{0,n}/{\Mod}_{0,n}}+\sum_s \mu_s\,\Gamma_s)\bigr)$ to a line bundle on $\ol{\UC}_{0,n}$ whose restriction to each fiber of $\pi$ over $\ol{\Mod}_{0,n}$ is still trivial. Kawamata remarks that a natural such extension is given by the divisor $d\Lambda$ where
$$
\Lambda:=K_{{\ol\UC}_{0,n}/\ol{\Mod}_{0,n}}+\sum_s \mu_s\,\Gamma_s-\sum_\Scal (1-\mu_\Scal) F^1_\Scal
$$
and the effective divisor $\sum_\Scal (1-\mu_\Scal) F^1_\Scal$ is defined in the following way: for any $\Scal\in\Pcal$, $\pi^{-1}(D_\Scal)$ is a divisor in ${\ol\UC}_{0,n}$ with two irreducible components $F^0_\Scal$ and $F^1_\Scal$. Over a generic point of  $D_\Scal$, these two components correspond respectively to the two irreducible components of the nodal curve  associated with the partition $\Scal=\{I_0,I_1\}$ (recall that by definition, $\mu_\Scal:=\sum_{s\in I_1}\mu_s<1$).

 It is easy to see that the restriction to each fiber $\pi^{-1}(m) \subset \ol{\UC}_{0,n}$ of the line bundle associated with the above divisor is indeed trivial for any $m\in\ol{\Mod}_{0,n}$. It is sufficient to check that its degree is $0$ in restriction to each irreducible component of any stable curve $C_m=C_m^0\cup\dots\cup C_m^r$. First remark that $\pi^{-1}(D_\Scal)=F^0_\Scal+F^1_\Scal$ is trivial in restriction to $C_m$ and so, for any $j$, ${F^1_\Scal}_{|C^j_m}=-{F^0_\Scal}_{|C^j_m}$. As a first consequence, if $F^0_\Scal\cap C^j_m=\vide$ or $F^1_\Scal\cap C^j_m=\vide$, then ${F^1_\Scal}_{|C^j_m}=0$. Moreover, noticing that $1-\mu_\Scal=1-\sum_{s\in I_1}\mu_s=\sum_{s\in I_0}\mu_s-1$ for any $\Scal$,  we have (using the notation of Section~\ref{sec:princ:component})
$$
d\Bigl(K_{{\ol\UC}_{0,n}/\ol{\Mod}_{0,n}}+\sum_s \mu_s\,\Gamma_s-\sum_\Scal (1-\mu_\Scal) F^1_\Scal\Bigr)_{|C^j_m}=d\Bigl(K_{\CP^1}+\sum_{i=1}^{s_j} y_{i}+\sum_{s\in\Sig_j}\mu_s x_s+\sum_{i=1}^{s_j} (\hat\mu(y_{i})-1) y_{i}\Bigr)
$$
whose degree is indeed equal to $0$. Finally, one defines $\hat{\Lcal}^{\otimes d}:=\pi_*\Ocal(d\Lambda)$.

\begin{Remark} In fact, we also have $\hat{\Lcal}^{\otimes d}=\pi_*{\cal O}\bigl(d\bigl(K_{{\ol\UC}_{0,n}/\ol{\Mod}_{0,n}}+\sum_s \mu_s\,\Gamma_s)\bigr)$, but the divisor $d\Lambda$ is more natural, even if less obvious at first glance; for instance, one has ${\cal O}(d\Lambda)=\pi^*\hat{\Lcal}^{\otimes d}$.
\end{Remark}


\subsection{Trivializations of Kawamata's extension}\label{sec:kawa}

In Sections~\ref{sec:sect:L:near:bdry} and~\ref{sec:sect:L:bdry:gen}, for each point $m\in\ol{\Mod}_{0,n}$ we found a neighborhood $\Ucal$ of $m$ in $\ol{\Mod}_{0,n}$ and we constructed a holomorphic section of $\Lcal$ on $\Ucal\cap {\Mod}_{0,n}$ that we denote by $\Phi_\Ucal$ or simply by $\Phi$.
We can regard $\Phi^{\otimes d}$ as a holomorphic section of ${\cal O}\bigl(d(K_{{\UC}_{0,n}/{\Mod}_{0,n}}+\sum_s \mu_s\,\Gamma_s)\bigr)$ on $\pi^{-1}(\Ucal\cap {\Mod}_{0,n})$.
From the description in Section~\ref{sec:sect:L:near:bdry}, we see immediately that as such, $\Phi^{\otimes d}$ extends as a section of ${\cal O}\bigl(d(K_{\ol{\UC}_{0,n}/\ol{\Mod}_{0,n}}+\sum_s \mu_s\,\Gamma_s)\bigr)$ on the whole of $\pi^{-1}(\Ucal)$ if $m$ is a generic point of $D_\Scal$.
Moreover, it vanishes exactly on $F^1_\Scal$ up to the order $d(1-\mu_\Scal)$, i.e. it is a non-vanishing holomorphic section of the extension ${\cal O}(d\Lambda)$ on $\pi^{-1}(\Ucal)$, hence providing a local trivialization of the line bundle ${\cal O}(d\Lambda)$  and so a trivialization of $\hat{\Lcal}^{\otimes d}$ on $\Ucal$.
In the same way, it can be proven  that $\Phi^{\otimes d}$ provides a local trivialization of $\hat{\Lcal}^{\otimes d}$ near any point $m$ of $\ol{\Mod}_{0,n}$ but we omit the proof since we only need to consider trivializations near generic points of $\partial \ol{\Mod}_{0,n}$.

It is also important to note that by Proposition~\ref{prop:Chern:form}, the Hermitian form $((.,.))$ defined in Section~\ref{sec:hermstruct} induces a continuous metric on the bundle $\hat{\Lcal}^{\otimes d}$, whose curvature on ${\Mod}_{0,n}$ is $d\Omega_\mu$.
Thus, the same arguments relying on Lemma~\ref{zeroextension:lemma} which enabled us to conclude that the extension by $0$ of $(N+1)\,\Omega_\mu$ is a representative of $c_1\bigl(K_{\ol{\Mod}_{0,n}}+D_\mu\bigr)$ prove that the extension by $0$ of $\Omega_\mu$ is a representative of $c_1(\hat\Lcal)=\frac1dc_1(\hat\Lcal^{\otimes d})$.
Summing up, we get the
\begin{Proposition}\label{prop:extend:L}
 Assume that $0<\mu_s< 1, \mu_s \in \Q$ for all $s \in \{1,\dots,n\}$. Let $d\in \N$ be a positive integer such that $d\mu_s \in \N$ for all $s$. Then the push-forward $\hat{\Lcal}^{\otimes d}$ of the Kawamata line bundle ${\cal O}(d\Lambda)$ is an extension of $\Lcal^{\otimes d}$  over $\ol{\Mod}_{0,n}$. If $m \in \ol{\Mod}_{0,n}$ is contained in a stratum of codimension $r$, with a neighborhood identified with $\Vcal\times(\Dis_{c^2})^r$, then the section $\Phi^{\otimes d}$ on $\Vcal\times(\Dis^*_{c^2})^r$ defined in Section~\ref{sec:sect:L:bdry:gen} extends naturally to a nowhere vanishing section of $\hat{\Lcal}^{\otimes d}$ in $\Vcal\times(\Dis_{c^2})^r$.
Moreover, the extension by $0$ of $\Omega_\mu$ is a representative of $c_1(\hat\Lcal)$.
\end{Proposition}

\begin{Remark}
 One can also prove that the restriction of $\hat{\Lcal}^{\otimes d}$ to the stratum of $m$ is the pull-back of the $d$-tensor power of the induced line bundle on the $\mu-$principal factor $\ol{\Mod}_{0,k_0}$ (see Section~\ref{sec:sect:L:bdry:gen} for the definition of $\mu$-principal component/factor).
\end{Remark}

Actually,  the above extension of $\Phi^{\otimes d}$ can be described in more  {\it concrete} terms. For this, let us give an alternative description of a plumbing family.
Let $m$ be a generic point of some divisor $D_\Scal$ with $\Scal=\{I_0,I_1\}$, and let $C_m^0=(\CP^1,(0,(x_s)_{s\in I_0})$ and $C_m^1=(\CP^1,(0,(y_s)_{s\in I_1})$ (here we denote the marked points on $C_m^1$ by $y_s$ rather than $x_s$) where we use  the conventions of Section~\ref{sec:sect:L:near:bdry}.
Consider the family $\ol\Ccal$ of rational curves above a disc $\Dis$ centered at $0$ which is described (in inhomogenous coordinates) by $\varpi:\ol\Ccal=\bigl\{(x,y,t)\in\CP^1\times\CP^1\times\Dis, \; xy=t\bigr\}\ra \Dis$, $(x,y,t)\mapsto t$ (note that the fibers are all smooth except the one above $0$ which is a nodal curve with $(0,0)$ as only node).
In this setting, $\Gamma_s=\{x_s\}\times \CP^1\times \Dis\cap \ol\Ccal$, for $s \in I_0$, and $\Gamma_s=\CP^1\times \{y_s\}\times \Dis\cap\ol\Ccal$, for $s \in I_1$.
Remark however that in general, this family is not isomorphic to the one described in the course of Section~\ref{sec:sect:L:near:bdry}, where we used additional changes of coordinates $F$ and $G$.
Let $p_i:\CP^1\times\CP^1\times\Dis \ra \CP^1$, $i=0,1$, be the natural projection onto the $(i+1)$-th factor $\CP^1$.
Define on $\CP^1$  two sections
$$
\omega_0:=\frac{(dz)^{\otimes d(2-\mu_\Scal)}}{\prod_{s\in I_0}(z-x_s)^{2d\mu_s}}\in\Gamma\Bigl(\CP^1,d\bigl((2-\mu_\Scal)K_{\CP^1}+2\sum_{s\in I_0} \mu_s x_s\bigr)\Bigr)
$$
and
$$
\omega_1:=\frac{(dz)^{\otimes d\mu_\Scal}}{\prod_{s\in I_1}(z-y_s)^{2d\mu_s}}\in\Gamma\Bigl(\CP^1,d\bigl(\mu_\Scal K_{\CP^1}+2\sum_{s\in I_1} \mu_s y_s\bigr)\Bigr).
$$
Then $\omega=p_0^*\omega_0\otimes p_1^*\omega_1$ induces a section of ${\cal O}\bigl(2d(K_{\ol\Ccal_m/\Dis}+\sum_s\mu_s\Gamma_s)\bigr)$ on $\ol\Ccal$.
Near the point $(0,0,0)\in\ol\Ccal$, in the coordinates $(x,y)$, a trivialization of $K_{\ol\Ccal/\Dis}$ is provided by
$\kappa=\frac{1}{2}\left(\frac{dx}{x}-\frac{dy}{y}\right)=\frac{dx}{x}=-\frac{dy}{y}$. Since we have
$$
\omega=\frac{(dx)^{\otimes d(2-\mu_\Scal)}\otimes(dy)^{\otimes d\mu_\Scal}}{\prod_{s\in I_0}(x-x_s)^{2d\mu_s}\prod_{s\in I_1}(y-y_s)^{2d\mu_s}},
$$
the section induced by $\omega$ (in restriction to $\ol\Ccal$) is given by
$$
(-1)^{d\mu_\Scal}\frac{x^{d(2-\mu_\Scal)}y^{d\mu_\Scal}}{\prod_{s\in I_0}(x-x_s)^{2d\mu_s}\prod_{s\in I_1}(y-y_s)^{2d\mu_s}} \kappa^{\otimes 2d}.
$$
If we factorize by $(-xy)^{d\mu_\Scal}=(-t)^{d\mu_\Scal}$ and take the ``square root'', we find
$$
\tau_m:=\frac{x^{d(1-\mu_\Scal)}}{\prod_{s\in I_0}(x-x_s)^{d\mu_s}\prod_{s\in I_1}(y-y_s)^{d\mu_s}} \kappa^{\otimes d}.
$$
Since $\tau_m$ does not vanish outside of the nodal curve $C_m=\varpi^{-1}(0)$, and vanishes to order $d(1-\mu_\Scal)$ on the component $C^1_m$,  we conclude that as a section of ${\cal O}(d\Lambda)$ on $\ol\Ccal_m$, $\tau_m$ is equal to $\Phi^{\otimes d}$ up to the multiplication by an invertible function on $\Dis$. It will be more convenient below to use the coordinates $(x,t)$ (even if those are only coordinates away from $x=0$) in which
$$\tau_m=\frac{\prod_{s\in I_1}(-x_s)^{d\mu_s}}{\prod_{s\in I_0}(x-x_s)^{d\mu_s}\prod_{s\in I_1}(x-tx_s)^{d\mu_s}}(dx)^{\otimes d}
$$
and where we used the notation $x_s:=1/y_s$ if $s\in I_1$.

\subsection{Other formulas for the volume and proof of Theorem~\ref{thm:kawa:formpart}}\label{sec:other formula}
As a direct consequence of the discussion in Section~\ref{sec:kawa:ext:def}, we get
\begin{Theorem}\label{thm:kawa:form}
Under the assumptions of Proposition~\ref{prop:extend:L}, let $\sigma$ be a global section of $\hat{\Lcal}^{\otimes d}$ over
$\ol{\Mod}_{0,n}$ and let us define $D_\sigma:=\frac1d{\rm div}(\sigma)$ where ${\rm div}(\sigma)$ is the divisor of $\sigma$. Then
$$
\int_{\Mod_{0,n}}\Omega_\mu^{N}= \frac1{(N+1)^N}(K_{\ol{\Mod}_{0,n}}+D_\mu)^{N}= D_\sigma^N.
$$
\end{Theorem}
%

In what follows, we will construct explicitly a holomorphic section $\sigma$ of $\hat{\Lcal}^{\otimes d}$ and determine the corresponding divisor $D_\sigma$. Our goal is to prove Theorem~\ref{thm:kawa:formpart}.

In~\cite{kaw}, Y. Kawamata constructs global sections of $\hat{\Lcal}^{\otimes d}$ over $\ol{\Mod}_{0,n}$. For our purpose, we present here a slight variation of those sections: define $J=(j_1,j'_1,\dots,j_d,j'_d)\in\N^{2d}$ by
$$j_i=s\ \ {\rm if}\ \ d\sum_{k=1}^{s-1}\mu_k< i\leq d\sum_{k=1}^s\mu_k,
$$
$$j'_i=s\ \ {\rm if}\ \ d\sum_{k=1}^{s-1}\mu_k< d+i\leq d\sum_{k=1}^s\mu_k
$$
for all $1\leq i\leq d$, where by convention $\sum_{k=1}^{0}\mu_k=0$. The only two important points in the definition of $J$ are that (1) each $s$ appears $d\mu_s$ times in $J$, and (2) for any $i$, $j_i\not=j'_i$.

For each $1\leq s<s'\leq n$, we define
$$\lambda(s,s')=\frac 1 d\#\bigl\{i,\;\, (j_i,j'_i)=(s,s')\bigr\}$$
that is the number of times the pair $(s,s')$ occurs as $(j_i,j'_i)$ divided by $d$. Alternatively, with our choice of $J$,
\begin{equation}\label{eq:lambda:def}
\lambda(s,s')=\left\{
\begin{array}{l}
0\ \ {\rm if}\ \sum_{k=s}^{s'}\mu_k\leq1\ {\rm or}\ \sum_{k=s+1}^{s'-1}\mu_k\geq1,\\
\min\bigl\{\mu_s,\mu_{s'},\sum_{k=s}^{s'}\mu_k-1,1-\sum_{k=s+1}^{s'-1}\mu_k\bigr\}\ \ {\rm otherwise}.
\end{array}
\right.
\end{equation}
Let $\{x_1,\dots,x_n\}$ be $n$ distinct points on $\CP^1$. For any pair $(j,j')$ of distinct elements of $\{1,\dots,n\}$, we denote by $\omega_{j,j'}$ the unique non-vanishing rational $1$-form on $\CP^1$ with simple poles at $x_j$ and $x_{j'}$, and satisfying ${\rm res}_{x_j}=1$, ${\rm res}_{x_{j'}}=-1$. If the points $x_j, x_{j'}$ are in $\CP^1\backslash\{\infty\}$ then
$$
\omega_{j,j'}=(x_j-x_{j'})\frac{dz}{(z-x_j)(z-x_{j'})}.
$$
Finally, let us define
$$\omega_J:=\prod_{i=1}^d\omega_{j_i,j'_i}\in\Gamma\Bigl(\CP^1,d\bigl(K_{\CP^1}+\sum_{s=1}^n\mu_s x_s\bigr)\Bigr).
$$
Remark that $\omega_J$ is invariant by the action of ${\rm PGL}(2,\C)$, thus it gives rise to  a well-defined non-vanishing section of $\Lcal^{\otimes d}$ on $\Mod_{0,n}$.
This section extends to the whole $\ol{\Mod}_{0,n}$ as a section $\sigma$ of $\hat{\Lcal}^{\otimes d}$.
We are now going to determine the support of its zero divisor $D_\sigma$, which must be contained in the boundary divisor of $\ol{\Mod}_{0,n}$, by using the above trivializations $\tau_m$.
Let us fix $\Scal=\{I_0,I_1\}$ as before. In the notation of Section~\ref{sec:kawa}, and using the coordinates $(x,t)$ for the universal family above a small disc $\Dis$ transverse to $D_\Scal$ at a generic point $m$, the section $\omega_J$ writes
$$\omega_J=\frac{\prod_{j_i,j'_i\in I_0}(x_{j_i}-x_{j'_i})\prod_{j_i\in I_0,j'_i\in I_1}(x_{j_i}-tx_{j'_i})\prod_{j_i\in I_1,j'_i\in I_0}(tx_{j_i}-x_{j'_i})\prod_{j_i,j'_i\in I_1}t(x_{j_i}-x_{j'_i})}{\prod_{s\in I_0}(x-x_s)^{d\mu_s}\prod_{s\in I_1}(x-tx_s)^{d\mu_s}}(dx)^{\otimes d}
$$
i.e.
$$
\omega_J=t^{\#\{i,\;\, j_i,j'_i\in I_1\}}f(t)\tau_m = t^{d\sum_{1\leq s < s'\leq n}\delta_\Scal(s,s')\la(s,s')}f(t)\tau_m
$$
where $f$ is an invertible function on $\Dis$ and $\delta_\Scal(s,s')=\left\{
\begin{array}{l}
1\ {\rm if}\ \{s,s'\}\subset I_1\\ 
0\ {\rm otherwise}
\end{array}
\right.
$.
As a consequence, we get
\begin{equation}\label{eq:c1L:div:sig}
\hat\Lcal\sim_\Q D_\sigma:=\frac1 d {\rm div}(\sigma)=\sum_\Scal\sum_{1\leq s<s'\leq n} \delta_\Scal(s,s')\lambda(s,s')\, D_\Scal.
\end{equation}
Notice that $\sigma$ and $D_\sig$ depend on the multi-indices $J$. By choosing other multi-indices $J$ satisfying conditions $(1)$ and $(2)$ above, we would obtain other divisors to which $\hat\Lcal$ is $\Q$-linearly equivalent.

\subsection*{Proof of Theorem~\ref{thm:kawa:formpart}}
\begin{proof}
Applying Theorem~\ref{thm:kawa:form} with $D_\sigma$ given by \eqref{eq:c1L:div:sig},  we see that Theorem~\ref{thm:kawa:formpart} is proved if the weights in $\mu$ are all rational. If not, one can approximate them by rational numbers in such a way that the numbers $\delta_\Scal(s,s')$ remain unchanged. From \eqref{eq:lambda:def} we see that $\lambda(s,s')$ depends continuously on $\mu$. Thus $D_\sig$ depends continuously on $\mu$.  On the other hand, from Theorem~\ref{theorem:main}, we know that the total volume of $\Mod_{0,n}$ with respect to $\Omega_\mu$ depends also continuously on $\mu$.  Thus by continuity with respect to $\mu$, Theorem~\ref{thm:kawa:formpart} is shown in full generality, that is for all $\mu$ satisfying the hypothesis of Theorem~\ref{theorem:main}.
\end{proof}

\subsection{Another look at Theorem~\ref{theorem:main}} As a final remark, we would like to show now that the 
point of view adopted in this section also provides an alternative way to find the expression of $\{\Omega_\mu\}$ obtained in the proof of the main theorem.

Here, as before, we have to assume the weights $\mu_s$ to be rational, multiply them by a positive integer $d$ in such a way that the numbers $d\mu_s$ are integers, and consider $\hat\Lcal^{\otimes d}$. The general case of real weights then follows again by continuity arguments. However, as the reader can easily check, the computations can be made directly as if $\hat\Lcal$ was actually a line bundle.

In Section~\ref{sec:other formula} we exhibited sections of $\hat\Lcal$ whose zero divisor provides representatives of $c_1(\hat\Lcal)$ which is equal to $\{\Omega_\mu\}$. As $\hat\Lcal$ is the pushforward of a line bundle on the universal curve, it is also natural to use the Grothendieck-Riemann-Roch formula to compute $c_1(\hat\Lcal)$.

Again, we refer to~\cite{ACG11} or~\cite{zvonkine} for the basic material. Let us define $\Delta$ as the codimension 2 subvariety of ${\ol\UC}_{0,n}$ consisting of the nodes of the singular fibers of the projection $\pi: \ol{\UC}_{0,n} \ra \ol{\Mod}_{0,n}$ and $K=K_{{\ol\UC}_{0,n}/\ol{\Mod}_{0,n}}(\sum_s\Gamma_s)$. The Todd class of $\pi$ is given by
$${\rm td}(\pi)=1-\frac12\Bigl(K-\sum_s\Gamma_s\Bigr)+\frac1{12}\Bigl(K^2+\sum_s\Gamma_s^2+\Delta\Bigr)+\dots
$$
and recall that $\hat\Lcal=\pi_*\Ocal(\Lambda)$ where
$$
\Lambda=K_{{\ol\UC}_{0,n}/\ol{\Mod}_{0,n}}+\sum_s \mu_s\,\Gamma_s-\sum_\Scal (1-\mu_\Scal) F^1_\Scal.
$$
Notice that $R^1\pi_*\Ocal(\Lambda)=0$ as $\Ocal(\Lambda)$ is trivial along the fibers of $\pi$, hence by the Grothendieck-Riemann-Roch formula, a representative of $c_1(\hat\Lcal)$ is
\begin{eqnarray}
\frac12\pi_*\left\{\Biggl(K+\sum_s(\mu_s-1)\Gamma_s+\sum_\Scal(\mu_\Scal-1)F^1_\Scal\Biggr)\Biggl(\sum_s\mu_s\Gamma_s+\sum_\Scal(\mu_\Scal-1)F^1_\Scal\Biggr)\right\}\nonumber
\end{eqnarray}
because $\frac1{12}\pi_*\Bigl(K^2+\sum_s\Gamma_s^2+\Delta\Bigr)$ represents the first Chern class of the Hodge bundle, which is trivial on $\ol{\Mod}_{0,n}$. Now, it is well known that for any $s$ and any $s'\not=s$,
$$K\cdot\Gamma_s=0\ ,\ \ \Gamma_s\cdot\Gamma_{s'}=0\ ,\ \ \pi_*\bigl(\Gamma_s^2\bigr)=-\psi_s$$
and straightforward computations show that
$$
$$
$$
\pi_*(K\cdot F^1_\Scal)=(|I_1|-1)\,D_\Scal\ ,\ \ \pi_*(F^1_\Scal\cdot F^1_{\Scal'})=\left\{\begin{array}{cl}
0 & {\rm\ if\ \Scal\not=\Scal'}\\
-D_\Scal & {\rm\ if\ \Scal=\Scal'}
\end{array}\right.
\ \ {\rm and}\ \ 
\pi_*(F^1_\Scal\cdot\Gamma_s)=\left\{\begin{array}{cl}
0 & {\rm\ if\ }s\not\in I_1\\
D_\Scal & {\rm\ if\ }s\in I_1
\end{array}\right.
$$
for any $s$, $\Scal$ and $\Scal'$. Therefore,
$$c_1(\hat\Lcal)=\frac12\left(-\sum_s\mu_s(\mu_s-1)\psi_s+\sum_\Scal\mu_\Scal(\mu_\Scal-1)D_\Scal\right).
$$
Finally, by a slight variation of the computation in Section~\ref{sec:proof:main} we obtain
$$\begin{array}{rcl}
\displaystyle(n-2)\sum_{i}\mu_i(2-\mu_i)\,\psi_i&=&\displaystyle\sum_{i,j,k}\mu_i(\mu_j+\mu_k)\,\psi_i\\
&\sim&\displaystyle\sum_{i,j,k}\mu_i(\mu_j+\mu_k)\,\delta_{i|jk}\\
&=&\displaystyle\sum_\Scal\Bigl(\bigl(|I_1|-1\bigr)\sum_{j\in I_1}\mu_j\sum_{i\in I_0}\mu_i+\bigl(|I_0|-1\bigr)\sum_{j\in I_0}\mu_j\sum_{i\in I_1}\mu_i\Bigr)D_\Scal\\
&=&(n-2)\displaystyle\sum_\Scal\mu_\Scal(2-\mu_\Scal)D_\Scal\\
\end{array}
$$
which implies
$$c_1(\hat\Lcal)=\frac12\left(-\sum_s\mu_s\psi_s+\sum_\Scal\mu_\Scal D_\Scal\right)$$
as expected.
\appendix
\section{Intersection theory on $\ol{\Mod}_{0,n}$}\label{appendix:intersection}
In this section we describe an algorithm to compute the intersection numbers of vital divisors in $\ol{\Mod}_{0,n}$. This algorithm is well known to experts in the field and can be found in~\cite{Kontsevich-Manin}. We include it here only for the sake of completeness. We are grateful to D.~Zvonkine for having explained it to us.

Intersections of vital divisors in $\ol{\Mod}_{0,n}$ will produce formal sums of trees whose vertices are labelled by subsets in a partition of $\{1,\dots,n\}$. At every vertex, the sum of the cardinal of the corresponding subset and the number of edges containing it must be at least $3$.   Such a tree corresponds to a stratum of $\ol{\Mod}_{0,n}$.  Note that we allow $\varnothing$ to be part of a partition. A vital divisor $D_\Scal$, where $\Scal=\{I_0,I_1\}$ is a partition of $\{1,\dots,n\}$ such that $\min\{|I_0|,|I_1|\}\ge 2$, corresponds to a tree with two vertices labelled by $I_0$ and $I_1$.

Here below we will give the rule to compute the intersection of a divisor $D_{\Scal}$ with a tree $\T$  as above.  Recursively, this allows us to compute any product $D_{\Scal_1}\cdot\dots\cdot D_{\Scal_{n-3}}$.  We first color the vertices of $\T$ with respect to the partition $\{I_0,I_1\}$ as follows: the vertices labelled by subsets contained in $I_0$ are given the red color, those labelled by subsets contained in $I_1$ are given the blue color. The vertices corresponding to subsets which are not contained in $I_0$ nor in $I_1$ are given the black color. Finally, the vertices corresponding to the empty set are given the white color. We have three cases:
\begin{itemize}
\item[$\bullet$] {\bf Case 1}: there is more than one black vertex. In this case the intersection is empty, we get $0$.

\item[$\bullet$] {\bf Case 2:} there is exactly  one black vertex. If there is an edge in $\T$ which connects a red vertex and a blue one then we get $0$. Otherwise the black vertex separates the red vertices from the blue ones. We subdivide the subset corresponding to the black vertex  into two subsets: one is contained in $I_0$, the other in $I_1$. We then replace this vertex of $\T$ by an edge whose ends are labelled by the two subsets above. We color the new vertices using the same rule. There is a unique configuration such that the new edge separates the red vertices from the blue ones. The intersection is then given by this new tree.

\item[$\bullet$] {\bf Case 3:} there are no black vertices. We will  say that a vertex or an edge  of $\T$ separates the red vertices from the blues ones if it is contained in any path  joining a red vertex to a blue one. We have several subcases:

\begin{itemize}
\item[(a)] There are no edges and no vertices  that separate the red vertices from the blue ones. In this case the intersection is $0$.

\item[(b)] There are no edges that separate the red vertices from the blue ones, but there is a vertex $A$ that satisfies this property. Note that $A$ is then unique.
We first notice that all the leaves of $\T$ must be either red or blue. Thus we can subdivide the set of edges incident to $A$ into two subsets: $E'$ is the set of edges that are contained in some paths joining $A$ to a red leaf, $E''$ is the set of edges that are contained in some paths joining $A$ to a blue leaf. That $\{E',E''\}$ is a partition of the set of edges incident to $A$ is a consequence of the hypothesis that $A$ separates the red vertices from the blue ones.

We form a new tree by splitting $A$ into two vertices $A',A''$ connected by an edge, where $A'$ is attached to all the edges in $E'$, and $A''$ is attached to all the edges in $E''$. We associate to $A'$ the subset $A\cap I_0$, and to $A''$ the subset $A\cap I_1$. In more concrete terms, if $A$ is red then $A'=A, A''=\varnothing$, if $A$ is blue then $A'=\varnothing, A''=A$, if $A=\varnothing$ then $A'=A''=\varnothing$. This new tree is the intersection of $D_\Scal$ and $\T$. Notice that it is necessary stable because otherwise, there would exist an edge separating the red vertices from the blue ones.

\item[(c)] There is an edge $e$ that separates the red vertices from the blue ones. In this case this edge must be unique. Let $A$ and $B$  denote the ends of $e$. By a slight abuse of notation we will also denote by $A$ and $B$ the corresponding subsets of $\{1,\dots,n\}$. Note  that $A$ and $B$ can be empty.

Let $\hat{A}$ be the union of the indices contained in $A$ and the edges incident to $A$. We pick a pair $\{a_1,a_2\}$ in $\hat{A}$ such that $e\not\in \{a_1,a_2\}$. Consider all the partitions of $\hat{A}$ into two subsets $\{\hat{A}_1,\hat{A}_2\}$ such that $e\in \hat{A}_1$, $\{a_1,a_2\} \subset \hat{A}_2$, and $\min\{|\hat{A}_1|,|\hat{A}_2|\}\geq 2$.  For any such partition, we remove the vertex $A$ from $\T$ and construct a new  tree from $\T$ as follows: form two new vertices $A_1$ and $A_2$, attach $A_i$ to all the edges in $\hat{A}_i$ and add a new edge connecting $A_1$ and $A_2$. The new vertex $A_i$ is associated with the set of indices in $\{1,\dots,n\}\cap\hat{A}_i$. Let $\Sig_A$ denote the formal sum of all the trees obtained this way.

We apply the same  to $B$, and let $\Sig_B$ denote the corresponding formal sum. The intersection of $D_\Scal$ with $\T$ is then equal to $-(\Sig_A+\Sig_B)$.
\end{itemize}
\end{itemize}
The intersection number $D_{\Scal_1}\cdot\dots\cdot D_{\Scal_{n-3}}$ is then the sum of all the coefficients of the trees in the final formal sum obtained from this algorithm.

\medskip

Using this algorithm, we can compute the intersection numbers of vital divisors in $\ol{\Mod}_{0,5}$ and $\ol{\Mod}_{0,6}$. As $\Scal=\{I_0,I_1\}$ is of course determined by either $I_0$ or $I_1$, we denote below $D_\Scal$ by $D_{I_0}$ or $D_{I_1}$.

\subsection*{Case $\ol{\Mod}_{0,5}$}
\[
D_{ij}\cdot D_{ij}=-1, \quad D_{ij}\cdot D_{jk}=0, \quad D_{ij}\cdot D_{k\ell}=1.
\]
\subsection*{Case $\ol{\Mod}_{0,6}$}
Recall that $D_{I}\cdot D_{J}=0$ if  neither of $J$ and $J^c$ is  contained in $I$ or in $I^c$.  The intersections which do not vanish due to this simple rule are recorded here below.
\[\begin{array}{l}
D_{ij}\cdot D_{ij}\cdot D_{ij}=1, \quad D_{ij}\cdot D_{ij}\cdot D_{ijk}=0, \quad D_{ij}\cdot D_{ij}\cdot D_{k\ell}=-1,\\
D_{ij}\cdot D_{ijk}\cdot D_{ijk}=-1, \quad D_{ij}\cdot D_{ijk} \cdot D_{j'k'}=1,\\
D_{ijk}\cdot D_{ijk} \cdot D_{ijk}=2, \\
D_{ij}\cdot D_{k\ell} \cdot  D_{k'\ell'}=1.\\
\end{array}
\]

\section{Computation of the volume in $\Mcal_{0,5}$}

Here we compute the volume of $\Mcal_{0,5}$ with respect to $\Omega_\mu$ using the results of Section~\ref{sec:mu:rat:ext:sect}.
We may assume that $1> \mu_1\geq \mu_2\geq \mu_3\geq \mu_4\geq \mu_5>0$. Note that in any case, $\mu_2+\mu_4\leq 1$ since $\sum\mu_s=2$. As a consequence, only the following can happen:
\begin{itemize}
\item[$\bullet$] $\mu_2+\mu_3\leq1$ and
\begin{itemize}
\item[.] $\mu_1+\mu_5\geq1$:
$$D_\sigma=(1-\mu_1)D_{13}+(1-\mu_1)D_{14}+(1-\mu_1) D_{25}\ \ {\rm and}$$
$$\int_{\Mod_{0,5}}\Omega_\mu^{2}=(1-\mu_1)^2
$$
\item[.] $\mu_1+\mu_4\geq1$, $\mu_1+\mu_5\leq1$:
$$D_\sigma=(1-\mu_1)D_{13}+\mu_5D_{14}+(1-\mu_1-\mu_5)D_{24}+\mu_5 D_{25}\ \ {\rm and}$$
$$\int_{\Mod_{0,5}}\Omega_\mu^{2}=(1-\mu_1)^2-(1-\mu_1-\mu_5)^2
$$
\item[.] $\mu_1+\mu_3\geq1$, $\mu_1+\mu_4\leq1$:
$$D_\sigma=(1-\mu_1)D_{13}+(1-\mu_2-\mu_3)D_{14}+(1-\mu_1-\mu_5)D_{24}+\mu_5 D_{25}\ \ {\rm and}$$
$$\int_{\Mod_{0,5}}\Omega_\mu^{2}=(1-\mu_1)^2-(1-\mu_1-\mu_4)^2-(1-\mu_1-\mu_5)^2
$$
\item[.] $\mu_1+\mu_2\geq1$, $\mu_1+\mu_3\leq1$:
$$D_\sigma=\mu_3D_{13}+(1-\mu_2-\mu_3)D_{14}+(1-\mu_1-\mu_5)D_{24}+\mu_5 D_{25}\ \ {\rm and}$$
$$\int_{\Mod_{0,5}}\Omega_\mu^{2}=2\,\mu_3\,\mu_5-(1-\mu_1-\mu_4)^2-(1-\mu_2-\mu_4)^2
$$
\item[.] $\mu_1+\mu_2\leq1$:
$$D_\sigma=(1-\mu_4-\mu_5)D_{13}+(1-\mu_2-\mu_3)D_{14}+(1-\mu_1-\mu_5)D_{24}+(1-\mu_3-\mu_4)D_{25}+(1-\mu_1-\mu_2)D_{35}$$
$${\rm  and}\ \ \int_{\Mod_{0,5}}\Omega_\mu^{2}=2\sum_{i=1}^5(1-\mu_{i-1}-\mu_i)(1-\mu_i-\mu_{i+1})-\sum_{i=1}^5(1-\mu_i-\mu_{i+1})^2
$$
\end{itemize}
\item[$\bullet$] $\mu_2+\mu_3\geq1$ and $\mu_1+\mu_4\leq1$:
$$D_\sigma=(\mu_4+\mu_5)D_{13}+\mu_4D_{24}+\mu_5 D_{25}\ \ {\rm and}$$
$$\int_{\Mod_{0,5}}\Omega_\mu^{2}=2\,\mu_4\,\mu_5.
$$
\end{itemize}
All the formulas are obtained as a straightforward application of Theorem~\ref{thm:kawa:form}. However, one can prove after some more (tedious) computations that if $\mu_2+\mu_3\leq 1$ and $\mu_1+\mu_{s-1}\geq 1$, $\mu_1+\mu_s\leq 1$ for some $2\leq s\leq 6$ (which happens for all but the last exceptional case) then
$$\int_{\Mod_{0,5}}\Omega_\mu^{2}=(1-\mu_1)^2-\sum_{i=s}^5(1-\mu_1-\mu_i)^2.
$$

\section{An example in $\Mod_{0,6}$}
The fact that $\Omega_\mu$ is a representative of the first Chern class of the Kawamata line bundle $\hat{\Lcal}$ can be exploited to simplify the evaluation of $\int_{\Mod_{0,n}}\Omega_\mu^{n-3}$ in certain cases, especially when the weight vector $\mu$ has some symmetry. To illustrate this observation, let us consider the family of weights $\mu=(\alpha,\alpha,\alpha,\beta,\beta,\beta)$, with $0<\beta\leq\alpha$ and $\alpha+\beta=2/3$. Assume that $\alpha$ and $\beta$ are both rational, we can find $d\in \N^*$ such that $d\alpha\in 2\N$ and $d\beta\in 2\N$. Define a section $\sigma$ of the Kawamata line bundle by
$$
\sigma =\frac{(x_1-x_2)^{d\frac{\alpha}{2}}(x_2-x_3)^{d\frac{\alpha}{2}}(x_3-x_1)^{d\frac{\alpha}{2}}(x_4-x_5)^{d\frac{\beta}{2}} (x_5-x_6)^{d\frac{\beta}{2}}(x_6-x_4)^{d\frac{\beta}{2}}}{(z-x_1)^{d\alpha}(z-x_2)^{d\alpha}(z-x_3)^{d\alpha}(z-x_4)^{d\beta}(z-x_5)^{d\beta}(z-x_6)^{d\beta}}dz^{\otimes d}
$$
We will use the following equality (which is a consequence of Theorem~\ref{thm:kawa:form}) $\int_{\Mod_{0,6}}\Omega_\mu^3=\left(\frac{\mathrm{div}(\sigma)}{d}\right)^3$ to compute the volume of $\Mod_{0,6}$ with respect to $\Omega_\mu$.

In what follows, for any subset $I\subset \{1,\dots,6\}$ such that $2\leq |I| \leq 4$, $D_I$ is the boundary divisor of $\ol{\Mod}_{0,6}$ corresponding to the partition $\{I,I^c\}$. In particular, any boundary divisor of $\ol{\Mod}_{0,6}$ can be written as $D_I$ with $|I| \leq 3$. Set
$$
A_1 =  D_{123}, \quad  A_2 = \sum_{1\leq i \leq 3} \sum_{4\leq j< k \leq 6} D_{ijk}, \quad B = \sum_{1\leq i <j \leq 3} D_{ij},\quad C = \sum_{4\leq i < j \leq 6} D_{ij}.
$$
Applying the algorithm described in Appendix~\ref{appendix:intersection}, we get the following
$$\left\{
\begin{array}{l}
A_1^3=2,  \quad A_2^3=18, \quad B^3=3, \quad C^3= 3,\\
A_1A_2=0, \\
A_1^2B=A_1^2C=-3, \quad A_1B^2=A_1C^2=0,\\
A_2^2B=A_2^2C=-9, \quad A_2B^2=A_2C^2=0,\\
B^2C=BC^2=-9,\\
A_1BC=A_2BC=9.
\end{array}
\right.
$$
We have two cases
\begin{itemize}
 \item[$\bullet$] Case I: $0<\beta\leq \frac{1}{6} \Leftrightarrow \alpha \geq \frac{1}{2}$.  We have $  \frac{\mathrm{div}(\sigma)}{d}=\frac{3\beta}{2}A_1 +\frac{\beta}{2}A_2+\frac{3\beta}{2}B+\frac{\beta}{2}C$.  Therefore
 $$
 \left(\frac{\mathrm{div}(\sigma)}{d}\right)^3=(3A_1+A_2+3B+C)^3\left(\frac{\beta}{2}\right)^3=48\times\frac{\beta^3}{8}=6\beta^3.
 $$

 \item[$\bullet$] Case II: $\frac{1}{6} \leq \beta \leq \frac{1}{3} \Leftrightarrow \frac{1}{3} \leq \alpha \leq \frac{1}{2}$.  We have $\frac{\mathrm{div}(\sigma)}{d}=\frac{3\beta}{2}A_1 +\frac{\beta}{2}A_2+\frac{\alpha}{2}B+\frac{\beta}{2}C $.  It follows
 $$
 \left(\frac{\mathrm{div}(\sigma)}{d}\right)^3 =  \frac{3}{8}((\alpha-3\beta)^3+16\beta^3) =  6\beta^3- 3(2\beta-\frac{1}{3})^3.
 $$
 \end{itemize}
 To sum up, we have
 \begin{equation*}
  \int_{\Mod_{0,6}} \Omega_\mu^3= 6\beta^3-3(\max\{2\beta-\frac{1}{3}, 0\})^3, \text{ for all } \beta \in (0,\frac{1}{3}].
 \end{equation*}

\end{document}